\documentclass[11pt,reqno]{amsart}

\usepackage{amsmath,amssymb,amsfonts,amsthm,enumerate,natbib,color,ifthen,hyperref}
\usepackage{xargs}
\usepackage[textwidth=4cm, textsize=footnotesize]{todonotes}
\usepackage[latin1]{inputenc}

\usepackage{float}

\usepackage[small,bf]{caption}

\usepackage{graphicx}

\usepackage{aliascnt,bbm}
\def\rset{\mathbb R}
\def\zset{\mathbb Z}

\def\eqsp{\;}

 \newcommand{\pscal}[2]{\left\langle#1,#2\right\rangle}

\newcommand{\eqdef}{\ensuremath{\stackrel{\mathrm{def}}{=}}}

\def\F{\mathcal{F}} % filtration
\def\M{\mathcal{M}}

\def\D{\mathcal{D}}

\newcommandx\sequence[3][2=t,3=\zset]
{\ifthenelse{\equal{#3}{}}{\ensuremath{\{ #1_{#2}\}}}{\ensuremath{\{ #1_{#2}, \eqsp #2 \in #3 \}}}}

\def\PP{\mathbb{P}} % proba
\newcommand{\CPP}[3][]
{\ifthenelse{\equal{#1}{}}{{\mathbb P}\left(\left. #2 \, \right| #3 \right)}{{\mathbb P}_{#1}\left(\left. #2 \, \right | #3 \right)}}
\def\PE{\mathbb{E}} % esperance
\newcommand{\CPE}[3][]
{\ifthenelse{\equal{#1}{}}{{\mathbb E}\left[\left. #2 \, \right| #3 \right]}{{\mathbb E}_{#1}\left[\left. #2 \, \right | #3 \right]}}

\newcommand{\normfro}[1]{\left\Vert#1\right\Vert_{\textsf{F}}}

\def\Prox{\operatorname{Prox}}

\def \thetaen{\hat\theta}

\usepackage{bm}%bold matH

\theoremstyle{plain}

\newtheorem{theorem}{Theorem}

\newaliascnt{proposition}{theorem}

\aliascntresetthe{proposition}

\newaliascnt{lemma}{theorem}
\newtheorem{lemma}[lemma]{Lemma}
\aliascntresetthe{lemma}
\newaliascnt{corollary}{theorem}
\newtheorem{corollary}[corollary]{Corollary}
\aliascntresetthe{corollary}

\theoremstyle{definition}
\newaliascnt{definition}{theorem}

\aliascntresetthe{definition}

\newtheorem{algorithm}{Algorithm}

\newaliascnt{remark}{theorem}
\newtheorem{remark}[remark]{Remark}
\aliascntresetthe{remark}

\newaliascnt{example}{theorem}

\aliascntresetthe{example}

\def\rmd{\mathrm{d}}

\def\1{\mathbbm{1}}

\newcommand{\sbt}{\mbox{subject to}}
\DeclareMathOperator*{\mini}{minimize}
\DeclareMathOperator*{\maxi}{maximize}
\DeclareMathOperator*{\argmin}{{\textsf{Argmin}}}

\usepackage{rotating}
\renewcommand{\sbt}{\mbox{subject to}}
\newcommand{\diag}{\mathrm{diag}}
\usepackage{verbatim,float}

\usepackage[hmargin=3.3cm,vmargin=3.5cm]{geometry}

\newcommand{\Glasso}{\textsc{Glasso} }
 \newcommand{\mm}{\textbf}
 \newcommand{\quic}{\textsc{Quic~}}
   \newcommand{\matlab}{\textsc{Matlab}}

 \newcommand{\Glassoperiod}{\textsc{Glasso.}}

  \newcommand{\tr}{\text{Tr}}
  \newcommand{\epsilonmod}{\ell}
  \newcommand{\deltamod}{\varepsilon}
  \newcommand{\kappamod}{\psi}
\DeclareMathOperator{\sign}{sign}

 \newenvironment{myarray}[2][1]
  {\def\arraystretch{#1}\array{#2}}
  {\endarray}

 \usepackage[noend]{algorithmic} 
\begin{document}

%\title[Proximal gradient algorithms for the elastic net]{Proximal gradient algorithms for the elastic net}
%\title[Regularized Covariance Computation via Stochastic Optimization]{Regularized Inverse Covariance computation via Stochastic Optimization}

%\title[Covariance Matrix Computation via Stochastic Optimization]{Computing Regularized Massive Inverse Covariance Matrices via Stochastic Optimization}
%\title[Covariance Matrix Computation via Stochastic Optimization]{Regularized Massive Inverse Covariance Matrix Computation via Stochastic Optimization}
\title[Precision Matrix Computation via Stochastic Optimization]{Scalable Algorithms for Regularized Precision Matrices via Stochastic Optimization}

\author{Yves F. Atchad\'e}  \thanks{ Yves F. Atchad\'e: University of Michigan, 1085 South University, Ann Arbor, 48109, MI, USA. {\em E-mail address:} yvesa@umich.edu}
\author{Rahul Mazumder}  \thanks{Rahul Mazumder: MIT Sloan School of Management and Operations Research Center, Cambridge, MA, USA. {\em E-mail address:} rahulmaz@mit.edu}
\author{Jie Chen}  \thanks{Jie Chen: IBM Thomas J. Watson Research Center, 1101 Kitchawan Road, Yorktown Heights, 10598 NY, USA. {\em E-mail address:} chenjie@us.ibm.com}

\subjclass[2000]{60F15, 60G42}

\keywords{Inverse covariance estimation, stochastic optimization, graphical lasso, proximal gradient}
\date{May, 2015}
\maketitle

\begin{abstract}
We consider the problem of computing a positive definite $p \times p$
 inverse covariance matrix aka precision matrix $\theta=(\theta_{ij})$ which optimizes a
regularized Gaussian maximum likelihood problem, with the elastic-net
regularizer
$\sum_{i,j=1}^{p} \lambda (\alpha|\theta_{ij}| + \frac{1}{2}(1- \alpha)
\theta_{ij}^2),$ with regularization parameters $\alpha \in [0,1]$ and $\lambda>0$. 
The
associated convex semidefinite optimization problem is notoriously difficult to scale to
large problems and has demanded significant attention over the past
several years.
 We propose a new algorithmic framework based on 
 stochastic proximal optimization (on the primal problem) that can be used to obtain near optimal
solutions with substantial computational savings over deterministic algorithms. 
A key challenge of our work stems from the fact that the optimization problem being investigated
does not satisfy the usual assumptions required by stochastic gradient methods. 
Our proposal has (a) computational guarantees and (b) scales well to large problems, even if the solution is not too sparse; thereby, enhancing 
the scope of regularized maximum likelihood problems to many large-scale problems of contemporary
interest. 
 An important aspect of our proposal is to
bypass the \emph{deterministic} computation of a matrix inverse 
by drawing random samples from a suitable multivariate Gaussian distribution.
%We demonstrate the usefulness of our proposal on a range of
%synthetic and real data examples, for up to $p = 15,000$.
\end{abstract}

%\bigskip
%
%\setcounter{secnumdepth}{3}

\section{Introduction}\label{sec:intro}

We consider the problem of estimating an inverse covariance matrix aka precision matrix~\citep{Laur1996} $\theta$, from a data matrix ${X}_{n \times p}$ comprised  
of $n$ samples from a $p$ dimensional  multivariate Gaussian distribution with mean zero and covariance matrix $\Sigma = \theta^{-1}$, i.e., $\mathbf{x}_{i} \stackrel{\text{i.i.d.}}{\sim} \textbf{N}(0 , \Sigma)$ for  $i = 1,\ldots,n$.
If $n < p$ it is a well known fact that the Maximum Likelihood Estimate (MLE) does not 
exist, and even if it does exist ($n \geq p$) the MLE
 can be poorly behaved and regularization is often called for. 
Various forms of regularization are used to improve the statistical behavior of covariance matrix estimates~\citep{pourahmadi2013high,buhlmann2011statistics,FHT-09-new} and is a 
topic of significant interest in the statistics and machine learning communities. This paper deals with the problem of computing such regularized matrices, in the settings where $p$ is much larger than $n$ or both $p$ and $n$ are large.
To motivate the reader, we briefly review two popular forms of precision matrix regularization schemes under a likelihood framework: sparse precision matrix estimation via 
$\ell_{1}$-norm regularization, 
 and its dense counterpart, using an $\ell_{2}$-norm regularization (ridge penalty); both on the entries of the matrix $\theta$.

\subsubsection*{Sparse precision matrix estimation --- the Graphical Lasso}\label{glasso-intro-1}
One of the most popular regularization approaches and the main motivation behind this paper is the Graphical Lasso~\citep{yuan_lin_07,BGA2008,FHT2007a} procedure aka \Glassoperiod~
Here, we estimate $\theta$ under the assumption that it is sparse, 
with a few number of non-zeros. Under the multivariate Gaussian modeling set up, 
$\theta_{ij} = 0$ (for $i \neq j$) is equivalent to the conditional independence of $x_{i}$ and $x_{j}$ given the remaining variables, where, 
$\mathbf{x}=(x_{1}, \ldots, x_{p})\sim \textbf{N}(0, \Sigma)$. 
\Glasso~minimizes the negative log-likelihood subject to a
penalty on the $\ell_{1}$ norm of the entries of the precision matrix $\theta$. This leads to the following convex optimization problem~\citep{BV2004}: 
\begin{equation}\label{eq-glasso-1}
\mini_{\theta \in\M_+}\;\;\;\; \underbrace{-\log\det\theta + \textsf{Tr}(\theta S)}_{:=f(\theta)}  + \lambda \sum_{i, j} | \theta_{ij} |, 
\end{equation}
where, $S=\frac{1}{n} \sum_{i=1}^{n}\mm{x}_{i}\mm{x}'_{i}$ is the sample covariance matrix, $\M_+$ denotes the set of positive definite matrices and $\lambda > 0$ is a tuning parameter that controls the 
degree of regularization\footnote{As long as $\lambda>0,$ the minimum of Problem~\eqref{eq-glasso-1} is finite (see Lemma~\ref{keylem}) and there is a unique minimizer.
In some variants of Problem~\eqref{eq-glasso-1}, the diagonal entries of $\theta$ are not penalized---such an estimator 
can infact be written as a version of Problem~\eqref{eq-glasso-1}, with $ S \leftarrow S - \lambda \mathbb{I}_{p \times p}$ where, $\mathbb{I}$ is a $p \times p$ identity matrix. The minimum of this problem 
need not be finite. In this paper, however, we will consider formulation~\eqref{eq-glasso-1} where the diagonals are penalized.}. In passing, we note that the \Glasso criterion, though motivated as a regularized 
negative log-likelihood problem, can be used more generally for any positive semidefinite (PSD) matrix $S$.

In modern statistical applications we frequently encounter examples where Problem~\eqref{eq-glasso-1} needs to be solved for $p$ of the order of several thousands. 
Thus there is an urgent need to develop fast and scalable algorithms for Problem~\eqref{eq-glasso-1}.
  In this vein, the past several years have witnessed a flurry of interesting work in developing fast and efficient solvers for Problem~\eqref{eq-glasso-1}. 
  We present a very brief overview of the main approaches used for the \Glasso problem, with further additional details presented in the Appendix, Section~\ref{related-work}.
  A representative list of  popular algorithmic approaches include (a) block (where, each row/column is a block) coordinate methods~\citep{BGA2008,FHT2007,mazumder2012graphical};
(b) proximal gradient descent type methods~\citep{BGA2008,Lu:09,rolfs2012iterative}; (c) methods based on Alternating Direction Method of Multipliers~\citep{scheinberg2010sparse,boyd-admm1,yuan2012alternating}; (d) specialized interior point methods~\citep{Li:10}; and
(e) proximal Newton type methods~\citep{JMLR:v15:hsieh14a,oztoprak2012newton}. 
 All the aforementioned methods are deterministic in nature. Precise (global) computational guarantees are available for some of them. It appears that most of the aforementioned computational approaches for 
Problem~\eqref{eq-glasso-1}, have a (worst-case) cost of at 
least $O(p^3)$ or possibly larger---this is perhaps not surprising, since for $\lambda=0$, finding the MLE requires 
computing $S^{-1}$ (assuming that the inverse exists), with cost $O(p^3)$.
Many of the state-of-the art algorithms for \Glasso~\citep{JMLR:v15:hsieh14a,FHT2007} (for example) make clever use of the fact that solutions to Problem~\eqref{eq-glasso-1} are sparse, for large values of $\lambda$.
Another important structural property of \textsc{Glasso}, that enables the scalable computation of Problem~\eqref{eq-glasso-1} is the exact thresholding property~\citep{mazumder2012exact,witten2011new}. The method is particularly useful for large values of $\lambda$, whenever the solution to the \Glasso problem decomposes into smaller connected components; and becomes less effective when the solution to the \Glasso problem is not sufficiently sparse. 
In short, computing solutions to Problem~\eqref{eq-glasso-1} become increasingly difficult as soon as $p$ exceeds a few thousand.

All existing algorithms proposed for \textsc{Glasso}, to the best of our knowledge, are deterministic batch algorithms. To improve the computational scalability of Problem~\eqref{eq-glasso-1},
we consider a different approach in this paper. Our approach, uses for the first time, ideas from stochastic convex optimization for the \Glasso problem. 
%With the motivation of developing scalable computational approaches for Problem~\eqref{eq-glasso-1} with rigorous computational guarantees; we propose, to our knowledge for the first time, %novel stochastic convex optimization methods for \Glassoperiod
%\medskip 
%
%\noindent {\emph{Exact thresholding into connected components:}} The exact covariance thresholding \citep{mazumder2012exact,witten2011new} plays a crucial role in the scalability of the 
%\Glasso to large problems for large values of $\lambda$. \cite{mazumder2012exact} showed that the connected components of 
%$\mathbf{1}( |s_{ij} | > \lambda),$ (where $\mathbf{1}(\cdot)$ is the indicator function) are exactly equal to the connected components of the \emph{concentration graph}\footnote{Given an inverse covariance matrix aka a concentration matrix $\theta$,
%the concentration graph is given by the matrix of edges, i.e., $(1(\theta_{ij} \neq 0))$ representing the non-zero pattern of $\theta$} estimated via \Glassoperiod~For large values of $\lambda$, the method can be used to split a large \Glasso problem into smaller tractable problems, making it possible to solve an otherwise computationally infeasible large \Glasso problem. 
%This method can be used as a wrapper around any existing \Glasso solver. However the method is not a panacea. The thresholding property depends upon the data matrix $S$ and the choice of $\lambda$; and in some cases the size of the largest connected component may still be quite large, thereby calling for scalable methods for the \Glasso problem. 
%
\subsubsection*{From sparse to dense regularization}
We consider another traditionally important regularization scheme, given via the following optimization problem:
\begin{equation}\label{eq-ridge}
\mini_{\theta \in\M_+}\;\;\;\; -\log\det\theta + \textsf{Tr}(\theta S)  + \frac{\lambda}{2} \sum_{i, j} \theta^2_{ij},
\end{equation}
for some value of $\lambda>0$. This can be thought of as the ridge regularized version\footnote{Note that some authors~\citep{warton2008penalized} refer to a different problem as a ridge regression problem, namely one where one penalizes the trace of $\theta$ instead of the frobenius norm of $\theta$. 
Such regularizers are often used in the context of regularized
discriminant analysis~\citep{Fr89_mod,hastie1995penalized}. However, in 
this paper we will denote Problem~\eqref{eq-ridge} as the ridge regularized version of the Gaussian maximum likelihood problem.} of Problem~\eqref{eq-glasso-1}. 
We will see in Section~\ref{sec:ridge-regu} that Problem~\eqref{eq-ridge} admits an analytic solution which requires computing the eigen-decomposition of $S$, albeit difficult when both $n$ and $p$ are large. Note that many of the tricks employed by modern solvers for \textsc{Glasso}, anticipating a sparse solution, no longer apply here.
The stochastic convex optimization framework that we develop in this paper also applies to Problem~\eqref{eq-ridge}, thereby enabling the computation of near-optimal solutions 
for problem-sizes where the exact solution becomes impractical to compute.

In this paper, we study a general version of Problems~\eqref{eq-glasso-1} and~\eqref{eq-ridge} by taking a convex combination of the ridge and $\ell_{1}$ penalties:
\begin{equation}\label{eq-glasso-enet-1}
\mini_{\theta \in\M_+} \;\;\; \underbrace{-\log\det\theta + \textsf{Tr}(\theta S)}_{:=f(\theta)}  +  \underbrace{\sum_{i,j} \left(\alpha \lambda_{} | \theta_{ij} | + \frac{(1-\alpha)}{2}\lambda_{} \theta_{ij}^2 \right)}_{:= g_{\alpha}(\theta)},
\end{equation}
with $\alpha \in [0,1]$. Following~\citet{ZH2005}, we dub the above problem as the \emph{elastic net} regularized version of the 
negative log-likelihood. Notice that for $\alpha=1$ we get \Glasso and $\alpha=0$ corresponds to Problem~\eqref{eq-ridge}.
We propose a novel, scalable framework for computing near-optimal solutions to Problem~\eqref{eq-glasso-enet-1} via techniques in 
stochastic convex optimization.

\subsection{Organization of the paper} The remainder of the paper is organized as follows. 
Section~\ref{sec:contri} provides an outline of the methodology and our contributions in this paper.
 We study deterministic proximal gradient algorithms in Section \ref{sec:prox:grad}. 
We present  the stochastic algorithms, proposed herein---Algorithm~\ref{algo1} and Algorithm~\ref{algo2} in Section \ref{sec:sto:prox:grad}. 
We describe the exact thresholding rule for Problem~\eqref{eq-glasso-enet-1} in Section \ref{sec:thres}. The application of the stochastic algorithm (Algorithm~\ref{algo1}) to the ridge regularized problem (Problem~\ref{eq-ridge}) is presented in Section~\ref{sec:ridge-regu}.
 We present some numerical results that illustrate our theory in Section \ref{sec:numerics}. 
 The proofs are collected in Section \ref{sec:proofs}, and some additional material are presented in the appendix.

%Our motivation stems from the fact that 
%With the motivation of developing scalable computational approaches for Problem~\eqref{eq-glasso-1} with rigorous computational guarantees; we propose, to our knowledge for the first time, %novel stochastic convex optimization methods for \Glassoperiod

\section{Outline of the paper and our contributions}\label{sec:contri}

\subsubsection*{Deterministic Algorithms} The starting point of our analysis, is the study  
of a (deterministic) proximal gradient (\cite{nest-07-new,fista-09, becker2011templates, parikh:boyd:2013}) algorithm (Algorithm~\ref{algo0}) for solving Problem~\eqref{eq-glasso-enet-1}.
%Related iterative thresholding algorithms were studied in~\cite{rolfs2012iterative} for the special case $\alpha=1$.
A direct application of the proximal gradient algorithm~(\cite{nest-07-new,fista-09}, for example) to Problem~\eqref{eq-glasso-enet-1} has some issues. 
Firstly, the basic assumption of Lipschitz continuity of the gradient $\nabla f(\theta)$, demanded by the proximal gradient algorithm, is not satisfied here. Secondly, 
the proximal operator associated with Problem~\eqref{eq-glasso-enet-1} is difficult to compute, as it involves minimizing an $\ell_{1}$ regularized quadratic function  over the cone $\M_+$.  
%We address these issues by showing that with a sufficiently small step-size the proximal operator can be computed by ``dropping'' the constraint set $\M_+$. 
We show that these hurdles may be overcome by controlling the step-size. Loosely speaking, we also establish that $\nabla f(\theta)$ satisfies a Lipschitz condition
(and $f(\theta)$ satisfies a strong convexity condition) across the iterations of the algorithm---a notion that we make precise in Section~\ref{sec:prox:grad}. Using these key aspects of our algorithm,  
we derive a global linear convergence rate of Algorithm~1, even though the objective function is not strongly convex on the whole feasible set $\M_+$. Furthermore, the algorithm has an appealing convergence behavior that we highlight:  its convergence rate is dictated by the condition number\footnote{defined as the ratio of the largest eigenvalue over the smallest eigenvalue} of $\thetaen$,
a solution to Problem~\eqref{eq-glasso-enet-1}. For a given accuracy $\deltamod>0$, our analysis implies that Algorithm~\ref{algo0} has a computational cost complexity of $O\left(p^3\textsf{cond}(\thetaen)^2\log(\deltamod^{-1})\right)$ to reach a $\deltamod$-accurate solution, where
$\textsf{cond}(\thetaen)$ is the condition number  of $\thetaen$. 
 The computational bottleneck of the algorithm is the evaluation of  the gradient of the smooth component at every iteration, which in this problem  is  
$\nabla f(\theta) = -\theta^{-1} + S$. Computing the gradient requires performing a matrix inversion, an operation that scales with $p$ as $O(p^3)$---we refer the reader to Figure~\ref{fig-inv-chol-evd-1} for an idea about the scalability behavior of direct dense matrix inversion for a $p \times p$ matrix, for different sizes of $p$.

Proximal gradient descent methods on the primal of the \Glasso problem has been studied by~\cite{rolfs2012iterative}. 
Our approaches however, have some differences---our analysis hinges heavily on basic tools and techniques made available by the general theory of 
proximal methods; and we analyze a generalized version: Problem~\eqref{eq-glasso-enet-1}. The main motivation behind our analysis of Algorithm~1 is that it 
lays the foundation for the stochastic algorithms, our primary object of study in this paper.

\subsubsection*{Stochastic Algorithms} For large values of $p$ (larger than a few thousand), Algorithm~1 slows down considerably, due to repeated computation of the inverse: $\theta^{-1}$ (See also Figure~\ref{fig-inv-chol-evd-1})
across the proximal gradient iterations.
%\textcolor{red}{RM: Jie, we need a few words over here saying that even if $\theta$ is sparse, the inversion can be slow}.
Even if the matrix $\theta$ is sparse and sparse numerical linear algebra methods are used for computing $\theta^{-1}$, the computational cost depends quite heavily upon the sparsity pattern of $\theta$ and the re-ordering algorithm used to reduce fill-ins; and need not be robust\footnote{In fact, in our experiments we observed that \textsc{Matlab} performs dense Cholesky decomposition more efficiently than sparse Cholesky decomposition, even when the matrix is sparse. This is due in part to multithreading: the dense Cholesky decomposition is automatically multithreaded in \textsc{Matlab}, but the sparse Cholesky decomposition is not. Another reason is the difficulty of finding a good re-ordering algorithm to limit fill-ins when performing sparse Cholesky decomposition.}
 across different problem instances.
 Thus, our key strategy in the paper is to develop a stochastic method that completely bypasses the exact computation (via direct matrix inversion) of the gradient $\nabla f(\theta) = S-\theta^{-1}$.
We propose to draw $N_k$ samples $z_1,\ldots,z_{N_k}$ (at iteration $k$) from $\textbf{N}(0, \theta_{k-1}^{-1})$ to form a noisy estimate $S-N_k^{-1}\sum_{k=1}^{N_k} z_iz_i'$ of the gradient $S-\theta_{k-1}^{-1}$.
This scheme forms the main workhorse of our stochastic proximal gradient algorithm, which we call Algorithm~\ref{algo1}. 
%Figure \ref{fig-fixed-sample-size} uses Stochastic-A to illustrates that 
% approximating the gradient by sampling is more appealing from a computational viewpoint over computing a direct matrix inverse. 
 
\begin{figure}[h!]
\centering
\scalebox{0.92}[.8]{\begin{tabular}{c c c}
%\rotatebox{90}{\sf {\small { \hspace{1cm} $(\phi_{\alpha}(\theta_k) - \hat{\phi}_{\alpha})/|\hat{\phi}_{\alpha}|$ }}}&\includegraphics[width=0.5\textwidth,height=0.3\textheight,  trim = 1cm 2.5cm 0cm 1.5cm, clip = true ]{\plots/times_chol_eig_inv1.pdf} &
& & (Zoomed) \\
\rotatebox{90}{ \hspace{1.5cm} \sf {\small Time (in seconds) }} & \includegraphics[width=0.5\textwidth,  height=0.3\textheight, trim = .6cm 1.2cm 0cm 1.5cm, clip = true ]{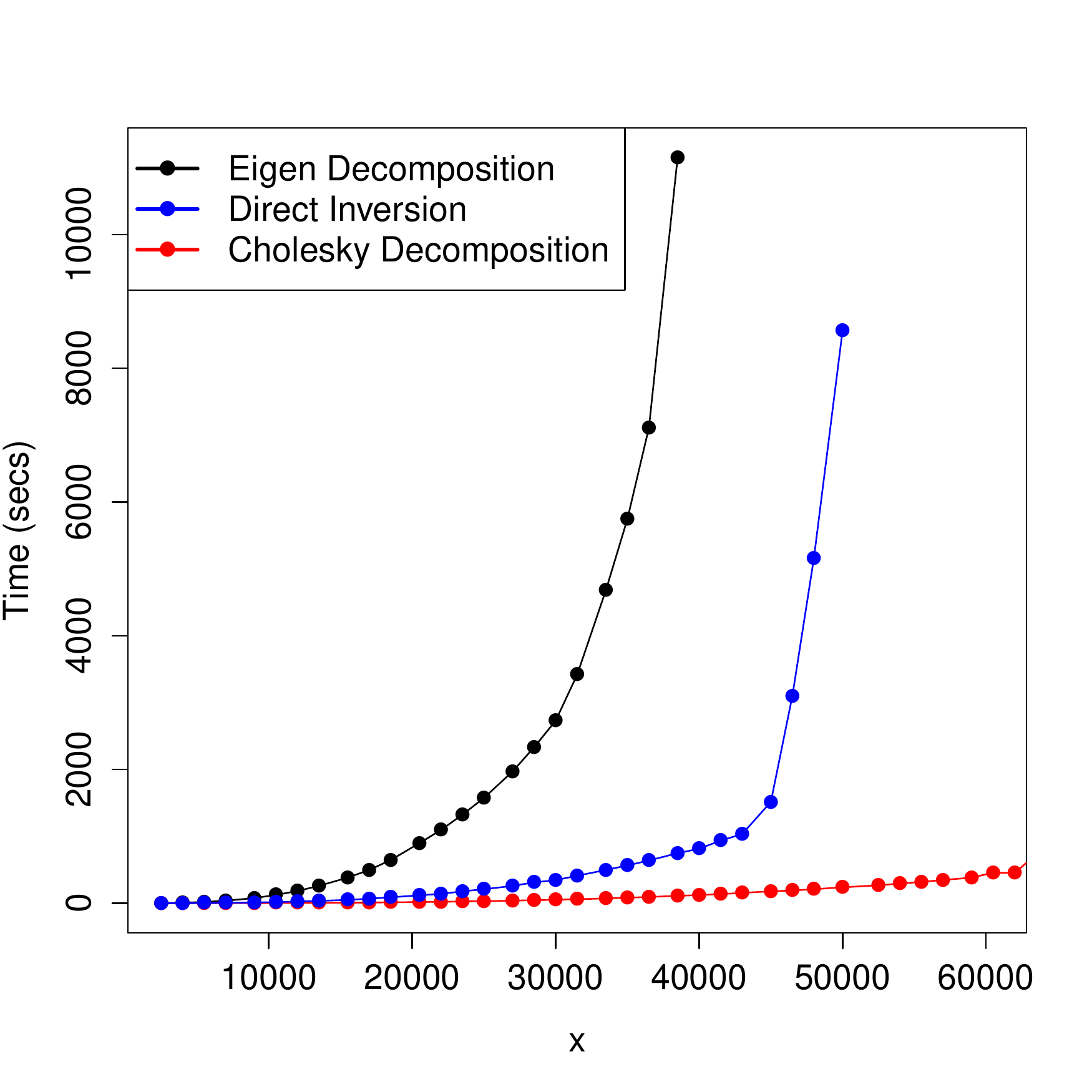} &
\includegraphics[width=0.5\textwidth, height=0.3\textheight, trim = .6cm 1.2cm 0cm 1.5cm, clip = true ]{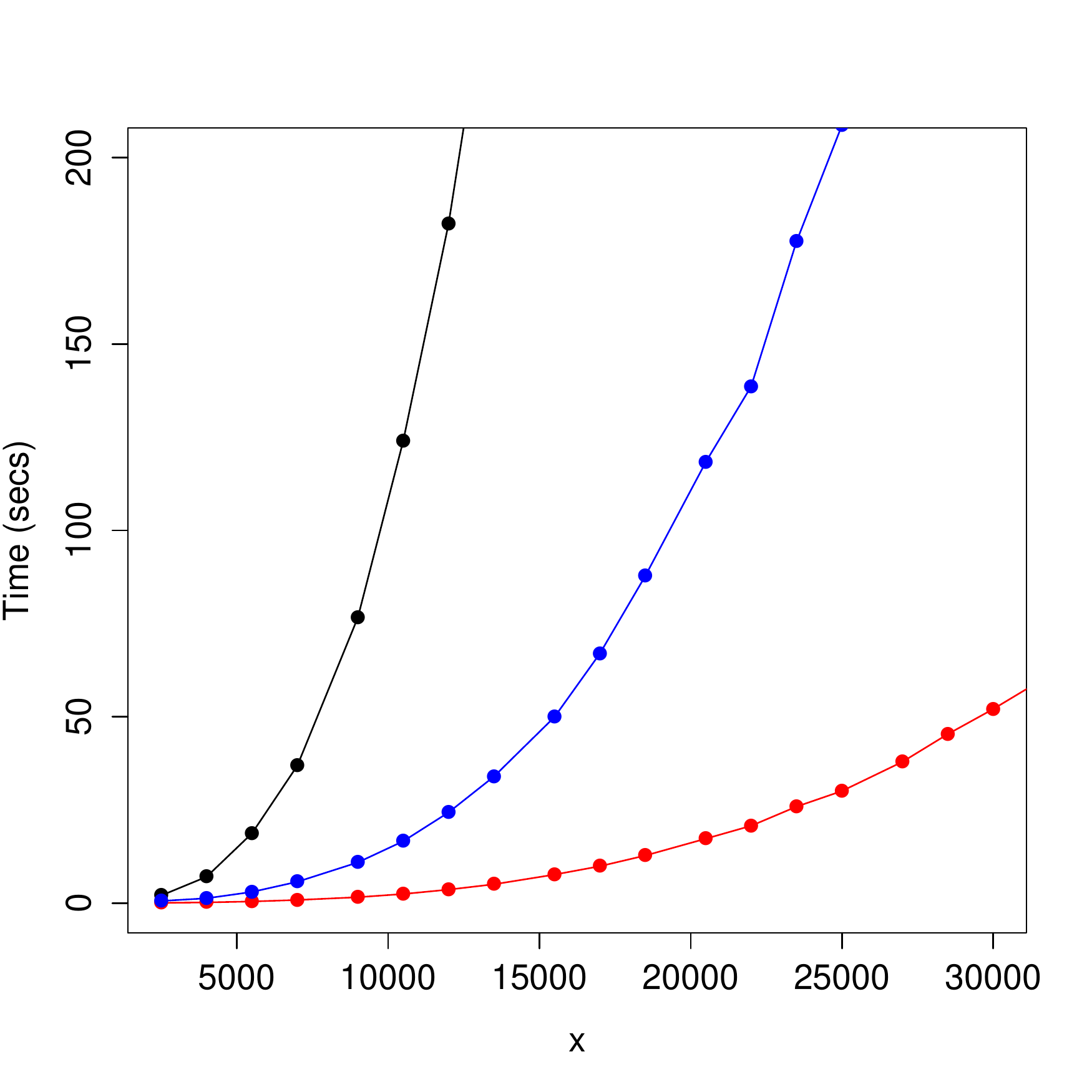} \\
 & \sf $p$   &  \sf $p$ \\
\end{tabular}}
\caption{{\small{Figure showing the times in seconds to perform a direct eigen-decomposition, inversion and Cholesky decomposition using dense direct numerical linear algebra methods, for real symmetric matrices
with size of upto $p=65,000$. Eigen decompositions and matrix inversions are less memory friendly, when compared to Cholesky decompositions for large problem sizes.
The timings displayed in the graphs support the practical feasibility of using Cholesky decomposition methods for large matrices---a main workhorse for the stochastic optimization algorithms proposed in the paper.
[Right panel] displays a zoomed in version of the left panel plot, showing that Cholesky decompositions are significantly faster than inversion and eigen-decomposition methods even for smaller problems $p \leq 30,000$.
%We have observed that treating the matrices as \emph{dense} is computationally fiendlier than sparse matrices, in part because dense linear algebra in \textsc{Matlab} are multithreaded. 
The tail of the direct inversion curve on the left deviates from the $O(p^3)$ trend because the storage requirement has exceeded the capacity of main memory. Thus, the extra time is consumed by the slower virtual memory access. [The matrices used here were sparse with proportion of non-zeros $10/p$, positive definite with the reciprocal of the condition number given by $0.2$---we used the \textsc{Matlab} function {\texttt{sprandsym}} to generate the matrices.]} }}\label{fig-inv-chol-evd-1}
\end{figure}

 Stochastic optimization algorithms based on noisy estimates of the gradient have a long history that goes back to the pioneering works of~\cite{robbinsMonro51,kieferwolfowitz52}. 
 As datasets encountered by statisticians in the modern day grow larger and the optimization problems associated with statistical estimation tasks become increasingly challenging,
 the importance of stochastic algorithms to deliver scalable solvers is being progressively recognized in recent years. See for instance, the recent works in the optimization and machine learning communities (\cite{bertsekas:11,duchietal12, shalev-shwartz:zhang13,konecny:richtarick13,xiao:zhang14,atchade:etal14}, and the references therein). We note however, that our stochastic optimization formulation of Problem~(\ref{eq-glasso-enet-1}) differs from the usual stochastic optimization problem (for instance as in \cite{bertsekas:11}) which solves problems of the form 
 
\begin{equation}\label{sto:probl}
\mini_{\theta} \;\;\; \int f(\theta;x)\pi(\rmd x) + g(\theta),\end{equation}
for an intractable integral $\int  f(\theta;x)\pi(\rmd x)$, where, the map $\theta\mapsto  f(\theta;x)$ is smooth, and $g$ is possibly non-smooth. 
A special instance of~\eqref{sto:probl} is when
 $\pi$ is a discrete probability distribution over a very large set, making the integral $\int  f(\theta;x)\pi(\rmd x)=\frac{1}{N}\sum_{i=1}^N f(\theta;x_i)$ a large sum and difficult to work with.
%\footnote{In many instances $\pi$ is a discrete probability distribution over a very large set, making the integral $\int  f(\theta;x)\pi(\rmd x)=\frac{1}{N}\sum_{i=1}^N f_(\theta;x_i)$ intractable}. 
We make the following remarks that highlight the differences between our approach and generic approaches for Problem~\eqref{sto:probl}:
\begin{itemize}
\item  Problem~\eqref{eq-glasso-enet-1} does not admit a straightforward representation of the form (\ref{sto:probl}). 
\item The gradient $\nabla f(\theta)=S-\theta^{-1}$ has the integral representation $S-\int xx'\pi_\theta(\rmd x)$, where $\pi_\theta$ is the density of $\textbf{N}(0, \theta^{-1})$, which depends on $\theta$ --- a distinctive feature that sets our stochastic optimization framework apart from Problem~\eqref{sto:probl}.  
\item Last, but not least, the gradient map $\theta \mapsto \nabla f(\theta)$ is not Lipschitz continuous on $\M_+$, the feasible set of Problem~\eqref{eq-glasso-enet-1}.
\end{itemize}
A main contribution of our paper is to address the above challenges in the context of the stochastic optimization framework being proposed herein. 
In fact, our stochastic optimization framework is more in sync with the Robbins-Monro algorithm (\cite{robbinsMonro51}) and can be viewed as a large-scale and non-smooth variant of the Robbins-Monro algorithm, along the lines of \cite{atchade:etal14}. Note however, that the theory of \cite{atchade:etal14}  cannot be directly applied here, as it requires the classical Lipschitz-continuity assumption of the smooth component of the objective function, and the ability to compute the proximal map  of the non-smooth component. As explained above, these properties are not readily available in our case.

The main cost of Algorithm~2 lies with generating multivariate Gaussian random variables from $\textbf{N}(0,\theta^{-1})$. A given iteration of Algorithm~\ref{algo1} is more cost-effective than an iteration of the deterministic algorithm, if the Monte Carlo sample size used in that iteration is smaller than $p$. This is because the cost of approximating $\theta^{-1}$ using $p$ random samples from $\textbf{N}(0,\theta^{-1})$ is similar to the cost of computing 
 $\theta^{-1}$ by direct matrix inversion. We show that with an appropriate choice of the Monte Carlo batch size sequence $\{N_k\}$ (see Section \ref{sec:cost-chol-1} for details), Algorithm~\ref{algo1} reaches a solution with accuracy $\deltamod$, before the Monte Carlo sample size becomes larger than $p$ if $p\geq \textsf{cond}(\thetaen)^2\deltamod^{-1}$. This result implies that  Algorithm~\ref{algo1} is more cost-effective than Algorithm~\ref{algo0} in finding $\deltamod$-accurate solutions in cases when $p$ is large, the solution $\thetaen$ is well-conditioned, and we seek a low-accuracy approximation of $\thetaen$. The total cost of Algorithm~\ref{algo1} is then $O\left(p^3\textsf{cond}(\thetaen)^2\log(\deltamod^{-1})\right)$. While on the surface, the cost looks similar to Algorithm~1 which performs a direct matrix inversion at every iteration, 
 the constant involved in the big-O notation favors Algorithm~2 ---see for example, Figure~\ref{fig-inv-chol-evd-1} showing the differences in computation times between a dense Cholesky decomposition and a direct dense matrix inversion.
This is further substantiated in our numerical experiments (Section~\ref{sec:numerics}) where we do systematic comparisons between Algorithms~1 and 2.

A deeper investigation of our stochastic optimization scheme (Algorithm~\ref{algo1}) outlined above, reveals the following. 
At each iteration $k$, all the random variables (samples) used to estimate $\theta_{k-1}^{-1}$ are discarded, and new random variables are generated to 
approximate $\theta_k^{-1}$. We thus ask, is there a modified algorithm that makes clever use of the information associated with an approximate 
$\theta_{k-1}^{-1}$ to approximate $\theta_{k}^{-1}$? In this vein, 
we propose a new algorithm: Algorithm~\ref{algo2} which recycles previously generated samples.
Algorithm~\ref{algo2} has a per-iteration cost of $O(Np^3)$ when a Cholesky factorization  is used to generate the Gaussian random variables, and where $N$ is the Monte Carlo batch-size. The behavior of the algorithm is more complex, and thus developing a rigorous convergence guarantee with associated computational guarantees analogous to Algorithm~\ref{algo1} is beyond the 
scope of the current paper.  We however, present some global convergence results on the algorithm. 
In particular, we show that when the sequence produced by Algorithm~\ref{algo2} converges, it necessarily converges to the solution of Problem~\eqref{eq-glasso-enet-1}. 
%A thorough analysis of its convergence behavior is left as future work.

\subsubsection*{Dense problems} We emphasize that a sizable component of our work relies on the speed and efficiency of modern dense numerical linear algebra methods for scalability, 
and thus our approach is relatively agnostic to the sparsity level of $\hat{\theta}$, a solution to Problem~\eqref{eq-glasso-enet-1}. 
In other words, our approach adapts to Problem~\eqref{eq-ridge} for large $n$ and $p$, a problem which is perhaps not favorable for several current specialized implementations 
for Problem~\eqref{eq-glasso-1}.

\subsubsection*{Exact covariance thresholding} We also extend the exact thresholding rule \citep{mazumder2012exact} originally proposed for the \Glasso problem, to the more general case of Problem~\eqref{eq-glasso-enet-1}. Our result established herein, 
implies that the connected components of the graph $(1(|s_{ij}| > \lambda \alpha))$ are \emph{exactly} equal to the connected components of  the graph induced by the non-zeros of $\hat{\theta}$, a solution to Problem~\eqref{eq-glasso-enet-1}.
This can certainly be used as a wrapper around any algorithm to solve Problem~\eqref{eq-glasso-enet-1}; and leads to dramatic performance gains whenever the size of the largest connected component of 
$\left(1\left(|s_{ij}| > \lambda \alpha\right)\right)$  is sufficiently smaller than $p$.

We note that developing the fastest algorithmic implementation for Problem~\eqref{eq-glasso-enet-1} or its special case, \textsc{Glasso}, is neither the intent nor focus of this paper.
 We view our work 
as one that proposes a new framework based on stochastic optimization that enables the \emph{scalable computation} for the general class of Problems~\eqref{eq-glasso-enet-1}, across a wide range of 
the regularization parameters. The scalability properties of our proposal seem to be favorable over deterministic batch methods and in particular, proximal gradient descent methods tailored for Problem~\eqref{eq-glasso-enet-1}.

 \subsection{Notation}\label{subsec:notations}
Throughout the paper, the regularization parameters $\lambda_{}$ and $\alpha\in (0,1]$, appearing in Problem (\ref{eq-glasso-enet-1}) are assumed fixed and given.  Let $\M$ denote the set of $p\times p$ symmetric matrices with inner product $\pscal{A}{B}=\textsf{Tr}(A'B)$ and the Frobenius norm $\normfro{A}\eqdef \sqrt{\pscal{A}{A}}$. $\M_+$ denotes the set of positive definite elements of $\M$. Let $f$ be the  function $\M\to (0,\infty]$ defined by
%\[f(\theta) =-\log\det\theta + \textsf{Tr}(\theta S),\;\;\;\theta\in\M_+,\;\;\;\mbox{ and }\;\;\; f(\theta)=+\infty \;\;\mbox{ for }\;\;\theta\in\M\setminus\M_+.\]
$$
f(\theta) = \begin{cases} -\log\det\theta + \textsf{Tr}(\theta S) \;\ & \text{if $\theta\in\M_+$}\\
+\infty \;\; & \mbox{if}\;\;\theta\in\M\setminus\M_+.
\end{cases}
$$

We shall write the regularization term in Problem~(\ref{eq-glasso-enet-1}) as
\[g_\alpha(\theta)\eqdef \sum_{ij} \left(\alpha \lambda_{} | \theta_{ij} | + \frac{(1-\alpha)}{2}\lambda_{} \theta_{ij}^2 \right),\;\;\]
and 
\begin{equation}\label{defn-phi-obj}
\phi_\alpha(\theta)\eqdef f(\theta) + g_\alpha(\theta),\;\;\theta\in\M .
\end{equation}
For a matrix $A\in\M$, $\|A\|_2$ denotes the spectral norm of $A$, $\lambda_{\textsf{min}}(A)$ (respectively $\lambda_{\textsf{max}}(A)$) denotes the smallest (respectively, the largest) eigenvalue of $A$, and $\|A\|_1\eqdef \sum_{i,j}|A_{ij}|$. For a subset $\D\subseteq\M$,  $\iota_{\D}$ denotes the indicator function of $\D$, i.e.
\[\iota_{D}(u)\eqdef \left\{\begin{array}{cl} 0 & \mbox{ if }u\in\D\\ +\infty & \mbox{ otherwise. }\end{array}\right.\]
For $\theta\in\M_+$, and $\gamma>0$, we denote the proximal operator associated with Problem~(\ref{eq-glasso-enet-1}) as
\begin{equation}\label{idealTmap}
 \bar T_\gamma(\theta;\alpha) \eqdef \argmin_{u\in\M_+} \left\{g_\alpha(u)+ \frac{1}{2\gamma}\normfro{u-\theta +\gamma(S-\theta^{-1})}^2\right\}.\end{equation}
 
For $0<\epsilonmod\leq \kappamod$, we define
\[\M_+(\epsilonmod,\kappamod)\eqdef\left\{\theta\in\M_+:\;\lambda_{\textsf{min}}(\theta)\geq \epsilonmod,\;\mbox{ and }\; \lambda_{\textsf{max}}(\theta)\leq \kappamod\right\}.\]
%
%
%%%%%%%%%%%
%\newpage
%
%\section*{\textcolor{red}{New Results start here}}
%
%\input{more_results_RM.tex}
%
%\newpage
%
%%%%%%%%%%%
%

 %We gather in Section~\ref{related-work} a more detailed description of some of the main approaches for solving Problem~\eqref{eq-glasso-1}.

\section{A proximal gradient algorithm for Problem~\eqref{eq-glasso-enet-1}}\label{sec:prox:grad}
We begin this section with a brief review of proximal gradient algorithms, following~\cite{nest-07-new}, which concerns the minimization of the following generic convex optimization problem:
\begin{equation}\label{gen-form-prox-grad}
\min_{\omega \in \Omega} \;\;  \left \{\bar{\phi}(\omega) \eqdef \bar{f}(\omega) + \bar{g}(\omega) \right \},  
\end{equation}
where, $\Omega$ is a convex subset of a Euclidean space with norm $\|\cdot\|$; $\bar{g}(\cdot)$  is a closed convex function and $\bar{f}(\cdot)$ is convex, smooth on $\Omega$ satisfying:
\begin{equation}\label{lip-cond-1}
\| \nabla \bar{f}(\omega) - \nabla \bar{f}(\omega') \| \leq \bar{L}\| \omega - \omega'\|,
\end{equation}
for $\omega, \omega' \in \Omega$ and $0 < \bar{L} < \infty$. The main ingredient in proximal gradient descent methods is the efficient computation of the proximal-operator (``prox-operator'' for short), given by:
\begin{equation}\label{prox-map-gen-1}
\bar{T}_{\gamma}( \bar{\omega}) \eqdef \argmin_{ \omega \in \Omega} \;\; \left\| \omega - \left(\bar{\omega} - \gamma \nabla f(\bar{\omega}) \right) \right\|^2 +  \bar{g}(\omega),
 \end{equation}
for some choice of $0<\gamma \leq 1/\bar L$. The following simple recursive rule:
$$ \omega_{k+1} =  \bar{T}_{\gamma}( \omega_{k} ),  \quad k \geq 1,$$
for some initial choice of $\omega_{1} \in \Omega$ and $\gamma = 1/\bar{L},$ then leads to a solution of Problem~\eqref{gen-form-prox-grad} (See for example,~\cite{nest-07-new}).

Problem~\eqref{eq-glasso-enet-1} has striking similarities to an optimization problem of the form~\eqref{gen-form-prox-grad}, with 
$\bar{f}(\cdot)=f(\cdot)$, $\Omega = \M_{+}$ endowed with the Frobenius norm, $\bar{g}(\cdot) = g_\alpha(\cdot)$, and with $\bar T_\gamma$ given by (\ref{idealTmap}).
However, the use of the proximal gradient algorithm for Problem~\eqref{eq-glasso-enet-1} presents some immediate challenges since:
\begin{itemize}
\item The gradient of the smooth component, namely, $ \nabla f(\theta) = - \theta^{-1} + S$ is not Lipschitz on the entire domain $\M_+$ (as required in (\ref{lip-cond-1})), due to the unboundedness of the map  
$\theta \mapsto \theta^{-1}$.  

\item The corresponding proximal map $\bar T_\gamma(\cdot\; ; \alpha)$ defined in (\ref{idealTmap}) need not be simple to compute.

\end{itemize}
Our first task in this paper, is to show how each of the above problems can be alleviated. 
%We present a self-contained analysis of this, which is different from 
%the one presented in~\cite{rolfs2012iterative}
We note that~\cite{rolfs2012iterative} also analyze a proximal gradient descent algorithm for the case $\alpha =1$. 
We present here a self-contained analysis:  our proofs have some differences with that of~\cite{rolfs2012iterative}; 
and lays the foundation for the stochastic optimization scheme that we analyze subsequently.
Loosely speaking, we will show that even if the function $f(\theta)$ does not have Lipschitz continuous gradient on the \emph{entire} feasible set $\M_{+}$, it \emph{does} satisfy~\eqref{lip-cond-1}
across the iterations of the proximal gradient algorithm.
In addition, we demonstrate that by appropriately choosing the step-size $\gamma$, the proximal map $\bar T_\gamma(\cdot\;;\alpha)$ can be computed by ``dropping'' the constraint $\theta \in \M_{+}$.
We formalize the above in the following discussion.

 For $\alpha\in[0,1]$, $\gamma>0$, and $\theta\in\M$ (i.e., the set of $p\times p$ symmetric matrices), the proximal operator associated with the function $g_\alpha$ is defined as 
$$\Prox_\gamma(\theta;\alpha) \eqdef  \argmin_{u\in \M}\left\{ g_\alpha(u) + \frac{1}{2\gamma}\normfro{u-\theta}^2\right\}.$$
This operator has a very simple form. It is a matrix whose $(i,j)$th entry is given by:
% $\textsf{s}_\gamma(\theta;\alpha)_{jj}=\theta_{jj}(1+(1-\alpha)\lambda_{}\gamma)^{-1}$ for $1\leq j\leq p$, and for $i\neq j$,

\begin{equation}\label{prox:formula}
\left(\Prox_\gamma(\theta;\alpha)\right)_{ij}=\left\{\begin{myarray}[1.2]{lc} 0 & \mbox{ if } |\theta_{ij}|<  \alpha\lambda_{} \gamma\\
\frac{\theta_{ij}-\alpha\lambda_{}\gamma}{1+(1-\alpha)\lambda_{}\gamma} & \mbox{ if } \theta_{ij}\geq \alpha\lambda_{}\gamma \\ \frac{\theta_{ij}+\alpha\lambda_{}\gamma}{1+(1-\alpha)\lambda_{}\gamma} & \mbox{ if } \theta_{ij}\leq  -\alpha\lambda_{}\gamma\, .\end{myarray}\right.
\end{equation}

For $\gamma>0$, and $\theta\in\M_+$, we consider a seemingly minor modification of the operator~\eqref{idealTmap}, given by:
\begin{equation}\label{Tmap}
\begin{myarray}[1.2]{ccl}
T_\gamma(\theta;\alpha)&\eqdef& \argmin\limits_{u\in\M} \left\{g_\alpha(u)+ \frac{1}{2\gamma}\normfro{u-\theta +\gamma(S-\theta^{-1})}^2\right\}\\
&=&\Prox_\gamma\left(\theta-\gamma(S-\theta^{-1});\alpha\right).
\end{myarray}
\end{equation}
%\begin{eqnarray}\label{Tmap}
%T_\gamma(\theta;\alpha)&\eqdef& \argmin_{u\in\M} \left\{g_\alpha(u)+ \frac{1}{2\gamma}\|u-\theta +\gamma(S-\theta^{-1})\|^2\right\}\\
%&=&\Prox_\gamma\left(\theta-\gamma(S-\theta^{-1});\alpha\right)\nonumber.\end{eqnarray}
Compared to (\ref{idealTmap}), one can notice that in (\ref{Tmap}) the positive definiteness constraint is relaxed. It follows from (\ref{prox:formula}) that $T_\gamma(\theta;\alpha)$ is straightforward to compute. Notice that if $T_\gamma(\theta;\alpha)$ is positive definite, then $T_\gamma(\theta;\alpha) = \bar T_\gamma(\theta;\alpha)$. We will show that if $\gamma$ is not too large then indeed $T_\gamma(\theta;\alpha)=\bar T_\gamma(\theta;\alpha)$ for all $\theta$ in certain subsets of $\M_+$.

At the very onset, we present a result which provides bounds on the spectrum of $\hat{\theta}$, a solution to Problem~\eqref{eq-glasso-enet-1}.
The following lemma can be considered as a generalization of the result of~\cite{Lu:09} obtained for the \Glasso problem (with $\alpha=1$). Let us define the following quantities: $\lambda_{1} \eqdef  \alpha \lambda$, $\lambda_{2} \eqdef  (1-\alpha)\lambda/2$, $\mu\eqdef  \| S\|_{2} + \lambda_{1} p$, 
\begin{equation}\label{for-thm-st-1}
\begin{aligned}
\ell_{\star} \eqdef &  \begin{cases} 
\frac{-\mu + \sqrt{\mu^2 + 8 \lambda_{2}}}{4 \lambda_{2}} &\text{if }\alpha\in [0,1) \\
\frac{1}{\mu} & \mbox{if } \alpha=1,
\end{cases}\\
U_{1} \eqdef & \frac{1}{\lambda_{1}} \left(p - \ell_{\star} \tr(S) - 2p\lambda_{2} \ell_{\star}^2 \right)\\
c(t) \eqdef & \frac{1}{\lambda_{1}(1 - t)} \left(  \lambda_{1} \| \theta(t) \|_{1} - t \lambda_{1} \tr(\theta(t)) + \lambda_{2} \normfro{\theta(t)}^2 \right) -  \frac{\lambda_{2}  \ell_{\star}^2p }{\lambda_{1}(1 - t)}\\
U_{2} \eqdef & \inf_{t \in (0,1)} c(t),
\end{aligned}
\end{equation}
where, we take $t\in(0,1)$ and $\theta(t) \eqdef  (S+t\lambda_1 I)^{-1}$.

\begin{lemma}\label{lem:spec-bounds-1} If $\hat{\theta}$ is a solution to Problem~\eqref{eq-glasso-enet-1} with $\lambda>0$, then $\hat{\theta}$ is unique, and $\hat\theta \in\M_+(\ell_\star,\psi_{UB})$, with $\psi_{UB} = \min \{ U_{1}, U_{2} \},$ where, $U_{1},U_{2}$ are as defined in~\eqref{for-thm-st-1}. In other words, we have the following bounds on the spectrum of $\hat{\theta}$
$$ \lambda_{\min}(\hat{\theta}) \geq \ell_{\star}, \;\;\;\;\lambda_{\max}(\hat{\theta}) \leq \psi_{UB}. $$

%where, 
%\begin{equation}
%\ell_{\star} :=  \begin{cases} 
%\frac{-\mu + \sqrt{\mu^2 + 8 \lambda_{2}}}{4 \lambda_{2}} &\text{if $\lambda_{2} \neq 0$}\\
%\mu & \text{otherwise},
%\end{cases}
%\end{equation}
%with ;
%and $\ell_{UB} = \min \{ U_{1}, U_{2} \},$ with $U_{1},U_{2}$ are as defined in~\eqref{for-thm-st-1}. 
\begin{proof}
The proof is presented in Section~\ref{proof-lem:spec-bounds-1}.
\end{proof}
\end{lemma}
%\begin{remark}

We make a few remarks about the bounds in~\eqref{for-thm-st-1}.

\medskip

$\bullet$ Computing $U_{2}$ requires performing a one dimensional minimization which can be carried out quite easily. Conservative but \emph{valid} 
bounds can be obtained by replacing $U_{2}$ by evaluations of $c(\cdot)$ at some values of $t \in (0,1)$ for example: $t = \frac12$ and $t = 0+$ (provided $S$ is invertible). 
%$\square$\end{remark}
%\begin{remark}

$\bullet$ Since the condition number of $\thetaen$ is $\textsf{cond}(\thetaen)=\lambda_{\textsf{max}}(\thetaen)/\lambda_{\textsf{min}}(\thetaen)$, the result above implies that $\textsf{cond}(\thetaen)\leq \psi_{UB}/\ell_\star$. We note, however, that this upper bound $\psi_{UB}/\ell_\star$ may not be an accurate estimate of  $\textsf{cond}(\thetaen)$.%

%%$\square$\end{remark}

\medskip

We now present an important property (Lemma~\ref{keylem}) of the proximal gradient update step, for our problem. Towards this end, we 
define  $\nu \eqdef \lambda_{\textsf{min}}(S)-\lambda_1p$ and
 \[\psi_\star^1 \eqdef  \left\{\begin{array}{ll}\frac{-\nu +\sqrt{\nu^2+8\lambda_2}}{4\lambda_2} & \mbox{ if } \alpha\in [0,1),\\
 \frac{1}{\nu} & \mbox{ if } \alpha=1 \mbox{ and } \nu>0\\
 +\infty & \mbox{ if } \alpha=1 \mbox{ and } \nu\leq  0. \end{array}\right.\]
 It is obvious that $0<\epsilonmod_\star\leq \kappamod^1_\star\leq \infty$. We also define
 \[\psi_\star \eqdef \min\left(\psi_\star^1,  \psi_{UB} +\sqrt{p}\left(\psi_{UB}-\ell_\star\right)\right).\]

\begin{lemma}\label{keylem}
Take $\gamma\in (0, \epsilonmod_\star^2]$ and let $\{\theta_j,\;j\geq 0\}$ be a sequence such that $\theta_j = T_\gamma(\theta_{j-1};\alpha)$, for $j\geq 1$. If $\theta_{0} \in \M_+(\epsilonmod_\star, \min \{ \psi_{UB}, {\psi}^1_\star \})$, then $\theta_j\in \M_+(\ell_\star,\psi_\star)$ for all $j\geq 0$.
%In particular, we have $T_\gamma(\theta_{j};\alpha) = \bar T_\gamma(\theta_{j};\alpha),$ for all $j \geq 0$.
%Furthermore, Problem (\ref{eq-glasso-enet-1}) has a unique solution $\hat\theta$, and $\hat\theta\in\M_+(\epsilonmod_\star,\kappamod_\star)$.
\end{lemma}
\begin{proof}
See Section \ref{proof:keylem}.
\end{proof}

In the special case of \Glasso ($\alpha=1$), the results of Lemma \ref{keylem} correspond to those obtained by \cite{rolfs2012iterative}. Our proof, however, has differences since 
we rely more heavily on basic properties of proximal maps.

Lemma~\ref{keylem} shows that for appropriate choices of $\gamma>0$, the two proximal maps $T_\gamma$ and $\bar T_\gamma$ produce the identical sequences that remain 
in the set $\M_+(\epsilonmod_\star,\kappamod_\star)$. This suggests that one can solve Problem~\eqref{eq-glasso-enet-1} using the proximal operator $T_\gamma$, as the next result shows.

\begin{theorem}\label{thm1}
Fix  arbitrary $0<\epsilonmod < \kappamod<\infty$. For $k\geq 1$, let $\{\theta_{j},\; 0 \leq j  \leq k\} $ be a sequence obtained via the map $T_\gamma$: $\theta_{j+1}=T_\gamma(\theta_j;\alpha)$, for some $\gamma\in (0,\epsilonmod^2]$. Suppose that $\hat\theta,\theta_j\in\M_+(\epsilonmod,\kappamod)$, $0\leq j\leq k$.
%\textcolor{red}{RM: Yves, why do we need to assume that $\hat\theta,\theta_k\in\M_+(\epsilonmod,\kappamod)$?; isn't it implied?}
Then
\begin{equation}\label{bound:thm1}
\normfro{\theta_k-\thetaen}^2  \leq \rho^k \normfro{\theta_0-\thetaen}^2,\;\;\mbox{and } \;\;\left\{\phi_\alpha(\theta_k)-\phi_\alpha(\hat\theta)\right\} \leq
 \frac{\normfro{\theta_0-\hat\theta}^2}{2\gamma}\min\left \{ \frac{1}{k},\rho^k\right \},\end{equation}
where $\rho= 1-\frac{\gamma}{\kappamod^2}$.
\end{theorem}
\begin{proof}
See Section \ref{proof:thm1}.
\end{proof}

\begin{remark}
If $\epsilonmod=\epsilonmod_\star$ and $\kappamod=\kappamod_\star$, and $\theta_{0} \in \M_+(\epsilonmod_\star, \min \{ \psi_{UB}, {\psi}^1_\star \})$, then the assumption that $\hat\theta,\theta_j\in\M_+(\epsilonmod,\kappamod)$, $0\leq j\leq k$ is redundant, as shown in Lemma \ref{lem:spec-bounds-1}-\ref{keylem}, and (\ref{bound:thm1}) holds.
$\square$\end{remark}

An appealing feature of the iteration $\theta_{k+1} = T_\gamma(\theta_k;\alpha)$ is that its convergence rate is \emph{adaptive}, i.e., the algorithm automatically adapts itself to the fastest possible convergence rate dictated by the condition number of $\thetaen$. This is formalized in the following corollary:
\begin{corollary}\label{thm1:coro0}
Let $0<\epsilonmod_{\star\star}<\kappamod_{\star\star}<\infty$ be such that $\lambda_{\textsf{min}}(\thetaen)>\epsilonmod_{\star\star}$, and $\lambda_{\textsf{max}}(\thetaen)<\kappamod_{\star\star}$. Let $\{\theta_k,\;k\geq 0\}$ be a sequence obtained via the map $T_\gamma$: $\theta_{j+1}=T_\gamma(\theta_j;\alpha)$, for some $\gamma\in (0,\epsilonmod_{\star\star}^2]$.  If $\lim_k\theta_k=\thetaen$, then there exists $k_0\geq 0$, such that for all $k\geq k_0$,  
\[\normfro{\theta_k-\thetaen}^2  \leq \left(1-\frac{\gamma}{\kappamod_{\star\star}^2}\right)^{k-k_0} \normfro{\theta_{k_0}-\thetaen}^2.\]
\end{corollary}
\begin{proof}
By assumption, $\thetaen$ belongs to the interior of $\M_+(\epsilonmod_{\star\star},\kappamod_{\star\star})$. Since $\theta_k\to\thetaen$, there exists $k_0\geq 0$, such that $\theta_k\in \M_+(\epsilonmod_{\star\star},\kappamod_{\star\star})$ for $k\geq k_0$. Then we apply the bound (\ref{bound:thm1}), and the lemma follows.
\end{proof}

The analysis above suggests  the following practical algorithm for Problem~\eqref{eq-glasso-enet-1}. Let $\{\gamma_k\}$ denote a sequence of positive step-sizes with $\lim_k\gamma_k=0$. An example of such a sequence is $\gamma_k=\gamma_0/2^k$, for some $\gamma_0>0$.  For convenience, we 
summarize in Algorithm~1, the deterministic proximal gradient algorithm for Problem~\eqref{eq-glasso-enet-1}.

\begin{algorithm}[\textbf{Deterministic Proximal Gradient}]\label{algo0}~\\
Set $\textsf{r}=0$.
\begin{enumerate}
\item Choose $\theta_0\in\M_+$.
\item  Given $\theta_k$, compute: $\theta_{k+1} = T_{\gamma_\textsf{r}}\left(\theta_k;\alpha\right).$
\item If $\lambda_{\textsf{min}}(\theta_{k+1})\leq 0$, then restart: set $k\leftarrow 0$, $\textsf{r}\leftarrow \textsf{r}+1$, and go back to (1). Otherwise, set 
$k\leftarrow k+1$ and go back to (2).
\end{enumerate}
\end{algorithm}

We present a series of remarks about Algorithm~\ref{algo0}:
%\begin{remark}[Postive Definiteness]

\smallskip
\smallskip

\indent $\bullet$ \emph{Positive Definiteness}. In Step 3, positive definiteness is tested and the algorithm is restarted with a smaller step-size, if $\theta_{k+1}$ is no longer positive definite. 
The smallest eigenvalue of $\theta_{k+1}$, i.e.,  $\lambda_{\textsf{min}}(\theta_{k+1})$ can be efficiently computed by several means: (a) it can be computed via the Lanczos process (see e.g. \cite{golub:vl}~Theorem 10.1.2); 
(b) it may also be computed as a part of the step that approximates the spectral interval of $\theta_{k+1}$ using the procedure of~\cite{chen:etal:2011} (c) a  Cholesky decomposition of $\theta_{k+1}$ also returns 
information about whether $\theta_{k+1}$ is positive definite or not.
 
 An efficient implementation of the algorithm is possible by making Step 3 implicit. For instance the positive definiteness of $\theta_{k+1}$ can be checked as part of the computation of the gradient $\nabla f(\theta_{k+1})=S-\theta_{k+1}^{-1}$ in Step 2. 
%Furthermore, when the algorithm is restarted, the possibly better initial value $\theta_k$ can be used instead of $\theta_0$.
%$\square$\end{remark}
%\begin{remark}[Step Size]

\medskip

\indent $\bullet$ \emph{Step Size}. If the initial step-size satisfies $\gamma_0\leq \epsilonmod_\star^2$ and $\theta_0\in\M_+(\epsilonmod_\star,\kappamod_\star)$, the algorithm is never re-initialized according to Lemma \ref{keylem}, and Theorem \ref{thm1} holds. However, it is important to notice that Lemma \ref{keylem} and Theorem \ref{thm1} present a worst case analysis scenario and in practice the choice $\gamma_0=\epsilonmod_\star^2$ can be overly conservative.  In fact, Corollary \ref{thm1:coro0} dictates that a better choice of step-size is $\gamma_0=\lambda_{\textsf{min}}(\thetaen)^2$. Obviously $\lambda_{\textsf{min}}(\thetaen)$ is rarely known, but what this implies is that, in practice, one should initialize the algorithm with a large step-size and rely on the re-start trick (Step 3) to reduce the step-size, when $\theta_{k+1}$ is not positive definite.
%$\square$\end{remark}

\medskip

%\begin{remark}[Adaptive Convergence Rate]\label{rem:costAlgo1}
\indent $\bullet$ \emph{Adaptive Convergence Rate}. We have seen in Corollary \ref{thm1:coro0} that the convergence rate of the sequence $\{\theta_k\}$ improves with the iterations. This adaptive convergence rate behavior makes the cost-complexity analysis of Algorithm~\ref{algo0} more complicated. However, to settle ideas, if we set $\theta_0$ close to $\thetaen$, and the step-size obeys $\gamma\approx \lambda_{\textsf{min}}(\thetaen)^2$, Theorem \ref{thm1} and Corollary \ref{thm1:coro0} imply that the number of iterations of Algorithm~\ref{algo0} needed to reach the precision $\deltamod$ (that is $\normfro{\theta_k-\hat\theta}^2\leq \deltamod$) is 
\[O\left(-\frac{\kappamod_{\star\star}^2}{\epsilonmod_{\star\star}^2}\log\deltamod\right)\approx O\left(-\textsf{cond}(\thetaen)^2\log\deltamod\right).\]

\medskip

\indent $\bullet$ \emph{Computational Cost}. The bottleneck of Algorithm~\ref{algo0} is the computation of the inverse $\theta_k^{-1}$, which in general entails a computational cost of $O(p^3)$---See Figure~\ref{fig-inv-chol-evd-1}
showing the computation times of matrix inversions for real symmetric $p \times p$ matrices, in practice.
It follows that in the setting considered above, the computational cost of Algorithm~\ref{algo0} to achieve a $\deltamod$-accurate solution is $O\left(p^3\textsf{cond}(\thetaen)^2\log(1/\deltamod)\right)$.
%$\square$\end{remark}

\section{Stochastic Optimization Based Algorithms}\label{sec:sto:prox:grad}
When $p$ is large (for example, $p=5,000$ or larger), the computational cost of Algorithm~1 becomes prohibitively 
expensive due to the associated matrix inversions---this is a primary motivation behind the stochastic optimization methods that we develop in this section.  
For $\theta\in\M_+$, let $\pi_\theta$ denote the density of $\textbf{N}(0,\theta^{-1})$, the 
mean-zero normal distribution on $\rset^p$ with covariance matrix $\theta^{-1}$. We begin with the elementary observation that  
\[\theta^{-1} = \int z z'\pi_\theta(\rmd z).\]
This suggests that on $\M_+$, we can approximate the gradient $\nabla f(\theta)=S-\theta^{-1}$ by $S-N^{-1}\sum_{j=1}^N z_jz_j'$, where $z_{1:N}\stackrel{\text{i.i.d.}}{\sim}\pi_\theta$; here, the notation $z_{1:N}$ denotes a collection of random vectors $z_{i}, i \leq N$. 

To motivate the stochastic algorithm we will first establish an analog of Lemma \ref{keylem}, showing that iterating the stochastic maps obtained by replacing $\theta_{j-1}^{-1}$
 in computing $T_\gamma(\theta_{j-1};\alpha)$ in~\eqref{Tmap} by the Monte Carlo estimate described above, produces sequences that remain positive definite with high probability. Towards this end,  fix $\gamma>0$; 
a sequence of (positive) Monte Carlo batch-sizes: $\{N_k,\;k\geq 1\}$; and consider the stochastic process $\{\theta_k,\;k\geq 0\}$ defined as follows. First, we fix $\theta_0\in\M_+$. For $k\geq 1$, and given the sigma-algebra $\F_{k-1}\eqdef \sigma(\theta_0,\ldots,\theta_{k-1})$:
\begin{equation}\label{sto:proc:1}
\mbox{generate } z_{1:N_k}\stackrel{\text{i.i.d.}}\sim \textbf{N}(0,\theta_{k-1}^{-1}),\;\;\;\mbox{ compute } \;\;\Sigma_k = \frac{1}{N_k}\sum_{j=1}^{N_k}z_jz_j',\end{equation}
and set:
\begin{equation}\label{sto:proc:2}
\theta_{k} =\Prox_\gamma\left(\theta_{k-1}-\gamma \left(S-\Sigma_{k}\right)\right).
\end{equation}
For any $0<\ell\leq \psi\leq \infty$, we set
\[\tau(\epsilonmod,\kappamod)\eqdef\inf\left\{k\geq 0:\; \theta_k\notin \M_+(\epsilonmod,\kappamod)\right\},\]
with the convention that $\inf\emptyset=\infty$. For a random variable $\Psi\geq \epsilonmod$, we define $\tau(\epsilonmod,\Psi)$ as equal to $\tau(\epsilonmod,\psi)$ on $\{\Psi=\psi\}$. 

Given $\epsilon>0$, we define $\mu_\epsilon \eqdef \|S\|_2 + (\lambda_1+\epsilon)p$,
\[\ell_{\star}(\epsilon) \eqdef \left\{ \begin{array}{ll} \frac{-\mu_\epsilon + \sqrt{\mu_\epsilon^2 + 8 \lambda_{2}}}{4 \lambda_{2}} &\text{if }\alpha\in [0,1) \\
\frac{1}{\mu_\epsilon} & \mbox{if } \alpha=1.\end{array}\right.\]
Similarly, define  $\nu_\epsilon \eqdef \lambda_{\textsf{min}}(S)-(\lambda_1+\epsilon)p$, 
 \[\psi_\star^1(\epsilon) \eqdef  \left\{\begin{array}{ll}\frac{-\nu_\epsilon +\sqrt{\nu_\epsilon^2+8\lambda_2}}{4\lambda_2} & \mbox{ if } \alpha\in [0,1),\\
 \frac{1}{\nu_\epsilon} & \mbox{ if } \alpha=1 \mbox{ and } \nu_\epsilon>0\\
 +\infty & \mbox{ if } \alpha=1 \mbox{ and } \nu_\epsilon\leq  0. \end{array}\right.\]
It is easy to check that $0<\ell_\star(\epsilon)\leq \ell_\star\leq \kappamod_\star^1\leq \kappamod_\star^1(\epsilon)\leq \infty$.

The following theorem establishes the convergence of the stochastic process $\theta_{k}$, produced via the stochastic optimization scheme~\eqref{sto:proc:2}.
\begin{theorem} \label{thm3}
Let $\{\theta_k,\;k\geq 0\}$ be the stochastic process defined by the rules~(\ref{sto:proc:1}-\ref{sto:proc:2}). Fix $\epsilon>0$. Suppose that $\theta_0\in\M_+(\ell_\star(\epsilon),\min(\psi_{UB},\psi_\star^1(\epsilon)))$. Then there exists a random variable $\Psi_\star(\epsilon)\geq \ell_\star(\epsilon)$ such that
\[\PP\left[\tau\left(\epsilonmod_\star(\epsilon),\Psi_\star(\epsilon)\right)=\infty\right]\geq 1-4p^2 \sum_{j\geq 1} \exp\left(-\min\left(1,\frac{\epsilon^2\epsilonmod^2_\star(\epsilon)}{16}\right) N_{j-1}\right).\]
If $\sum_j N_j^{-1}<\infty$, then $\PE(\Psi_\star(\epsilon)^2)<\infty$ (hence $\Psi_\star(\epsilon)$ is finite almost surely), and on $\{\tau\left(\epsilonmod_\star(\epsilon),\Psi_\star(\epsilon)\right)=\infty\}$, $\lim_{k\to\infty}\theta_k = \thetaen$.
\end{theorem}
\begin{proof}
See Section \ref{proof:thm3}.
\end{proof}

\subsubsection*{Growth Condition on the Monte Carlo batch size}  If we let the Monte Carlo sample size $N_k$ increase as

%\textcolor{red}{Yves: we need to talk about choice of $\epsilon$ to optimize the bound in Theorem~6} Rahul: 

\[N_k\geq  \frac{3\log p}{\min\left(1,\epsilonmod^2_\star(\epsilon)\epsilon^2/16\right)} + \alpha k^q,\]
for some $q>1$, then $\sum_j N_j^{-1}<\infty$, and the bound in Theorem \ref{thm1} above, becomes
\[\PP\left[\tau\left(\epsilonmod_\star(\epsilon),\Psi_\star(\epsilon)\right)=\infty\right]\geq 1-\frac{4\mu}{p},\]
where $\mu=\sum_{j\geq 0}\exp\left(-\alpha\min(1,\epsilonmod^2_\star(\epsilon)\epsilon^2/16)j^q\right)<\infty$. Hence for high-dimensional problems, and for moderately large Monte Carlo sample sizes, $\PP\left[\tau\left(\epsilonmod_\star(\epsilon),\Psi_\star(\epsilon)\right)=\infty\right]$ can be made very close to one---this guarantees that positive definiteness of the process $\{\theta_k,\;k\geq 0\}$ is maintained and the sequence converges to $\thetaen$, with high probability. The convergence rate  of the process is quantified by the following theorem:
\begin{theorem} \label{thm4}
Let $\{\theta_k,\;k\geq 0\}$ be the stochastic process defined by (\ref{sto:proc:1}-\ref{sto:proc:2}). For some $0<\epsilonmod\leq \kappamod\leq +\infty$, suppose that $\theta_0,\thetaen\in\M_+(\epsilonmod,\kappamod)$, and $\gamma\leq \epsilonmod^2$. 
Then
\begin{multline}\label{bound:thm4}
\PE\left[\textbf{1}_{\{\tau(\epsilonmod,\kappamod)>k\}}\normfro{\theta_{k}-\thetaen}^2\right]\leq \left(1-\frac{\gamma}{\kappamod^2}\right)^k \normfro{\theta_0-\thetaen}^2 \\
+ 2\gamma^2\epsilonmod^{-2}(p+p^2)\sum_{j=1}^kN_j^{-1}\left(1-\frac{\gamma}{\kappamod^2}\right)^{k-j}. \end{multline}
\end{theorem}
\begin{proof}
See Section \ref{proof:thm4}.
\end{proof}

%%\begin{remark}\label{rem:thm4}
As with the deterministic sequence, the convergence rate of the stochastic sequence $\{\theta_k\}$  is determined by the condition number of $\thetaen$. To see this, take $0<\epsilonmod_{\star\star}<\kappamod_{\star\star}<\infty$, such that $\epsilonmod_{\star\star}<\lambda_{\textsf{min}}(\thetaen)$, and $\lambda_{\textsf{max}}(\thetaen)<\kappamod_{\star\star}$. It is easy to show that a conditional version of (\ref{bound:thm4}) holds almost surely: for $0\leq k_0\leq k$, and for $\tau^{k_0}(\epsilonmod,\kappamod)\eqdef \inf\{k\geq k_0:\;\theta_k\notin \M_+(\epsilonmod,\kappamod)\}$, 
\begin{multline}\label{bound:thm4:cond}
\textbf{1}_{\{\theta_{k_0}\in\M_+(\epsilon_{\star\star},\kappamod_{\star\star})\}} \PE\left[\textbf{1}_{\{\tau^{k_0}(\epsilonmod_{\star\star},\kappamod_{\star\star})>k\}}\normfro{\theta_{k}-\thetaen}^2\vert \F_{k_0}\right]\leq \left(1-\frac{\gamma}{\kappamod_{\star\star}^2}\right)^{k-k_0} \normfro{\theta_{k_0}-\thetaen}^2 \\
+ 2\gamma^2\epsilonmod_{\star\star}^{-2}(p+p^2)\sum_{j=k_0+1}^kN_j^{-1}\left(1-\frac{\gamma}{\kappamod_{\star\star}^2}\right)^{k-k_0-j}. \end{multline}
Therefore, as $\theta_k\to \thetaen$ almost surely, and since $\thetaen\in \M_+(\epsilon_{\star\star},\kappamod_{\star\star})$,  one can find $k_0$ such that with high probability, and for all $k\geq k_0$, the following holds:
\begin{equation}\label{eq:rem:thm4}
\textbf{1}_{\{\theta_{k_0}\in\M_+(\epsilon_{\star\star},\kappamod_{\star\star})\}}\textbf{1}_{\{\tau^{k_0}(\epsilonmod_{\star\star},\kappamod_{\star\star})>k\}}=1.\end{equation}
If we make the (strong) assumption that (\ref{eq:rem:thm4}) holds with probability one, then one can deduce  from (\ref{bound:thm4:cond}) that for $k\geq k_0$,
\begin{multline}\label{bound:thm4:final}
\PE\left[\normfro{\theta_{k}-\thetaen}^2\right]\leq \left(1-\frac{\gamma}{\kappamod_{\star\star}^2}\right)^{k-k_0} \PE\left[\normfro{\theta_{k_0}-\thetaen}^2\right] \\
+ 2\gamma^2\epsilonmod_{\star\star}^{-2}(p+p^2)\sum_{j=k_0+1}^kN_j^{-1}\left(1-\frac{\gamma}{\kappamod_{\star\star}^2}\right)^{k-k_0-j}, \end{multline}
which is an analogue of Corollary \ref{thm1:coro0}. As in the deterministic case, this adaptive behavior complicates the complexity analysis of the algorithm. In the discussion below, we consider the idealized case where $\epsilonmod_{\star\star} =\lambda_{\textsf{min}}(\thetaen)$, $\kappamod_{\star\star} = \lambda_{\textsf{max}}(\thetaen)$, and $\theta_0\in\M_+(\epsilonmod_{\star\star},\kappamod_{\star\star})$.
%%$\square$\end{remark}

\subsubsection*{Implications of Theorem~\ref{thm4} and choice of $N_{j}$} We now look at some of the implications of Theorem \ref{thm4} and (\ref{bound:thm4:final}) in the ideal setting where $k_0=0$. If $N_j$ is allowed to increase as $N_j= \lceil N+j^q \rceil$ for some $q>0$, then $\sum_{j=1}^k\left(1-\frac{\gamma}{\kappamod^2}\right)^{k-j}\frac{1}{N_j}\sim \frac{\kappamod^2}{\gamma}\frac{1}{N_k}$, as $k\to\infty$; then the implication of Theorem \ref{thm4} and (\ref{bound:thm4:final}) is that, as $k\to\infty$,
\begin{equation}\label{eq:rate1}
\PE\left[\normfro{\theta_{k}-\thetaen}^2\right] =O\left(\left(1-\frac{\gamma}{\kappamod_{\star\star}^2}\right)^k + \frac{\kappamod^2_{\star\star}}{\gamma N_k}\right)=O\left(\rho^k + \frac{\kappamod_{\star\star}^2}{\gamma k^q}\right),\end{equation}
with $\rho=1-\frac{\gamma}{\kappamod_{\star\star}^2}$. Notice that the best choice of the step-size is $\gamma = \epsilonmod_{\star\star}^2$. Setting $\gamma=\epsilonmod_{\star\star}^2$, it follows that the number of iterations to guarantee that the left-hand side of (\ref{eq:rate1}) is smaller than $\deltamod\in(0,1)$ is 
\begin{equation}\label{eq:n_star}
k_\star = \left(\frac{\kappamod_{\star\star}^2}{\epsilonmod_{\star\star}^2}\frac{1}{\deltamod}\right)^{\frac{1}{q}} \vee \frac{\log(\deltamod^{-1})}{\log(\rho^{-1})},\end{equation}
where $\rho=1-\frac{\epsilonmod_{\star\star}^2}{\kappamod_{\star\star}^2}$, and $a\vee b=\max(a,b)$. This implies that in choosing $N_j= \lceil N+j^q \rceil$, one should choose $q>0$ such that 
\begin{equation}\label{eq:cond:q}
\left(\frac{\kappamod^2_{\star\star}}{\epsilonmod_{\star\star}^2}\frac{1}{\deltamod}\right)^{\frac{1}{q}} = \frac{\log(\deltamod^{-1})}{\log(\rho^{-1})}=O\left( \frac{\kappamod^2_{\star\star}}{\epsilonmod_{\star\star}^2}\log(\deltamod^{-1})\right) = O\left(\textsf{cond}(\thetaen)^2\log(\deltamod^{-1})\right),\end{equation}
where $\textsf{cond}(\thetaen)=\lambda_{\textsf{max}}(\thetaen)/\lambda_{\textsf{min}}(\thetaen)$ is the condition number of $\thetaen$. Incidentally, (\ref{eq:cond:q}) shows that one should choose $q>1$, as also needed in Theorem \ref{thm3}.

%\textcolor{red}{RM: Yves, we probably need to summarize the above result in the form of a Corollary?}

%\begin{remark}
%We note that the sampling scheme in Algorithm~2 is effective as long as $N_{k}$ is substantially smaller than $p$, which is indeed the case as also argued in Section~\ref{sec:cost-chol-1}. If higher accuracy solutions are desired then the number of iterations $k$ required to attain such solutions 
%will be quite large --- $N_{k}$ may thus become comparable to $p$ --- in these cases, for large values of $k$ we advocate the use of direct matrix inversion techniques, instead of Monte Carlo sampling.
%$\square$\end{remark}

The results developed above  suggest the following  stochastic version of Algorithm~\ref{algo0}.
As above, let $\{\gamma_k,\;k\geq 0\}$ be a sequence of positive step-sizes decreasing to zero, and let  $\{N_k,\; k\geq 0\}$ be a sequence of Monte Carlo sample sizes. That is, $N_k$ is the number of Monte Carlo sample draws from $\pi_{\theta_k}$ at iteration $k$. Algorithm~\ref{algo1} is summarized below:

\medskip

%\begin{algorithm}\label{algo1}
%\caption{\textbf{(Stochastic-A)}}
\begin{algorithm}\label{algo1}~\\
Set $\textsf{r}=0$.
\begin{enumerate}
\item Choose $\theta_0\in\M_+$.
\item Given $\theta_k$, generate $z_{1:N_k}\stackrel{\text{i.i.d.}}{\sim} \pi_{\theta_k}$, i.e., the density of $\textbf{N}(0,\theta_k^{-1})$, and set 
\[\Sigma_{k+1}=\frac{1}{N_k}\sum_{j=1}^{N_k} z_jz_j'.\]
\item Compute
\[\theta_{k+1} =\Prox_{\gamma_\textsf{r}}\left(\theta_k-\gamma_\textsf{r} (S-\Sigma_{k+1});\alpha\right).\]
\item If $\lambda_{\textsf{min}}(\theta_{k+1})\leq 0$, then restart: set $k\leftarrow 0$, $\textsf{r}\leftarrow \textsf{r}+1$, and go back to (1). 
Otherwise, set $k\leftarrow k+1$ and go to (2).
\end{enumerate}
\end{algorithm}

%\begin{algorithm}[p]
%  \caption{Rank-Restricted Soft SVD}

\begin{remark}
As in Algorithm~\ref{algo0}, the actual implementation of Step 4 can be avoided. For instance if the simulation of the Gaussian random variables in Step 2 uses the Cholesky decomposition, it \emph{returns} 
the information whether $\lambda_{\textsf{min}}(\theta_{k+1})\leq 0$. In this case, we restart the algorithm from $\theta_0$ (or from $\theta_k$), and with a smaller step-size, and a larger Monte Carlo batch size.
$\square$\end{remark}

\subsection{Sampling via dense Cholesky decomposition}\label{sec:cost-chol-1}
The main computational cost  of  Algorithm~2 lies in generating multivariate Gaussian random variables. The standard scheme for simulating such random variables is to decompose the precision matrix $\theta$ as 
\begin{equation}\label{eq:chol}
\theta = R'R,\end{equation}
 for some nonsingular matrix $R\in\rset^{p\times p}$. Then a random sample from $\textbf{N}(0,\theta^{-1})$ is obtained by simulating $u\sim \textbf{N}(0,I_p)$ and returning $R^{-1}u$. The most common but remarkably effective 
 approach to achieve the above decomposition~\eqref{eq:chol}
  is via the Cholesky decomposition, which leads to $R$ being triangular. This approach entails a total cost of $O(p^2m + p^3/3)$ to generate a set of $m$ independent Gaussian random variables and computing the outer-product matrix, which 
  forms an approximation to $\theta^{-1}$.
  The term $p^3/3$ accounts for the cost of the Cholesky decomposition; and $p^2m$ accounts for doing $m$ back-solves $R^{-1}u_{i}$ for $m$ many standard Gaussian random vectors $u_{i}, i = 1, \ldots, m$; and subsequently computing 
  $\frac{1}{m} \sum_{i=1}^{m} (R^{-1}u_{i})(R^{-1}u_{i})'$ --- note that each back-solve $R^{-1}u_{i}$ can be performed with $O(p^2)$ cost since $R$ is triangular. This shows that an iteration of Algorithm~\ref{algo1}, implemented via Cholesky decomposition, is more cost-effective than an iteration of Algorithm~\ref{algo0}, if the number of Gaussian random samples generated in that iteration is less than $p$. Since $k_\star$  iterations (as defined in (\ref{eq:n_star})) are needed to reach the precision $\deltamod$, and $N_k = N+ k^q$ (we assume that $q$ is chosen as in (\ref{eq:cond:q})), we see that the number of samples per iteration of Algorithm~\ref{algo1} remains below $p$, if $p\geq \textsf{cond}(\thetaen)^2\deltamod^{-1}$. In this case the overall computational cost  of Algorithm~\ref{algo1}, to obtain a $\deltamod$-accurate solution is 
\[O\left(p^3\frac{\log(\deltamod^{-1})}{\log(\rho^{-1})}\right) = O\left(p^3\textsf{cond}(\thetaen)^2\log(\deltamod^{-1})\right).\]

We caution the reader that, on the surface, the above cost seems to be of the same order as that of the deterministic algorithm (Algorithm~\ref{algo0}), as seen from Theorem \ref{thm1}. 
However, the constants in the big-O notation differ, and are much better for the Cholesky decomposition than for inverting a matrix---see Figure~\ref{fig-inv-chol-evd-1} for a compelling illustration of this observation. In addition, as the problem sizes become 
much larger (i.e., larger than $p \approx 35,000$)  matrix inversions become much more memory intensive than Cholesky decompositions;  leading to prohibitely increased computation times---see Figure~\ref{fig-inv-chol-evd-1}.

\subsection{Sampling via specialized sparse numerical linear algebra methods}
As an alternative to the above approach, note that equation (\ref{eq:chol}) is also solved by $R=\theta^{1/2}$. If $\theta$ is sparse and very large, specialized numerical linear algebra methods can be used to compute $\theta^{-1/2}b$ for a vector or matrix $b$, with matching dimensions. These methods include Krylov space methods (\cite{Hale2008,eiermann:ernst:06}), or matrix function approximation methods (\cite{chen:etal:2011}).  These methods heavily exploit sparsity and typically scale better than the Cholesky decomposition when dealing with very large sparse problems. For instance, the matrix function approximation method of \cite{chen:etal:2011} has a computational cost of $O(m(p+C_p))$ to generate a set of $m$ samples from $\textbf{N}(0,\theta^{-1})$, where $C_p$ is the cost of performing a matrix-vector product $\theta b$ for some $b\in\rset^p$. 
%More details on this method can be found in Section~\ref{theta:b}.
As comparison, Figure \ref{Fig:GaussianComp} shows the time for generating $1,000$ random samples from $\textbf{N}(0,\theta^{-1})$, using dense Cholesky factorization, and using the matrix function approximation approach of (\cite{chen:etal:2011}), for varying values of $p$. The value of $p$ around which the matrix approximation method becomes better than the Cholesky decomposition depends on the sparsity of $\theta$, and the implementations of the methods. 
\begin{figure}[h!]
\centering
\scalebox{0.95}[.8]{\begin{tabular}{c c c}
%\rotatebox{90}{\sf {\small { \hspace{1cm} $(\phi_{\alpha}(\theta_k) - \hat{\phi}_{\alpha})/|\hat{\phi}_{\alpha}|$ }}}&\includegraphics[width=0.5\textwidth,height=0.3\textheight,  trim = 1cm 2.5cm 0cm 1.5cm, clip = true ]{\plots/times_chol_eig_inv1.pdf} &
& & (Zoomed) \\
\rotatebox{90}{ \hspace{1.5cm} \sf {\small Time (in seconds) }} & \includegraphics[width=0.5\textwidth,  height=0.3\textheight, trim = .6cm 1.2cm 0cm 1.5cm, clip = true ]{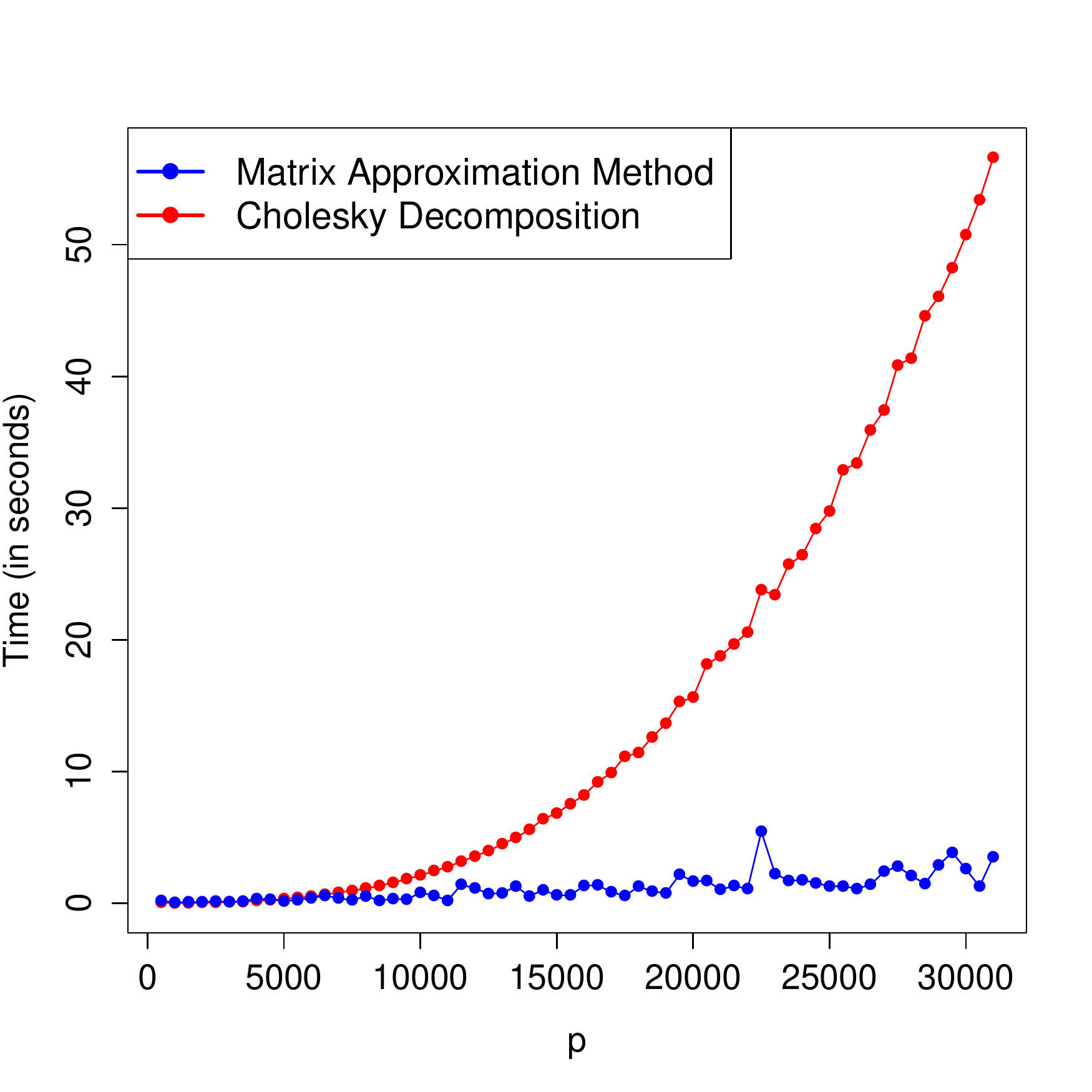} &
\includegraphics[width=0.5\textwidth, height=0.3\textheight, trim = .6cm 1.2cm 0cm 1.5cm, clip = true ]{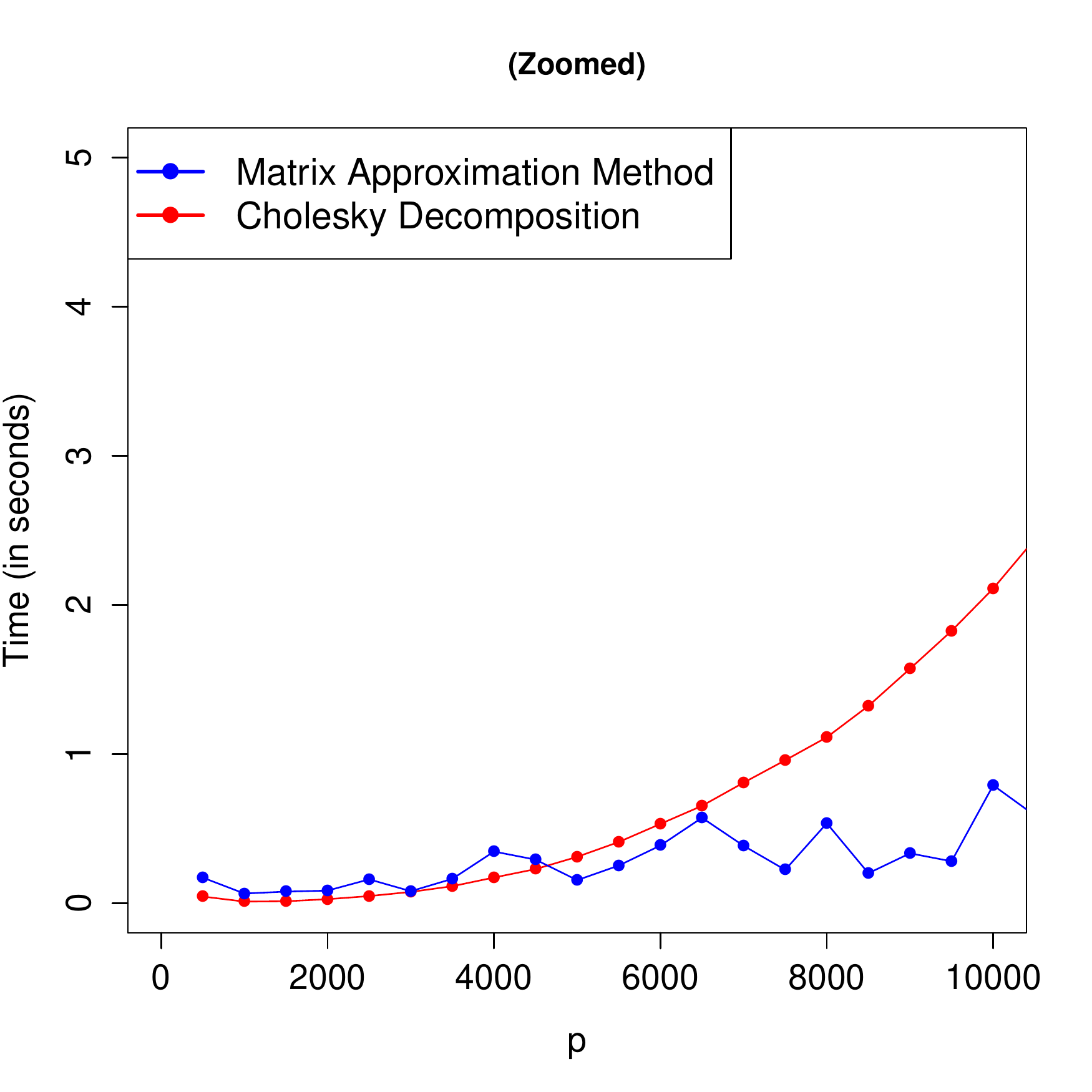} \\
 & \sf $p$   &  \sf $p$ \\
\end{tabular}}
\caption{{\small{Figure showing the times in seconds to generate $1,000$ Gaussian random samples from $\textbf{N}(0,\theta^{-1})$, where $\theta\in\rset^{p\times p}$ is constructed as explained in Section \ref{sec:simdata} with the proportion of non-zeros entries 
approximately set at $5/p$. 
%\textcolor{red}{Y: Rahul, can you redo this Figure and change the legend to "Matrix approximation methods" instead of "Krylov space method".
} }}\label{Fig:GaussianComp}
\end{figure}
These specialized sparse methods, however,  need to be used with caution.
For one thing, these methods are quite sensitive to the sparsity level of the iterates $\theta_k$, and ultimately to the sparsity level $\hat{\theta}$, the solution to Problem~\eqref{eq-glasso-enet-1} --- the methods are useful only when the solutions are sufficiently sparse. 
This behavior should be contrasted to that of dense Cholesky decomposition based methods, which are less sensitive to the sparsity level of $\hat{\theta}$. 
Based on our experiments (not reported here), we recommend the use of dense Cholesky decomposition methods in the initial stages of the algorithm, when the iterates $\theta_{k}$ are relatively dense.  As the number of iterations progresses and the estimates become more sparse, we recommend the use of specialized sparse numerical linear algebra methods for sampling from the Gaussian distributions.  
Since the use of dense Cholesky decomposition methods amply substantiates the main message of our paper---the effectiveness of 
stochastic gradient methods as a computationally scalable alternative to their deterministic counterparts, our experimental results reported in Section~\ref{sec:numerics} focus on dense numerical linear algebra methods.

\subsection{Borrowing information across iterations}\label{sec:refinement}
A main limitation of Algorithm~\ref{algo1} is that at each iteration $k$, all the Monte Carlo samples used to estimate $\theta_{k}^{-1}$ are discarded, and new samples are generated to approximate $\theta_{k+1}^{-1}$.
We thus ask, is there a modified algorithm that makes clever use of the information associated with an approximate 
$\theta_{k}^{-1}$ to approximate $\theta_{k+1}^{-1}$? In this vein, 
we propose herein a new stochastic algorithm: Algorithm~\ref{algo2} which recycles previously generated Monte Carlo samples in a novel fashion, to update its approximation for $\Sigma_{k+1}:=\theta^{-1}_{k+1}$ from $\Sigma_{k}:=\theta^{-1}_{k}$.
%with the obvious consequence that Algorithm~\ref{algo1} becomes very costly with the iterations.  A natural way to improve Algorithm~\ref{algo1} is to effectively use the information in
 %$\Sigma_k$ in order to compute $\Sigma_{k+1}$. 
 
This new algorithm relies on the following algorithm parameters  (a)  $N$, where $N\geq 1$ is a given integer, and (b) $\{\zeta_k,\;k\geq 1\}$ which  is a sequence of positive numbers such that 
\begin{equation}\label{cond:zeta}
\sum_{k\geq 1}\zeta_k =\infty,\;\;\mbox{ and }\;\;\; \sum_{k\geq 1}\zeta_k^2<\infty.\end{equation}
The algorithm is summarized below:
%\begin{algorithm}\label{algo2}
%\caption{({\bf{Stochastic-B}})}
\begin{algorithm}\label{algo2}
Set $\textsf{r}=0$.
\begin{enumerate}
\item Choose $\theta_0\in\M_+$, and $\Sigma_0\in\M_+$.
\item Given $\theta_k$, and $\Sigma_k$, generate $z_{1:N}\stackrel{\text{i.i.d.}}{\sim} \pi_{\theta_k}=\textbf{N}(0,\theta_k^{-1})$, and compute
\begin{equation}\label{algo2:rec1}
\Sigma_{k+1}=\Sigma_k +\zeta_{k+1} \left(\frac{1}{N}\sum_{k=1}^{N} z_kz_k' - \Sigma_k\right).\end{equation}
\item Compute
\begin{equation}\label{algo2:rec2}
\theta_{k+1} =\Prox_{\gamma_\textsf{r}}\left(\theta_k-\gamma_\textsf{r} (S-\Sigma_{k+1});\alpha\right).\end{equation}
\item If $\lambda_{\textsf{min}}(\theta_{k+1})\leq 0$, then restart: set $k\leftarrow 0$, $\textsf{r}\leftarrow \textsf{r}+1$, and go back to (1). Otherwise, set 
$k\leftarrow k+1$ and go to (2).
\end{enumerate}
\end{algorithm}

%\begin{remark}
Notice that in Algorithm~\ref{algo2}, the number of Monte Carlo samples is held fixed at $N$. Hence its cost per iteration is constant. 
%$\square$\end{remark}

Algorithm~\ref{algo2} is more difficult to analyze because the two recursive equations (\ref{algo2:rec1}) and (\ref{algo2:rec2}) are intimately coupled. However, the next result gives some theoretical guarantees by showing that when the sequence $\{\theta_k,\;k\geq 0\}$ converges, it necessarily converges to the minimizer of Problem~\eqref{eq-glasso-enet-1}, i.e., $\thetaen$. 

\begin{theorem}\label{thm5}
Let $\{\theta_k,\;k\geq 0\}$ be the stochastic process generated by Algorithm~\ref{algo2} where, the sequence $\{\zeta_k\}$ satisfies (\ref{cond:zeta}). Fix $0<\epsilonmod\leq \kappamod<\infty$.  Suppose that $\thetaen,\theta_0\in\M_+(\epsilonmod,\kappamod)$, and $\gamma\leq \epsilonmod^2$. 
Then, on the event 
\[ \left\{\tau(\epsilonmod,\kappamod)=+\infty,\;\mbox{ and }\; \{\theta_k\} \mbox{ converges} \right\},\]
we have that $\lim_{k \to\infty} \theta_k=\hat\theta$.
\end{theorem}
\begin{proof}
See Section \ref{sec:proofthm5}.
\end{proof}

\section{Exact Thresholding into connected components}\label{sec:thres}

As mentioned in Section~\ref{sec:intro}, the exact covariance thresholding rule~\citep{mazumder2012exact}, originally developed for the \Glasso problem plays a 
crucial role in the scalability of \Glasso to large values of $p$, for large values of $\lambda$. One simply requires that the largest 
connected component of the graph $((\mathbf{1}( |s_{ij} | > \lambda))),$ is of a size that can be handled by an algorithm for solving \Glasso of that size.~In this section, we extend this result to the more general case of Problem~\eqref{eq-glasso-enet-1}. 

Consider the symmetric binary matrix ${\mathcal E}:=((\mathcal E_{ij}))$ with ${\mathcal E}_{ij} = \mathbf{1}( |s_{ij} | > \alpha\lambda_{}),$ which defines a graph on the nodes ${\mathcal V} = \{ 1, \ldots, p \}$. 
Let $({\mathcal V}_{j}, {\mathcal E}_{j}), j = 1, \ldots, J$ denote the $J$ connected components of the graph $({\mathcal V}, {\mathcal E})$.
Let $\hat{\theta}$ be a minimizer of  Problem~\eqref{eq-glasso-enet-1} and consider the graph ${\mathcal{\widehat{E}}}$ induced by the sparsity pattern of $\hat{\theta}$, namely, 
${\mathcal{\widehat{E}}}_{ij} =  \mathbf{1}( |\hat{\theta}_{ij} | \neq 0)$. Let the connected components of $(\mathcal V, {\mathcal{\widehat{E}}})$ be denoted by 
$({\mathcal{\widehat{V}}}_{j}, {\mathcal{\widehat{E}}}_{j}), j = 1, \ldots, \widehat{J}$.
The following theorem states that these connected components are essentially the same.

\begin{theorem}\label{thm:conn-comp-1}
Let $({\mathcal V}_{j}, {\mathcal E}_{j}), j = 1, \ldots, J$ and $({\mathcal{\widehat{V}}}_{j}, {\mathcal{\widehat{E}}}_{j}), j = 1, \ldots, \widehat{J}$ 
denote the connected components, as defined above. 
 
Then, $J = \widehat{J}$ and there exists a permutation $\Pi$ on $\{ 1, \ldots, J\}$ such that
${\mathcal{\widehat{V}}}_{\Pi(j)} = {\mathcal {V} }_{j} $ and ${\mathcal{\widehat{E}}}_{\Pi(j)} = {\mathcal E}_{j}$ for all $j = 1, \ldots, J$.

\begin{proof}
See Appendix, Section~\ref{proof-conn-comp-1} for the proof.
\end{proof}
\end{theorem}
Note that the permutation $\Pi$ arises since the labelings of two connected component decompositions may be different.

Theorem~\ref{thm:conn-comp-1} is appealing because the connected components of the graph ${\mathcal E}_{ij} = \mathbf{1}( |s_{ij} | > \alpha\lambda_{})$ are fairly easy to compute even for massive sized graphs---see also~\cite{mazumder2012exact} for additional discussions
pertaining to similar observations for the \Glasso problem. A simple but powerful consequence of Theorem~\ref{thm:conn-comp-1} is that, once the connected components $({\mathcal V}_{j}, {\mathcal E}_{j}), j = 1, \ldots, J$ are obtained, 
Problem~\eqref{eq-glasso-enet-1} can be solved independently for each of the $J$ different connected component blocks. 
In concluding, we note that Theorem~\ref{thm:conn-comp-1} is useful if the maximum size of the connected components is small compared to $p$, which of course depends upon $S$ and $\lambda, \alpha$.

\section{Special Case: Ridge regularization}\label{sec:ridge-regu}
In this section, we focus our attention to a special instance of Problem~\eqref{eq-glasso-enet-1}, namely, the ridge regularized version, i.e., Problem~\eqref{eq-ridge}
%\begin{equation}\label{eq-ridge}
%\mini_{\theta \in\M_+} \;\;\; -\log\det\theta + \textsf{Tr}(\theta S)  +  \frac12\lambda_{2}  \| \theta \|_{F}^2,
%\end{equation}
for some value of $\lambda >0$. 
%In light of Problem~\eqref{eq-glasso-enet-1}, we set $\alpha=0$ to get Problem~\eqref{eq-ridge}.
Interestingly, the solution to this problem can be computed analytically as presented in the following lemma:
\begin{lemma}\label{closed-soon-ridge}
Let $S = U D U'$ denote the full eigendecomposition of $S$ where, $D = \diag(d_{1}, \ldots, d_{p})$.
For any $\lambda_{}>0$ and $\alpha = 0$ the solution to Problem~\eqref{eq-glasso-enet-1} is given by:
$ \hat{\theta} = U \diag(\widehat{\sigma}) U',$
where, $\diag(\widehat{\sigma})$ is a diagonal matrix with the $i$th diagonal entry given by
$$\widehat{\sigma}_{i} =  \frac{ - d_{i} + \sqrt{d_{i}^2 + 4 \lambda_{}}}{2 \lambda_{}}, \;\;\;  \text{for}\; \;\; i =1, \ldots, p.$$
\begin{proof}
For the proof see Section~\ref{proof-ridge-soln-1}
\end{proof}
\end{lemma}

We make the following remarks:
\begin{itemize}
\item Performing the eigen-decomposition of $S$ is clearly the most expensive part in computing a solution to Problem~\eqref{eq-ridge}; for a general 
real $p \times p$ symmetric matrix this has cost $O(p^3)$ and can be significantly more expensive than computing a direct matrix inverse or a Cholesky decomposition, as  reflected in Figure~\ref{fig-inv-chol-evd-1}.

\item When $p \gg n$ and $n$ is small, a minimizer for Problem~\eqref{eq-ridge} can be computed for large $p$, by  
observing that $S = \frac{1}{n} \sum_{i=1}^{n} \mathbf{x}_{i} \mathbf{x}_{i}' = \frac1n X'X$; thus the eigendecomposition of $S$ can be done efficiently via a SVD of the $n \times p$ 
rectangular matrix 
$X$ with $O(n^2p)$ cost, which reduces to 
$O(p)$ for values of $p \gg n$ with $n$ small.

\item However, computing the solution to Problem~\eqref{eq-ridge} becomes quite difficult when both $p$ and $n$ are large.
In this case, both our stochastic algorithms: Algorithms~\ref{algo1} and \ref{algo2}  are seen to be very useful to get an approximate solution within a fraction of the total computation time. Section~\ref{sec:num-dense-1} presents some numerical experiments. 

\end{itemize}

\section{Numerical experiments}\label{sec:numerics}
We performed some experiments to demonstrate the practical merit of our algorithm on some synthetic and real datasets.

\subsubsection*{Software Specifications} All our computations were performed in~{\textsc{Matlab}} (R2014a (8.3.0.532) 64-bit (maci64)) on a OS X 10.8.5 (12F45) operating system with a 3.4 GHz Intel Core i5 processor 
with 32 GB Ram,  processor speed 1600 MHz and DDR3 SDRAM.

\subsection{Studying sparse problems}
\subsubsection{Simulated data}\label{sec:simdata}
We test Algorithms~\ref{algo0}, \ref{algo1} and~\ref{algo2} with $p=10^3, 5\times 10^3$, and $p= 10^4$ for some synthetic examples.
The data matrix $S\in\rset^{p\times p}$ is generated as $S=n^{-1}\sum_{j=1}^n \mathbf{x}_j\mathbf{x}_j'$, where $n=p/2$, and  
$X_{1:n}\stackrel{\text{i.i.d.}}{\sim}\textbf{N}_p(0,\theta_\star^{-1})$, for a ``true" precision matrix  $\theta_\star$ generated as follows. 
First we generate a symmetric sparse matrix $B$ such that the proportion of non-zeros entries is $10/p$. We magnified the signal by adding $4$ to all the non-zeros entries of $B$ (subtracting $4$ for negative non-zero entries). Then we set $\theta_\star=B + (\epsilonmod-\lambda_{\min}(B))I_p$, where $\lambda_{\min}(B)$ is the smallest eigenvalue of $B$, with $\epsilonmod=1$. 

\medskip

Given $S$, we solve Problem (\ref{eq-glasso-enet-1}) with $\alpha \approx 0.9$ and $\lambda \propto \sqrt{\log(p)/n}$ such that the sparsity (i.e., the number of non-zeros) 
of the solution is roughly $10/p$. In all the examples, we ran the deterministic algorithm (Algorithm~1) for a large number of iterations (one thousand) with a step-size $\gamma=3.5$ to obtain 
a high-accuracy approximation of $\hat{\theta}$, the solution to Problem~\eqref{eq-glasso-enet-1} (we take this estimate as $\thetaen$ in what follows). Algorithms~1, 2 and 3  were then evaluated as how they progress towards the optimal solution $\hat{\theta}$ (recall that 
the optimization problem has a unique minimizer), as a function of time. All the algorithms were ran for a maximum of 300 iterations. Further details in setting up the solvers and parameter specifications are gathered in Section~\ref{comp-details-11} (appendix). 
To measure the quality of the solution, we used the following metric:
$$\text{Relative Error} = \|\theta_{k} - \hat{\theta} \|_{F}/\|\hat{\theta}\|_{F},$$ as a function of the number of iterations of the algorithms. 
Since the work done per iteration by the different algorithms are different, we monitored the progress of the algorithms as a function of time.
The results are shown in Figure~\ref{fig2:error}. 

We also compared the performance of our algorithms with the exact thresholding scheme (Section~\ref{sec:thres}) switched ``on'' --- this offered marginal improvements since the size
of the largest component was comparable to the size of the original matrix --- see Section~\ref{comp-details-11} for additional details on the sizes of 
the connected components produced. We also compared our method with a state-of-the algorithm: \quic~\citep{JMLR:v15:hsieh14a}, 
the only method that seemed to scale to \emph{all} the problem sizes that have been considered in our computational experiments. We used the {\texttt{R}} package
{\texttt{QUIC}}, downloaded from {\texttt{CRAN}} for our experiments.
The results are shown in Table~\ref{table-comp-quic-etc}.

We note that it is not fair to compare our methods versus \quic due to several reasons. 
Firstly the available implementation of \quic works for the \Glasso problem and the experiments we consider are for the generalized elastic net problem~\eqref{eq-glasso-enet-1}.
Furthermore, \quic is a fairly advanced implementation written in {\texttt{ C++}}, whereas our method is implemented entirely in \matlab.
In addition, the default convergence criterion used by \quic is different than what we use. 
However, we do report the computational times of \quic simply to give an idea of where we are in terms of the state-of-the art algorithms for \Glassoperiod
~Towards this end, we ran \quic for the \Glasso problem with $\lambda = \alpha \lambda$ for a large tolerance parameter (we took the native tolerance parameter, based on relative errors 
used in the algorithm \quic by setting 
its convergence threshold (tol) as $10^{-10}$), the solution thus obtained was denoted by $\hat{\theta}$.
We ran \quic for a sequence of twenty tolerance values of the form $0.5 \times 0.9^r$ for $r=1, \ldots, 20$; and then obtained the solution for which the relative error 
$\|\theta_{\text{r}} - \hat{\theta}\|_{F}/\|\hat{\theta}\|_{F} \leq \text{Tol} $ with $\text{Tol} \in \{ 0.1,  0.02\}$. 
For reference, the times taken by \quic to converge to its ``default'' convergence threshold (given by its relative error convergence threshold: tol$=10^{-4}$) were 501 seconds for for $p=5,000$  and 
3020 seconds for $p=10,000$.

\begin{figure}[]
\centering
\resizebox{\textwidth}{0.3\textheight}{\begin{tabular}{c c c c}
&\multicolumn{3}{c}{ Evolution of Relative Error of Algorithms~1---3 versus time \medskip}\\
&\small{\sf {$p=1000$}} & \small{\sf{ $p=5000$ } } & \small{ \sf{ $p=10,000$} } \\
\rotatebox{90}{\sf {\scriptsize{ \hspace{1cm} Relative error =$ \|\theta_{k} - \hat{\theta} \|_{F}/\|\hat{\theta}\|_{F} $ }}}&
\includegraphics[width=0.32\textwidth,height=0.3\textheight,  trim = 1cm .1cm 0cm 1.5cm, clip = true ]{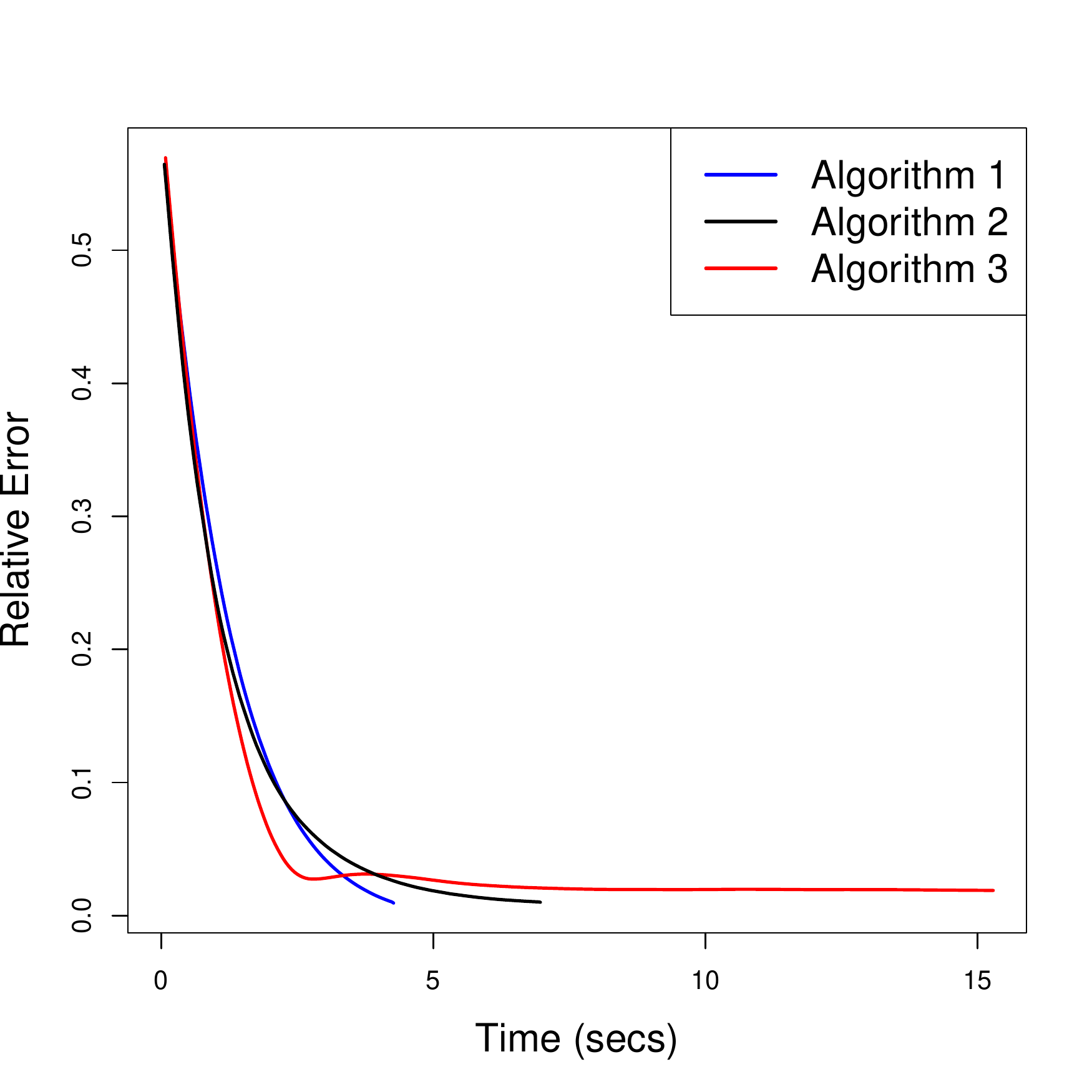}&
\includegraphics[width=0.32\textwidth,height=0.3\textheight ,  trim = 1cm .1cm 0cm 1.5cm, clip = true]{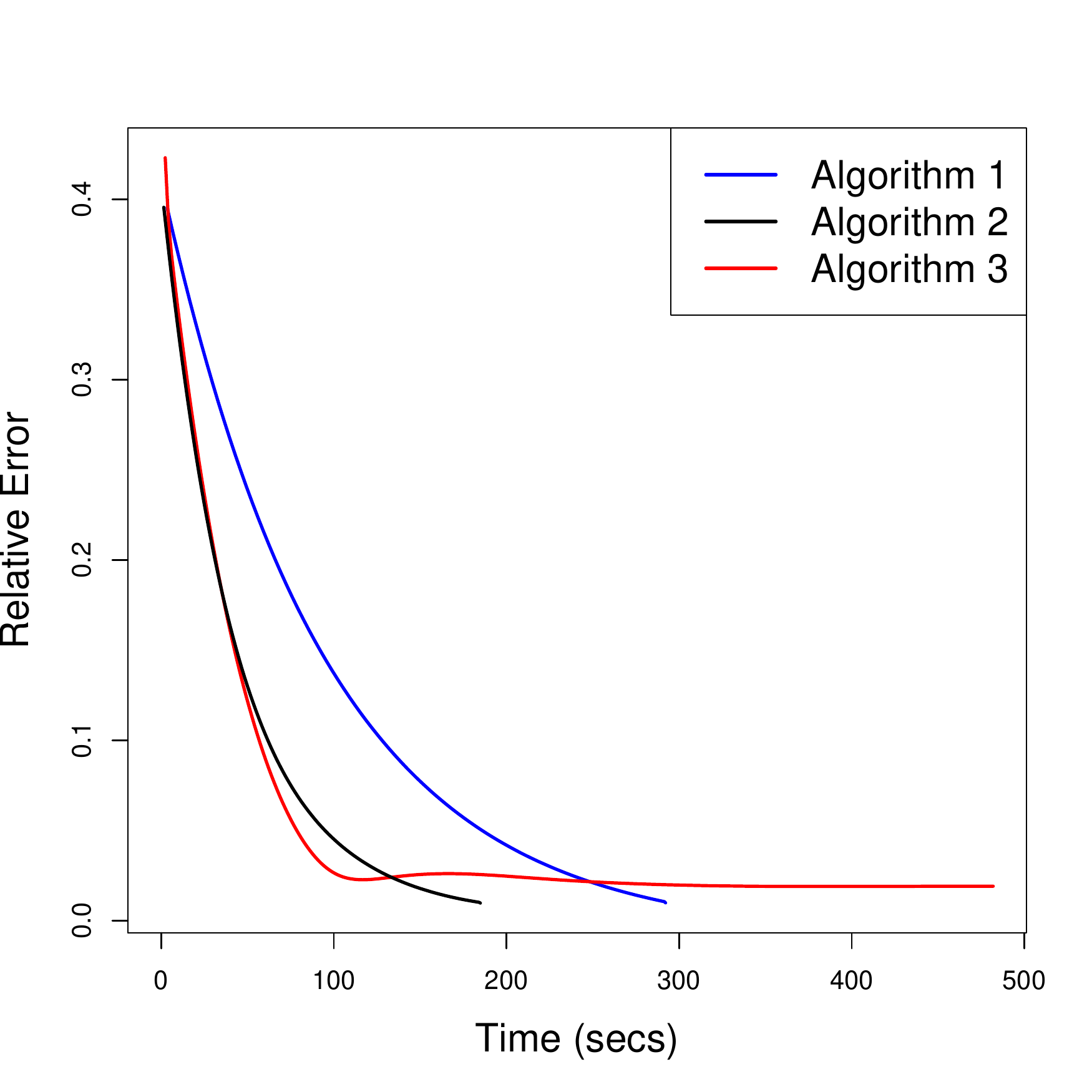}&
\includegraphics[width=0.32\textwidth,height=0.3\textheight,  trim = 1cm .1cm 0cm 1.5cm, clip = true]{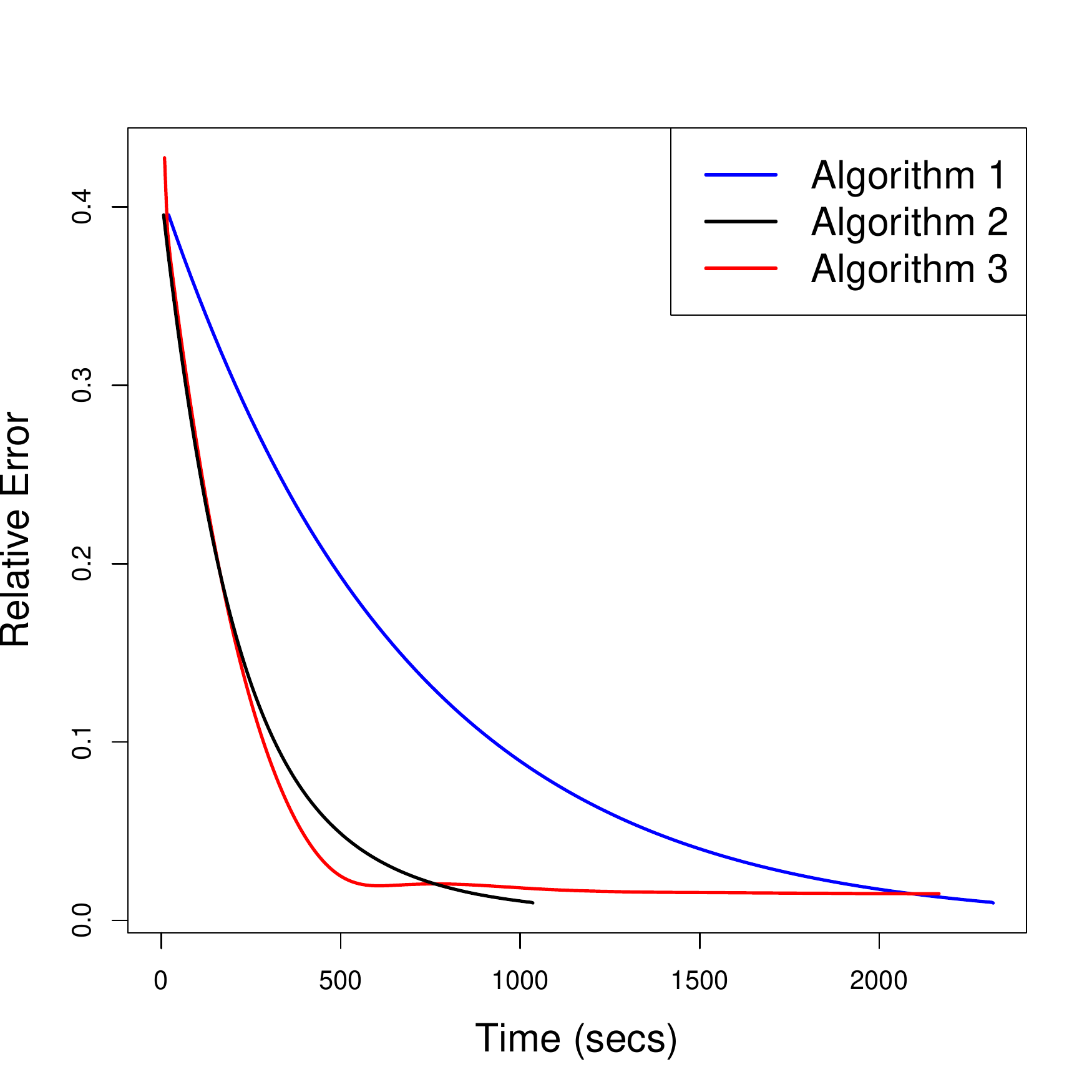} \\ \\
&\multicolumn{3}{c}{ Evolution of Sparsity of Algorithms~1---3 versus time \medskip}\\
&\small{\sf {$p=1000$}} & \small{\sf{ $p=5000$ } } & \small{ \sf{ $p=10,000$} } \\
\rotatebox{90}{\sf {\scriptsize{ \hspace{1cm} Sparsity = $\frac{1}{p}$ (\# non-zeros in $\theta_{k}$) }}}&
\includegraphics[width=0.32\textwidth,height=0.3\textheight,  trim = 1.cm .1cm 0cm 1.5cm, clip = true ]{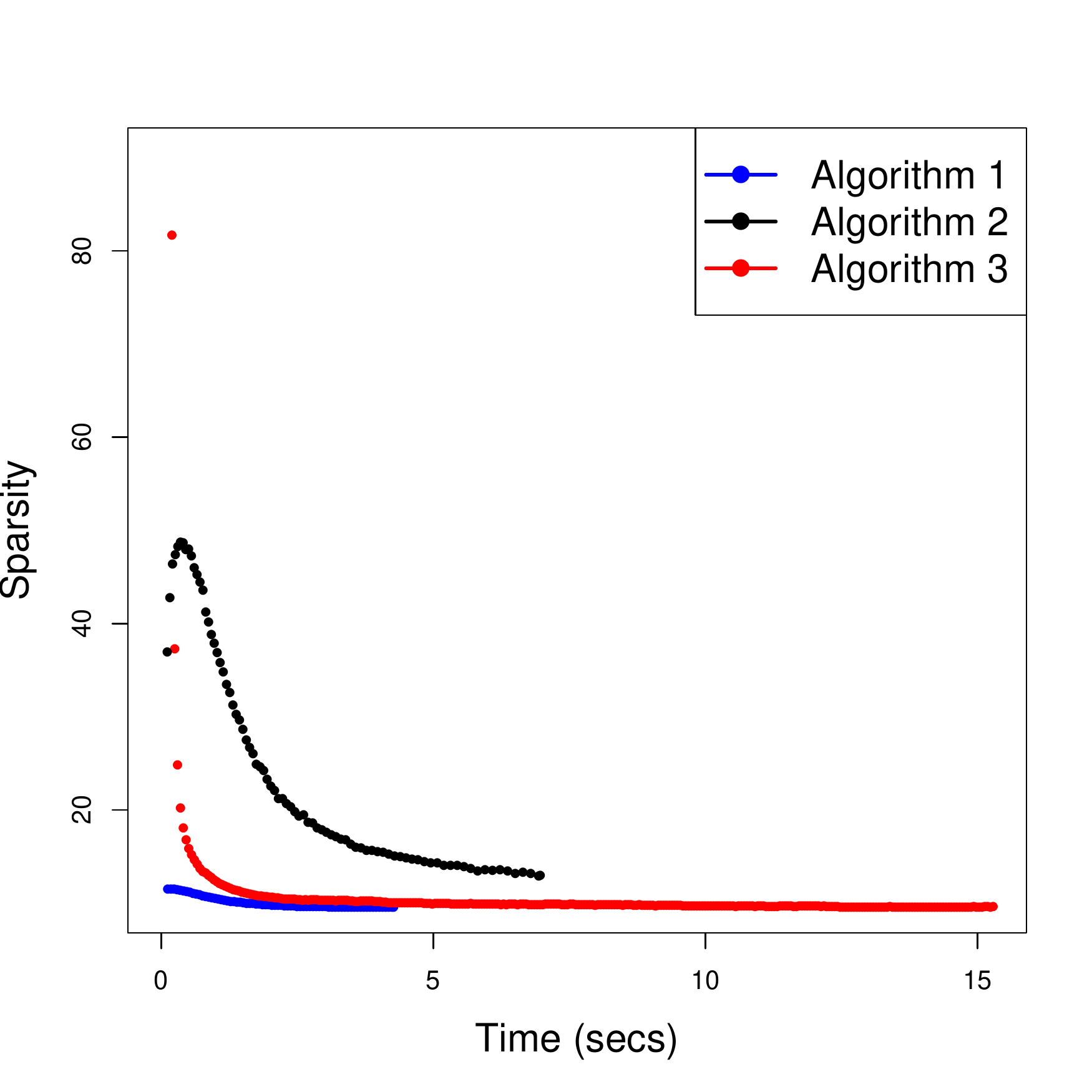}&
\includegraphics[width=0.32\textwidth,height=0.3\textheight ,  trim = 1cm .1cm 0cm 1.5cm, clip = true]{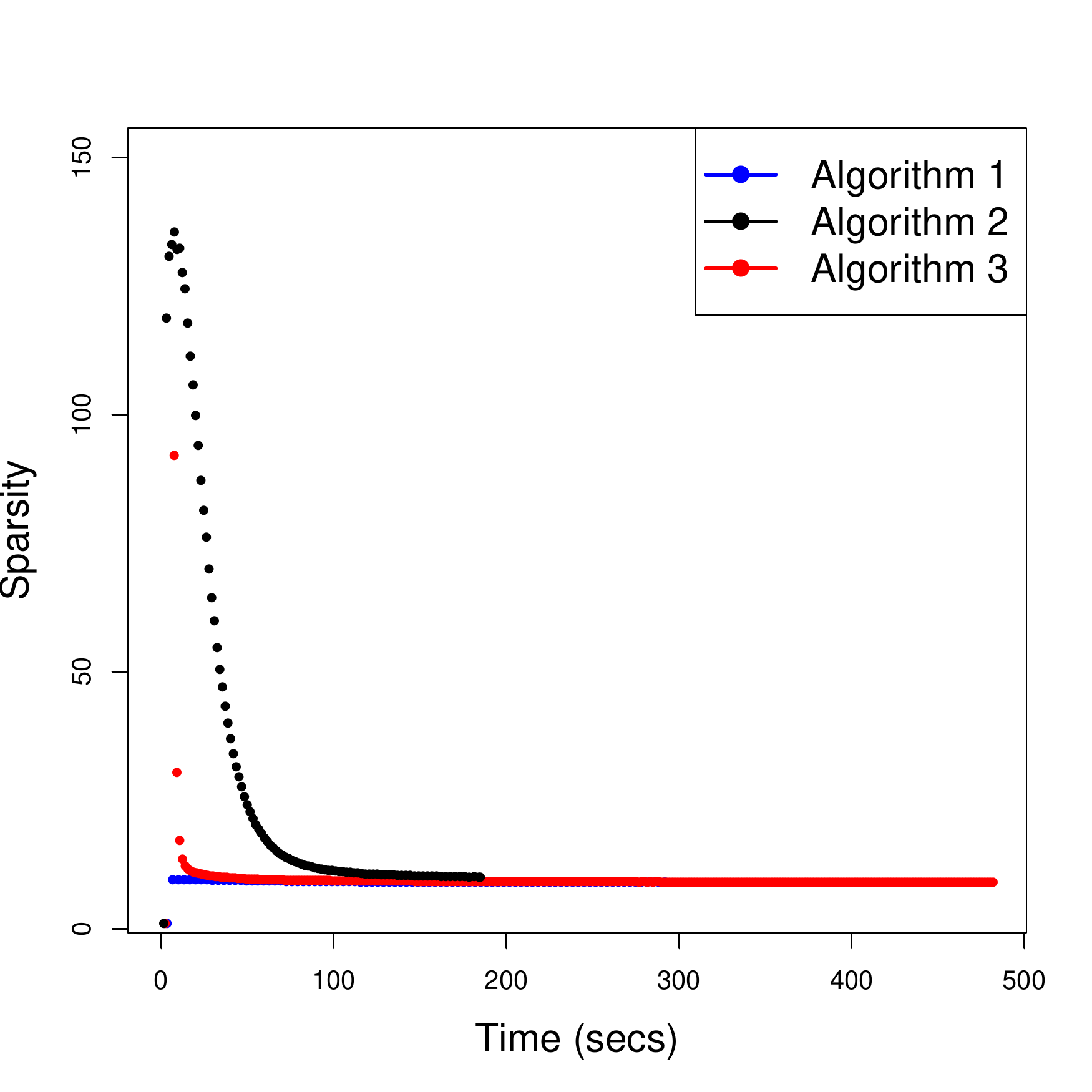}&
\includegraphics[width=0.32\textwidth,height=0.3\textheight,  trim = 1cm .1cm 0cm 1.5cm, clip = true]{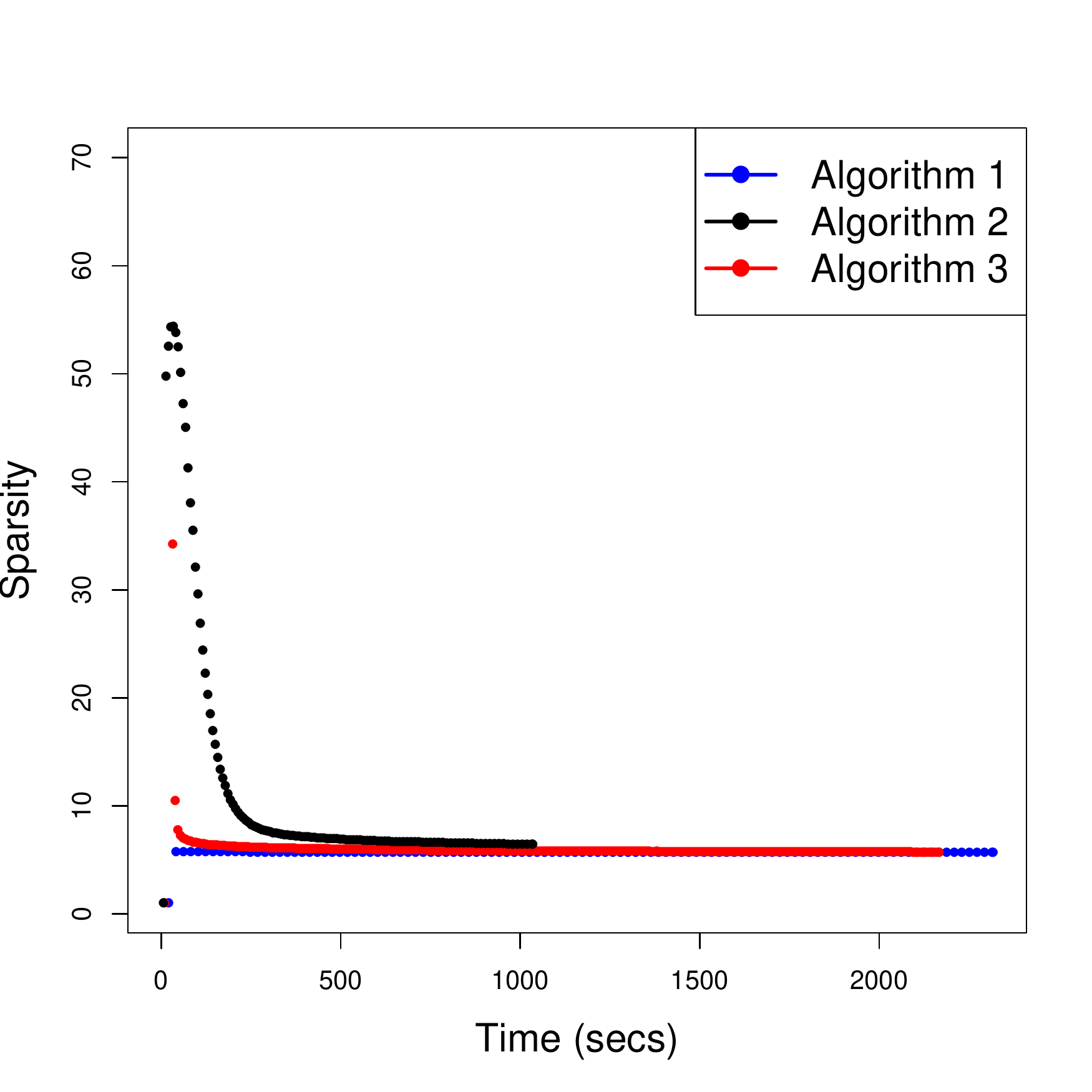} \\
\end{tabular}}
\caption{{\small{Evolution of relative error [top panel] and sparsity [bottom panel] of Algorithms~\ref{algo0}-\ref{algo2} versus time (in secs); for three different problem sizes: $p\in \{10^3, 5\times10^3, 10^4\}$ for the 
examples described in Section~\ref{sec:simdata}. We observe that for larger values of $p \geq 5\times10^3$, the new 
stochastic algorithms proposed in this paper: Algorithms~2 and 3 reach moderate accuracy solutions in times significantly smaller than the deterministic counterpart: Algorithm~1. Algorithm~3 reaches a low accuracy solution quicker, but 
is dominated by Algorithm~2 in obtaining a solution with higher accuracy. For small values of $p$ ($p=1000$) the different algorithms are comparable because direct matrix inversions are computationally less expensive, the situation changes quickly however, with larger values of $p$ (See also
Figure~\ref{fig-inv-chol-evd-1}). }} }\label{fig2:error}
\end{figure}

\subsubsection{Real dataset}
 The Patrick Brown dataset is an early example of an expression array, obtained from the Patrick Brown Laboratory
at Stanford University and was studied in~\cite{mazumder2012exact}. There are $n = 385$ patient samples of tissues from various regions of
the body (some from tumors, some not), with gene-expression measurements for $p = 4718$ genes.
For this example, the values of the regularization parameters were taken as
$(\alpha, \lambda) = (0.99,    0.16)$.
Here, splitting led to minor improvements since the size of the largest component was 4709, with all others having size one. We report the performance of our methods \emph{without} using the splitting method. 
We computed $\hat{\theta}$ by running the deterministic algorithm for 1000 iterations, using a step-size $\gamma=5\times 10^{-5}$. Unlike the synthetic experiments, in this case, we considered relative changes in objective values to determine the progress of the algorithm, namely,
$(\phi_{\alpha}(\theta_{k}) - \hat{\phi}_{\alpha})/ | \hat{\phi}_{\alpha}|,$ where, we define 
$\hat{\phi}_{\alpha} = \phi_{\alpha}(\hat{\theta})$; and recall that $\phi_{\alpha}(\cdot)$ is defined in~\eqref{defn-phi-obj}.

In this case, we also compared our method with \quic but the latter took a very long time in converging to even a moderate accuracy solution, so we took the solution delivered by its default mode as the reference solution $\hat{\theta}$. \quic took 5.9 hours to produce its default solution. Taking the objective value of this problem as the reference, we found that \quic took $6080.207$ and $10799.769$ secs to reach solutions with relative error $0.74$ and $0.50$ respectively.  We summarize the results in Table~\ref{tab:pat-brown}.

Our empirical findings confirm the theoretical results that for large $p$, the stochastic algorithms reach low-accuracy solutions much faster than the deterministic algorithms. We also see that the splitting rule helps, as it should --- major improvements are expected if the sizes of the connected
 components are significantly smaller than the original problem. 
  The sparsity plot (Figure \ref{fig2:error}) shows that the solution provided by Algorithm~\ref{algo1} tends to be noisy. The averaging step in estimating  $\theta_k^{-1}$ in Algorithm~\ref{algo2} makes these estimates much smoother, which results in solutions with good sparsity properties.

\subsection{Studying dense problems}\label{sec:num-dense-1}
We performed some experiments to demonstrate the performance of our method on dense inverse covariance estimation problems. 
Here, we took a sample of size $n= p$ with $p \in \{ 10^4 , 1.5 \times 10^4\}$, from a Gaussian density with independent covariates and mean zero. 
As described in Section~\ref{sec:ridge-regu}, it is indeed possible to obtain a closed form solution to this problem, but it requires performing a large scale eigen-decomposition 
on $S$, which can be quite expensive. In this application, proximal gradient algorithms and in particular the stochastic algorithms presented in this paper, become particularly useful.
They deliver approximate solutions to Problem~\ref{eq-ridge} in times that are orders of magnitude smaller than that taken to obtain an exact solution. 

In the experiments considered herein, we found the following scheme to be quite useful.
We took a subsample of size $m\ll n$ from the original $n$ samples and solved Problem~\eqref{eq-ridge} with a covariance matrix obtained from that subsample. 
This is indeed quite efficient since it requires computing the SVD of an $m \times p$ matrix, with $m \ll p$. We took the precision matrix and the covariance matrix associated with this 
subsample as a warm-start to the deterministic proximal gradient method, i.e., Algorithm~1 and Algorithm~2. 
This was seen to improve the overall run-time of the solution versus an initialization with a diagonal matrix.

We summarize our results in Table~\ref{tab:ridge-results}. 
For the case $p=10,000$ our Monte Carlo batch size was of $N_{k} = 1,000 + \lceil k^{1.4} \rceil$ and 
we took $\gamma = 0.1/\lambda^2_{\max}(S)$. The algorithms were warm-started with the solution of Problem~\eqref{eq-ridge} for a subsample of size $m=100$, which took $0.1$ seconds to compute.
For the case, $p=15,000$ we took $N_k= 2,000 + \lceil k^{1.4} \rceil$ and $\gamma$ as before.
 As a warm-start we took the solution of Problem~\eqref{eq-ridge} with a subsample of size $m=500$ which was obtained in 1 second.

\begin{table}[]
\begin{center}
%\small
\resizebox{.95\textwidth}{!}{\begin{tabular}{ c ccc  ccc c  } \hline
& \multicolumn{7}{c}{Time (in secs) taken by algorithms}\\
% Accuracy & \multicolumn{2}{c}{Algorithm~1}& \multicolumn{2}{c}{Stochastic-A} &  \multicolumn{2}{c}{Stochastic-B} & QUIC \\
 Accuracy & \multicolumn{2}{c}{\bf Algorithm~1}& \multicolumn{2}{c}{\bf Algorithm~2} &  \multicolumn{2}{c}{\bf Algorithm~3} & \textsc{ {Quic}} \\
Tol     &      No Splitting          & With Splitting            &   No Splitting          & With Splitting  &  No Splitting          & With Splitting & \\ \hline
 \multicolumn{8}{c}{$p=5000$}\\
 &&&&&&&\\
$10^{-1}$ & 125.78 & 122.16 & 62.73 & 60.84 & 56.78  & 53.31 &300.94 \\
$2\times10^{-2}$ &  251.92&241.49 &  161.18 &  142.83  &292.34    & 271.67  &350.29 \\ \hline  \\
 \multicolumn{8}{c}{$p=10,000$}\\
 &&&&&&&\\
$10^{-1}$ &  921.52  &612.11      &317.35 &155.33 & 289.58 &192.28 & 2046.73  \\
$2 \times 10^{-2}$& 1914.65  &1305.65 &766.69 & 463.66 &647.27  & 563.42  &2373.03 \\ \hline \\
\end{tabular}}
\caption{\small{Table showing the times (in secs) to reach an Accuracy of ``Tol" for different algorithms, where, Accuracy refers to 
$\|\theta_{k} - \hat{\theta}\|_{F}/\|\hat{\theta}\|_{F}$.
%Note: Stochastic-A, Stoch-B are Grad3 and Grad2 respectively. Here, relative error is defined as 
%$\|\hat{\theta} - \hat{\theta}\|_F/\| \hat{\theta}\|_F$.  
Algorithms~2 and 3 clearly shine over the deterministic method (Algorithm~1) for delivering moderate accuracy solutions.
Algorithm~3 reaches a solution of moderate accuracy faster than Algorithm~2 and Algorithm~1; for smaller values of ``Tol'' Algorithm~2 wins.
Here, splitting, which refers to the notion of covariance thresholding described in Section~\ref{sec:thres} is found to help, though not substantially --- the regularization parameters in this problem lead to connected components of sizes comparable to the original problem.
The timings of \quic are shown for reference purposes only, to get an idea of the times taken by state-of-the art algorithms.
For reference, the times taken by \quic to converge to its default convergence criteria were 501 seconds for $p=5,000$  and 
3020 seconds for $p=10,000$.}}\label{table-comp-quic-etc}
\end{center}
\end{table}

\begin{table}[h]
\centering
\scalebox{.9}{\begin{tabular}{c c c c}
  \hline
 Accuracy & \multicolumn{3}{c}{Time (in secs) taken by algorithms}\\
  Tol & \bf Algorithm~1 &\bf Algorithm~2 & \bf Algorithm~3  \\  \hline
 0.1 & 881.995  & 366.864 & 451.337 \\ 
 0.02 & 2030.405  & 942.924 & $>$ 654 \\  
   \hline \\
\end{tabular}}\caption{\small{Results on the Patrick Brown microarray dataset (here, $n = 385$ and $p = 4718$). Algorithm~3 reached a solution of relative accuracy $0.06$ within the first 500 iterations which took a total time of 654 seconds. 
Here, we use ``Accuracy'' to denote the relative error: $(\phi_{\alpha}(\theta_{k}) - \hat{\phi}_{\alpha})/ | \hat{\phi}_{\alpha}|,$ where, $\hat{\phi}_{\alpha}$ is the optimal objective value for the problem.
%In this particular example, $\hat{\theta}$ was estimated by running the deterministic algorithm apriori for 1000 iterations. 
For comparison, \quic for the same dataset when set to optimize the corresponding graphical lasso problem with the same tuning parameter, took 5.9 hours to converge to a solution with the native (default) tolerance criterion. Taking the objective value of this problem as the reference, we found that \quic took approximately, $6080$ secs ($\sim 1.7$ hours ) and $10800$ secs ($\sim 3$ hours) to reach solutions with relative errors $0.74$ 
and $0.50$ respectively. The algorithms presented in this paper show impressive performance for the particular tasks at hand. }}\label{tab:pat-brown}
\end{table}

\begin{table}[h]
\centering
\scalebox{.9}{\begin{tabular}{c c c c c}
  \hline
Accuracy  & & \multicolumn{2}{c}{Time (in secs) taken by algorithms}\\
  Tol & p & \bf Algorithm~1 & \bf Algorithm~2 \\  \hline
 0.1 & $10^4$ &  15.42 & 4.67  \\
 0.05 & $10^4$ &  93.46  & 48.490   \\ \\
 0.1 & $1.5 \times 10^4$ &  50.78 & 15.70\\
 0.05 & $1.5 \times 10^4$ & 408.87  & 176.11   \\
   \hline \\
\end{tabular}}
\caption{Results for ridge regression. Here, we use ``Accuracy'' to denote the measure: $(\phi_{\alpha}(\theta_{k}) - \hat{\phi}_{\alpha})/ | \hat{\phi}_{\alpha}|,$ where, $\hat{\phi}_{\alpha}$ is the optimal objective value for the problem.
For $p=10,000$ computing the exact solution (using a full eigen-decomposition) took 140 secs, for $p=15,000$ the exact solution was computed in 500 secs. Both Algorithms~1 and 2 obtained  
approximate solutions in times significantly smaller than computing the exact solution to the problem. For details see Section~\ref{sec:num-dense-1}.}\label{tab:ridge-results}
\end{table}

\section{Proofs}\label{sec:proofs}
This section gathers the proofs and technical details appearing in the paper.

\subsection{Proof of Lemma~\ref{lem:spec-bounds-1} }\label{proof-lem:spec-bounds-1}

\begin{proof}~\\

\noindent {\bf{Uniqueness of $\hat{\theta}$:}}

\smallskip

If $\lambda_{2} > 0$ then Problem~\eqref{eq-glasso-enet-1} is strongly convex due to the presence of the quadratic regularizer, hence $\hat{\theta}$ is unique. 
If $\lambda_{2} = 0$ and $\lambda_{1} >0$ then Problem~\eqref{eq-glasso-enet-1} becomes equivalent to~\Glasso for which uniqueness of $\hat{\theta}$ was established in~\cite{BGA2008,Lu:09}.

%\medskip
%\medskip

\noindent {\bf{Spectral bounds on $\hat{\theta}$:}}

\smallskip

Consider the stationary conditions of Problem~\eqref{eq-glasso-enet-1}:
\begin{equation}\label{line-pf-1}
-\hat{\theta}^{-1} + S + \lambda_{1} Z + 2\lambda_{2} \hat{\theta} = 0, 
\end{equation}
where, we use the notation: $Z = \sign(\hat{\theta})$, $\lambda_{1} = \alpha \lambda$ and $\lambda_{2} = (1-\alpha)\lambda/2$.
It follows from~\eqref{line-pf-1} that
\begin{equation}
\begin{aligned}\label{line-pf-2}
\hat{\theta}^{-1} - 2\lambda_{2}\hat{\theta} =& S + \lambda_{1} Z\\
 \leq&  \| S + \lambda_{1} Z \|_{2} \mm{I}  \\
 \leq&   \left(\| S\|_{2} + \lambda_{1} \|Z \|_{2}\right) \mm{I}  \\
 \leq&  \left(\| S\|_{2} + \lambda_{1} p\right) \mm{I} \;\;\;\; \text{(since, $z_{ij} \in [-1,1]$ implies $\|Z \|_{2} \leq p$)}
\end{aligned}
\end{equation}
If $\sigma_{i}$'s denote the eigen-values of $\hat{\theta}$ then it follows from~\eqref{line-pf-2}:
$$ 1/\sigma_{i} - 2\lambda_{2}\sigma_{i} \leq \| S\|_{2} + \lambda_{1} p = \mu. $$
Using elementary algebra, the above provides us a lower bound on all the eigen-values of  the optimal solution $\hat{\theta}$:
$\sigma_{i} \geq  (-\mu + \sqrt{\mu^2 + 8 \lambda_{2}})/(4 \lambda_{2})$ for $\lambda_{2} \neq 0$, for all $i = 1, \ldots, p$. 
Note that for the case, $\lambda_{2} =0$ we have $ \sigma_{i} \geq 1/\mu$ for all $i$. Combining these results we have the following:
\begin{equation*}
\lambda_{\min}(\hat{\theta}) \geq \ell_{\star} :=  \begin{cases} 
\frac{-\mu + \sqrt{\mu^2 + 8 \lambda_{2}}}{4 \lambda_{2}} &\text{if $\lambda_{2} \neq 0$}\\
\frac{1}{\mu} & \text{otherwise},
\end{cases}
\end{equation*}
which completes the proof of the lower bound on the spectrum of $\hat{\theta}$.

We now proceed towards deriving upper bound on the eigen-values of $\hat{\theta}$.

From~\eqref{line-pf-1} we have:
\begin{equation}\label{line-pf-3}
0 = \langle \hat{\theta},  -\hat{\theta}^{-1} + S + \lambda_{1} Z + 2\lambda_{2} \hat{\theta} \rangle \implies \lambda_{1} \| \hat{\theta} \|_{1} = p - \langle \hat{\theta}, S\rangle - 2\lambda_{2} \normfro{\hat{\theta}}^2 
\end{equation}
Now observe that: 
\begin{equation}\label{line-pf-301}
\langle \hat{\theta}, S \rangle \geq \lambda_{\min}(\hat{\theta}) \tr(S) \;\;\; \text{and} \;\;\; \normfro{\hat{\theta}}^2 \geq p \lambda^2_{\min}(\hat{\theta}).
\end{equation}
We use $\ell_{\star}$ as a lower bound for $\lambda_{\min}(\hat{\theta})$ and use~\eqref{line-pf-301} in~\eqref{line-pf-3} to arrive at:
\begin{equation}\label{line-pf-4}
 \|\hat{\theta}\|_{1} \leq \frac{1}{\lambda_{1}} \left(p - \ell_{\star} \tr(S) - 2p\lambda_{2} \ell_{\star}^2 \right) := U_{1}
\end{equation}
The above bound can be tightened by adapting the techniques appearing in~\cite{Lu:09} for the special case $\lambda_{2}=0$; as we discuss below.
Let $\hat{\theta}(t) := ( S + t \lambda_{1} \mm{I})^{-1}$ be a family of matrices defined on $t \in (0,1)$. 
It is easy to see that 
$$\hat{\theta}(t) \in \argmin_{{\theta}} \left \{ -\log\det({\theta})  + \langle S + t \lambda_{1} \mm{I}, {\theta} \rangle \right\},$$ which leads to 
\begin{equation}\label{line-pf-5}
\begin{aligned}
-\log\det(\hat{\theta}(t))  + \langle S + t \lambda_{1} \mm{I}, \hat{\theta}(t) \rangle &\leq& -\log\det(\hat{\theta})  + \langle S + t \lambda_{1} \mm{I}, \hat{\theta} \rangle&  \\
-\log\det(\hat{\theta})  + \langle S , \hat{\theta} \rangle  + \lambda_{1} \| \hat{\theta}\|_{1} + \lambda_{2} \normfro{\hat{\theta}}^2 &\leq& 
-\log\det(\hat{\theta}(t))  + \langle S , \hat{\theta}(t) \rangle  &\\
&&+ \lambda_{1} \| \hat{\theta}(t)\|_{1} + \lambda_{2} \normfro{\hat{\theta}(t)}^2,&
\end{aligned}
\end{equation}
where, the second inequality in~\eqref{line-pf-5} follows from the definition of $\hat{\theta}$.
Adding the two inequalities in~\eqref{line-pf-5} and doing some simplification, we have:
$$\lambda_{1} \| \hat{\theta}(t) \|_{1} - t \lambda_{1} \tr(\hat{\theta}(t)) + \lambda_{2} \normfro{\hat{\theta}(t)}^2 -\lambda_{2}\normfro{\hat{\theta}}^2 \geq \lambda_{1} \| \hat{\theta} \|_{1} - t\lambda_{1} \tr(\hat{\theta}) \geq (\lambda_{1} - t \lambda_{1}) \|\hat{\theta}\|_{1},$$
where, the rhs of the above inequality was obtained by using the simple observation $ \tr(\hat{\theta}) \leq \| \hat{\theta}\|_{1}$.
 Dividing both sides of the above inequality by $\lambda_{1} - t\lambda_{1}$ we have:
\begin{equation}\label{line-pf-6}
\| \hat{\theta}\|_{1} \leq \underbrace{\frac{1}{\lambda_{1}(1 - t)} \left(  \lambda_{1} \| \hat{\theta}(t) \|_{1} - t \lambda_{1} \tr(\hat{\theta}(t)) + \lambda_{2} \normfro{\hat{\theta}(t)}^2 \right)}_{:=a(t)} -  \underbrace{\frac{\lambda_{2}\normfro{\hat{\theta}}^2 }{\lambda_{1}(1 - t)}}_{:=b(t)}.
\end{equation}
Observing that $ \normfro{\hat{\theta}}^2 \geq \ell_{\star}^2p$ and applying it to~\eqref{line-pf-6} we obtain: 
\begin{equation} \label{line-pf-61}
\| \hat{\theta}\|_{1} \leq \left ( a(t) - \tilde{b}(t) \right),
\end{equation}
where, 
$a(t) = \frac{1}{\lambda_{1}(1 - t)} \left(  \lambda_{1} \| \hat{\theta}(t) \|_{1} - t \lambda_{1} \tr(\hat{\theta}(t)) + \lambda_{2} \normfro{\hat{\theta}(t)}^2 \right)$ and
$\tilde{b}(t):=\frac{\lambda_{2}  \ell_{\star}^2p }{\lambda_{1}(1 - t)}$.

Inequality~\eqref{line-pf-61} in particular implies:
\begin{equation} 
\| \hat{\theta} \|_{1} \leq \inf_{t \in (0,1)}  \left ( a(t) - \tilde{b}(t) \right)  := U_{2}
\label{line-pf-63}
\end{equation}
where, the minimization problem appearing above is a one dimensional optimization and can be approximated quite easily.
While a closed form solution to the minimization problem in~\eqref{line-pf-63} may not be available, 
$\| \hat{\theta} \|_{1}$ can be (upper) bounded by specific evaluations of $a(t) - \tilde{b}(t)$ at different values of $t \in (0,1)$.
In particular, note that if $S$ is invertible then, taking $t \approx 0+$ we get:
$$\| \hat{\theta} \|_{1} \leq \left(  \| S^{-1} \|_{1} + \frac{\lambda_{2}}{\lambda_{1}} \normfro{S^{-1}}^2 \right) - \frac{\lambda_{2} \ell^2_{LB}}{\lambda_{1}}, $$
otherwise, taking $t = \frac12$ leads to: $ \|\hat{\theta}\|_{1} \leq a(\frac12) - \tilde{b}(\frac{1}{2}).$
%\textcolor{red}{Rahul: please check the last equation}.

Combining~\eqref{line-pf-4} and~\eqref{line-pf-63}, we arrive at the following bound:
\begin{equation}\label{line-pf-7}
\| \hat{\theta} \|_{1} \leq \min \left \{   U_{1} ,  U_{2} \right \}
\end{equation}
Now observe that:
$$ \lambda_{\max}(\hat{\theta}) := \| \hat{\theta} \|_{2} \leq \normfro{\hat{\theta}} \leq \|\hat{\theta}\|_{1} \leq  \min \left \{   U_{1} ,  U_{2} \right \} := \psi_{UB}. $$ 
\end{proof}

\subsection{Proof of Lemma \ref{keylem}}\label{proof:keylem}
\begin{proof}

\noindent {\bf{First Part: Lower bound on $\lambda_{\min}(\theta_{j})$}} \\

Set $\bar\theta=\theta-\gamma(S-\theta^{-1})$. By definition, $T_\gamma(\theta;\alpha)=\argmin_{u\in\M}\left[g_\alpha(u)+\frac{1}{2\gamma}\normfro{u-\bar\theta}^2\right]$.
By the optimality condition of this optimization problem, there exists $Z\in\M$ in the sub-differential of the function $\theta\mapsto \|\theta\|_1$ at $T_\gamma(\theta;\alpha)$ such that $Z_{ij}\in [-1,1]$, $\pscal{Z}{T_\gamma(\theta;\alpha)}=\|T_\gamma(\theta;\alpha)\|_1$, and
\begin{equation}\label{optimality}
\frac{1}{\gamma}(T_\gamma(\theta;\alpha)-\bar\theta) +\alpha\lambda_{} Z +(1-\alpha)\lambda_{}T_\gamma(\theta;\alpha)=0.\end{equation}
The fact that $Z_{ij}\in [-1,1]$ implies that $\|Z\|_2\leq  \normfro{Z}\leq p$. Hence,
\begin{equation}\label{bound:Z}
\lambda_\textsf{min}(Z)\geq -p,\;\;\mbox{ and }\;\; \lambda_\textsf{max}(Z)\leq p.\end{equation}
Using $\bar\theta=\theta-\gamma(S-\theta^{-1})$, we expand (\ref{optimality}) to 
\[T_\gamma(\theta;\alpha)=\left(1+(1-\alpha)\lambda_{}\gamma\right)^{-1}\left(\theta-\gamma(S-\theta^{-1} +\alpha\lambda_{} Z)\right).\]
We write $\theta-\gamma(S-\theta^{-1} +\alpha\lambda_{} Z)=\theta+\gamma \theta^{-1} -\gamma(S+\alpha\lambda_{} Z)$. We will use the fact that for any symmetric matrices $A,B$, $\lambda_\textsf{min}(A+B)\geq \lambda_\textsf{min}(A) +\lambda_\textsf{min}(B)$,  $\lambda_\textsf{max}(A+B)\leq \lambda_\textsf{max}(A) +\lambda_\textsf{max}(B)$ (see e.g. \cite{golub:vl}~Theorem 8.1.5). In view of (\ref{bound:Z}) we have:
\begin{equation}\label{pf-yves-11}
\lambda_{\textsf{min}}\left(\theta-\gamma(S-\theta^{-1} +\alpha\lambda_{} Z)\right)\geq \lambda_{\textsf{min}}(\theta+\gamma\theta^{-1}) -\gamma\left(\lambda_{\textsf{max}}(S)+\alpha\lambda_{} p\right).
\end{equation}

Notice that the function $x\mapsto x+\frac{\gamma}{x}$ is increasing on $[\sqrt{\gamma},\infty)$, and by assumption $\epsilonmod_\star\geq \sqrt{\gamma}$. Therefore, if 
%$\theta\in\M_+(\epsilonmod_\star,\kappamod_\star)$, 
$\lambda_{\min}(\theta) \geq \ell_{\star}$,
we use the eigen-decomposition of $\theta$ to conclude that 
\begin{equation}\label{pf-yves-21}
\lambda_{\textsf{min}}\left(\theta +\gamma\theta^{-1}\right) =\lambda_{\textsf{min}}(\theta) +\frac{\lambda}{\lambda_{\textsf{min}}(\theta)}\geq \epsilonmod_\star+\frac{\gamma}{\epsilonmod_\star}.
\end{equation}
 Hence
 \begin{equation}\label{pf-yves-22}
 \lambda_{\textsf{min}}\left(T_\gamma(\theta;\alpha)\right) \geq \left(1+(1-\alpha)\lambda_{}\gamma\right)^{-1}\left[\epsilonmod_\star +\frac{\gamma}{\epsilonmod_\star}-\gamma(\lambda_{\textsf{max}}(S)+\alpha\lambda_{}p)\right]= \epsilonmod_\star,
 \end{equation}
 where the last equality uses the fact that $\ell_\star$ satisfies the equation
\[ (1-\alpha)\lambda_{}\epsilonmod_\star^2  + (\lambda_{\textsf{max}}(S) + \alpha\lambda_{} p)\epsilonmod_\star -1=0.\]

\medskip
\medskip

\noindent {\bf{Second Part: Upper bound on $\lambda_{\max}(\theta_{j})$}} \\

We will first show that if $\psi_\star^1\leq \psi_{UB}$,  then $\lambda_{\max}(\theta_{j}) \leq  \kappamod^1_{\star}$ for all $j \geq 1$. Following arguments similar to that used to arrive at~\eqref{pf-yves-11}, we have:
\begin{equation}\label{pf-yves-ub-11}
\lambda_{\textsf{max}}\left(\theta-\gamma(S-\theta^{-1} +\alpha\lambda_{} Z)\right)\leq \lambda_{\textsf{max}}(\theta+\gamma\theta^{-1}) -\gamma\left(\lambda_{\textsf{min}}(S) - \alpha\lambda_{} p\right).
\end{equation}

Using $\lambda_{\max}(\theta) \leq \kappamod^1_\star$; and
following arguments used to arrive at~\eqref{pf-yves-21},~\eqref{pf-yves-22} we have:
\[\lambda_{\textsf{max}}\left(\theta +\gamma\theta^{-1}\right) =\lambda_{\textsf{max}}(\theta) +\frac{\lambda}{\lambda_{\textsf{max}}(\theta)}\leq \kappamod_\star^1+\frac{\gamma}{\kappamod_\star^1}.\]
 Hence
 \[\lambda_{\textsf{max}}\left(T_\gamma(\theta;\alpha)\right) \leq \left(1+(1-\alpha)\lambda_{}\gamma\right)^{-1}\left[\kappamod_\star^1 +\frac{\gamma}{\kappamod_\star^1} - \gamma(\lambda_{\textsf{min}}(S) - \alpha\lambda_{}p)\right]= \kappamod_\star^1,\]
 where the last equality uses the fact that when $\psi_\star^1<\infty$, it satisfies the equation
\[(1-\alpha)\lambda_{} \kappamod_\star^2 +\left(\lambda_{\textsf{min}}(S)-\alpha\lambda_{} p\right)\kappamod_\star -1=0.\]

We now consider the case where, $\kappamod^1_{\star}>\psi_{UB}$, and $\theta_0\in\M_+(\ell_\star,\psi_{UB})$. The first part of the proof guarantees that $\theta_j\in\M_+(\ell_\star,+\infty)$ for all $j\geq 0$. For $j\geq 1$, by Lemma \ref{lem4} applied with $\ell=\ell_\star$, $\psi=+\infty$, $\theta=\theta_{j-1}$, and $H=\theta_{j-1}^{-1}$, we get
\[\normfro{\theta_j-\thetaen}\leq \normfro{\theta_{j-1}-\thetaen}.\]
This implies that for any $j\geq 1$,
\begin{eqnarray*} \|\theta_j\|_2 &\leq & \|\thetaen\|_2 +\|\theta_j-\thetaen\|_2\\
&\leq & \|\thetaen\|_2 + \normfro{\theta_0-\thetaen}\\
&\leq & \psi_{UB} + \sqrt{p}(\psi_{UB} -\ell_\star).\end{eqnarray*}

where the last inequality uses Weyl's inequality since $\theta_0,\thetaen\in\M_+(\ell_\star,\psi_{UB})$.

\end{proof}

\subsection{Proof of Theorem \ref{thm1}}\label{proof:thm1}
\begin{proof}
We follow closely the proof of Theorem 3.1. of \cite{fista-09}. Suppose that the sequence $\{\theta_i,\;0\leq i\leq k\}$ belongs to $\M_+(\epsilonmod,\kappamod)$. For any $i\geq 0$, since $\theta_{i+1} =\Prox_\gamma(\theta_i-\gamma(S-\theta_i^{-1});\alpha)$, we apply Lemma \ref{lem4} with $H=\theta_i^{-1}$ to obtain
\[\normfro{\theta_{i+1}-\thetaen}^2\leq 2\gamma \left(\phi_\alpha(\theta_{i+1})-\phi_\alpha(\thetaen)\right)  + \normfro{\theta_{i+1}-\thetaen}^2\leq \left(1-\frac{\gamma}{\kappamod^2}\right)\normfro{\theta_i-\thetaen}^2,\]
which implies that 
\begin{equation}\label{thm1:eq6}
2\gamma \left(\phi_\alpha(\theta_{k})-\phi_\alpha(\thetaen)\right) + \normfro{\theta_{k}-\thetaen}^2\leq \left(1-\frac{\gamma}{\kappamod^2}\right)^k\normfro{\theta_0-\thetaen}^2.\end{equation}
Again, from  (\ref{thm1:eq5}), we have
\[\phi_\alpha(\theta_{i+1})-\phi_\alpha(\thetaen)\leq  \frac{1}{2\gamma}\left[\normfro{\theta_i-\thetaen}^2 -\normfro{\theta_{i+1}-\thetaen}^2\right].\]
We then sum for $i=0$ to $k-1$ to obtain
\begin{equation}\label{prop1_eq5}
2\gamma\sum_{i=1}^k \left\{ \phi_\alpha(\theta_{i})-\phi_\alpha(\thetaen)\right\}  + \normfro{\theta_k-\hat\theta_0}^2 \leq \normfro{\theta_0-\thetaen}^2.\end{equation}
We now use  Lemma \ref{lem3} to write
\begin{eqnarray*}
g_\alpha(\theta_{i+1}) -g_\alpha(\theta_i) &\leq & \frac{1}{\gamma}\pscal{\theta_i-\theta_{i+1}}{\theta_{i+1}-\left(\theta_i-\gamma(S-\theta_i^{-1})\right)},\\
& = &-\frac{1}{\gamma}\normfro{\theta_{i+1}-\theta_i}^2 + \pscal{\theta_i-\theta_{i+1}}{S-\theta_i^{-1}}.
\end{eqnarray*}
This last inequality together with (\ref{prop1_eq1}) applied with $\bar\theta=\theta_{i+1}$ and $\theta=\theta_i$, yields
\begin{equation}\label{prop1_eq6}
\left\{\phi_\alpha(\theta_{i+1})-\phi_\alpha(\thetaen)\right\} \leq \left\{\phi_\alpha(\theta_{i})-\phi_\alpha(\thetaen) \right\}-\frac{1}{2\gamma}\normfro{\theta_i-\theta_{i+1}}^2.\end{equation}
By multiplying both sides of the last inequality by $i$ and summing from $0$ to $k-1$, we obtain
\begin{eqnarray*}
k\left\{\phi_\alpha(\theta_{k})-\phi_\alpha(\thetaen) \right\} &\leq &\sum_{i=1}^{k}\left\{\phi_\alpha(\theta_{i})-\phi_\alpha(\thetaen) \right\} -\frac{1}{2}\sum_{i=0}^{k-1}\frac{i}{\gamma}\normfro{\theta_i-\theta_{i+1}}^2\\
& \leq &\sum_{i=1}^{k} \left\{\phi_\alpha(\theta_{i})-\phi_\alpha(\thetaen) \right\}.\end{eqnarray*}
Hence, given (\ref{prop1_eq5}), we have
\[\left\{\phi_\alpha(\theta_{k})-\phi_\alpha(\thetaen) \right\}\leq \frac{1}{2\gamma k}\normfro{\theta_0-\thetaen}^2,\]
which together with (\ref{thm1:eq6}) yields the stated bound.
\end{proof}

\subsection{Proof of Theorem \ref{thm3}}\label{proof:thm3}
\begin{proof}
Write $\tau_\epsilon = \tau(\epsilonmod_\star(\epsilon),\kappamod^1_\star(\epsilon))$.
\[\PP\left[\tau_\epsilon=\infty\right]=1-\sum_{j=1}^\infty \PP\left[\tau_\epsilon=j\right],\]
and
\[\PP\left[\tau_\epsilon=j\right]=\PP\left[\left(\lambda_{\textsf{min}}(\theta_{j})<  \epsilonmod_\star(\epsilon) \mbox{ or } \lambda_{\textsf{max}}(\theta_j)>\kappamod^1_\star(\epsilon)\right),\;\tau_\epsilon>j-1\right].\]
Now we proceed as in the proof of Lemma \ref{keylem}. Given $\theta_{j-1}$, the optimality condition (\ref{optimality}) becomes: there exists a matrix $\Delta_j$, all entries of which belong to $[-1,1]$ (that can be taken as $\textsf{sign}(\theta_j)$), such that 
\[\theta_j=\left(1+(1-\alpha)\lambda_{}\gamma\right)^{-1}\left(\theta_{j-1} +\gamma\theta_{j-1}^{-1} -\gamma(S+(\theta_{j-1}^{-1}-G_j) + \alpha \lambda_{} \Delta_j)\right).\]
As in the proof of Lemma \ref{keylem} we have, 
\begin{multline*}
\lambda_{\textsf{max}}(S+ (\theta_{j-1}^{-1}-G_j) + \lambda \Delta_j) \leq \lambda_{\textsf{max}}(S) + p\|\theta_{j-1}^{-1}-G_j\|_\infty +p\lambda_{},\\
\mbox{ and }\;\;\;\; \lambda_{\textsf{min}}(S+ (\theta_{j-1}^{-1}-G_j) + \lambda \Delta_j) \geq \lambda_{\textsf{min}}(S) - p\|\theta_{j-1}^{-1}-G_j\|_\infty -p\lambda_{}.\end{multline*}
 where for $A\in \M$, $\|A\|_\infty\eqdef \max_{i,j}|A_{ij}|$. Therefore, with the same steps as in the proof of Lemma \ref{keylem}, we see that on the event $\{\tau>j-1, \|G_j-\theta_{j-1}^{-1}\|_\infty\leq \epsilon\}$, $\lambda_{\textsf{min}}(\theta_j)\geq \epsilonmod_\star(\epsilon)$, and $\lambda_{\textsf{max}}(\theta_j)\leq \kappamod^1_\star(\epsilon)$. We conclude that,
\[\PP\left[\tau_\epsilon=j\right] \leq \PP\left[\tau_\epsilon=j\vert \tau_\epsilon>j-1\right]\leq \PP\left[\|G_j - \theta_{j-1}^{-1}\|_\infty> \epsilon\vert \tau_\epsilon>j-1\right].
\]
We prove in Lemma \ref{lem:exp:bound} the exponential bound
\[\PP\left[\|G_j - \theta_{j-1}^{-1}\|_\infty> \epsilon \vert \lambda_{\textsf{min}}(\theta_{j-1})\geq \epsilonmod_\star(\epsilon)\right]\leq 8p^2\exp\left(-\min(1,\epsilonmod^2_\star(\epsilon)\epsilon^2/16) N_{j-1}\right).\]
Hence
\begin{equation}\label{lem:sto:eq1}
\PP\left[\tau_\epsilon=\infty\right] \geq 1- 8p^2 \sum_{j\geq 1} \exp\left(-\min(1,\epsilonmod^2_\star(\epsilon)\epsilon^2/16) N_{j-1}\right).\end{equation}
We will now show that there exists a random variable $\Psi_\star(\epsilon)$ such that on $\{\tau_\epsilon = +\infty\}$, $\lambda_{\textsf{max}}(\theta_j)\leq \Psi_\star(\epsilon)$ for all $j\geq 0$. 

We first note that on $\{\tau_\epsilon>k\}$, $\theta_0,\ldots,\theta_k\in\M_+(\epsilonmod,\kappamod)$, and $\theta_{j} = \Prox_\gamma\left(\theta_{j-1}-\gamma(S-\Sigma_{j};\alpha)\right)$ for $j=1,\ldots,k$. We then apply Lemma \ref{lem4} with $\theta=\theta_{j-1}$, $\bar\theta=\theta_j$ and $H=\Sigma_j$, to write 
\begin{multline*}
\normfro{\theta_{j}-\thetaen}^2\leq \normfro{\theta_{j}-\thetaen}^2 +2\gamma\left\{\phi_\alpha(\theta_j)-\phi_\alpha(\hat\theta)\right\} \\
\leq \left(1-\frac{\gamma}{\kappamod^2}\right)\normfro{\theta_{j-1}-\thetaen}^2 - 2\gamma\pscal{\thetaen-\theta_j}{\Sigma_j-\theta_{j-1}^{-1}}.\end{multline*}
We multiply both sides by $\textbf{1}_{\{\tau_\epsilon>j-1\}}$ and uses the fact that $\textbf{1}_{\{\tau_\epsilon>j-1\}}=\textbf{1}_{\{\tau_\epsilon=j\}} + \textbf{1}_{\{\tau_\epsilon>j\}}$ to write
\begin{multline}\label{thm4:eq1}
\textbf{1}_{\{\tau_\epsilon>j\}}\normfro{\theta_{j}-\thetaen}^2  \leq \left(1-\frac{\gamma}{\kappamod^2}\right)\textbf{1}_{\{\tau_\epsilon>j-1\}}\normfro{\theta_{j-1}-\thetaen}^2 \\
- 2\gamma\textbf{1}_{\{\tau_\epsilon>j-1\}}\pscal{\thetaen-\theta_j}{\Sigma_j-\theta_{j-1}^{-1}}.
\end{multline} 
Recall that $\theta_j=\Prox_{\gamma}\left(\theta_{j-1}-\gamma (S-\Sigma_j);\alpha\right)$, and split $\thetaen-\theta_j$ as 
\begin{equation}\label{thm4:split:eq}
\thetaen-\theta_j = \thetaen-T_{\gamma}(\theta_{j-1};\alpha) +T_{\gamma}(\theta_{j-1};\alpha) -\theta_j,\end{equation}
where $T_{\gamma}(\theta_{j-1};\alpha)=\Prox_{\gamma}\left( \theta_{j-1}-\gamma(S-\theta_{j-1}^{-1});\alpha\right)$.  It is well known that the proximal operator is non-expansive---see \cite[Propositions~12.26 and ~12.27]{bauschke:combettes:2011}. Hence
\begin{multline*}
\left|\pscal{T_{\gamma}(\theta_{j-1};\alpha) -\theta_j}{\Sigma_j-\theta_{j-1}^{-1}}\right|\leq \normfro{T_{\gamma}(\theta_{j-1};\alpha) -\theta_j}\normfro{\Sigma_j-\theta_{j-1}^{-1}}\\
\leq \gamma\normfro{\Sigma_j-\theta_{j-1}^{-1}}^2.\end{multline*}
We then set $V_j \eqdef \textbf{1}_{\{\tau_\epsilon>j-1\}} \pscal{\thetaen-T_{\gamma}(\theta_{j-1};\alpha)}{\Sigma_j-\theta_{j-1}^{-1}}$, and use the last inequality, (\ref{thm4:split:eq}), and (\ref{thm4:eq1}) to deduce that
\begin{multline}\label{thm4:eq2}
\textbf{1}_{\{\tau_\epsilon>j\}}\normfro{\theta_{j}-\thetaen}^2  \leq \left(1-\frac{\gamma}{\kappamod^2}\right)\textbf{1}_{\{\tau_\epsilon>j-1\}}\normfro{\theta_{j-1}-\thetaen}^2 \\
- 2\gamma V_j +2\gamma^2 \textbf{1}_{\{\tau_\epsilon>j-1\}}\normfro{\Sigma_j-\theta_{j-1}^{-1}}^2.
\end{multline} 
Summing (\ref{thm4:eq2}) for $j=1$ to $k$ yields
\begin{eqnarray}\label{thm4:eq22}
\sup_{k\geq 0} \textbf{1}_{\{\tau_\epsilon>k\}}\normfro{\theta_{k}-\thetaen}^2 & \leq &  \normfro{\theta_0-\thetaen}^2 + 2\gamma\sup_{k\geq 1}\left|\sum_{j=1}^k V_j\right| \nonumber\\
&&+2\gamma^2 \sum_{j=1}^\infty \textbf{1}_{\{\tau_\epsilon>j-1\}}\normfro{\Sigma_j-\theta_{j-1}^{-1}}^2,\nonumber\\
&  = & \normfro{\theta_0-\thetaen}^2 + \zeta,\end{eqnarray}
where $\zeta \eqdef 2\gamma\sup_{k\geq 1}\left|\sum_{j=1}^k V_j\right| +2\gamma^2 \sum_{j=1}^\infty \textbf{1}_{\{\tau_\epsilon>j-1\}}\normfro{\Sigma_j-\theta_{j-1}^{-1}}^2$. The bound (\ref{thm4:eq22}) in turn means that on the event $\{\tau_\epsilon=\infty\}$, for all $j\geq 0$,
\[\|\theta_j\|_2 \leq \|\thetaen\|_2 + \normfro{\theta_j -\thetaen} \leq \psi_{UB} + \sqrt{p(\psi_{UB}-\ell_\star(\epsilon))^2 +\zeta}.\]
Hence, with $\Psi_\star(\epsilon) \eqdef \min\left(\psi^1_\star(\epsilon), \psi_{UB} + \sqrt{p(\psi_{UB}-\ell_\star(\epsilon))^2 +\zeta}\right)$, we have shown that $\{\tau_\epsilon=\infty\}\subset\{\tau(\ell_\star(\epsilon),\Psi_\star(\epsilon))=\infty\}$, and the first part of the lemma follows from the bound (\ref{lem:sto:eq1}). 
\medskip

\paragraph{\tt Bound on $\PE(\Psi_\star(\epsilon)^2)$}\;\; Clearly it suffices to bound $\PE(\zeta)$. 
Recall that $\Sigma_j = \frac{1}{N_j}\sum_{k=1}^{N_j}z_kz_k'$, where $z_{1:N_j}\stackrel{i.i.d.}{\sim}\textbf{N}(0,\theta_{j-1}^{-1})$. We easily calculate (See Lemma \ref{lem:exp:bound} for details) that on the event $\{\tau_\epsilon>j-1\}$,
\[\PE\left(\normfro{\Sigma_j -\theta_{j-1}^{-1}}^2\vert \F_{j-1}\right) = \frac{1}{N_j}\left(\textsf{Tr}(\theta_{j-1}^{-1})^2 + \normfro{\theta_{j-1}^{-1}}^2\right),\]
and for $\theta_{j}\in\M_+(\epsilonmod_\star(\epsilon),\kappamod^1_\star(\epsilon))$, $\textsf{Tr}(\theta_{j}^{-1})^2 + \normfro{\theta_{j}^{-1}}^2\leq \epsilonmod_\star(\epsilon)^{-2}(p+p^2)$. 
Hence 
\begin{eqnarray*}
\PE\left[\sum_{j=1}^\infty \textbf{1}_{\{\tau>j-1\}}\normfro{\Sigma_j-\theta_{j-1}^{-1}}^2\right] &=& \sum_{j=1}^\infty \PE\left[\textbf{1}_{\{\tau>j-1\}}\PE\left(\normfro{\Sigma_j-\theta_{j-1}^{-1}}^2\vert\F_{j-2}\right)\right]\\
&\leq& \epsilonmod_\star(\epsilon)^{-2}(p+p^2) \sum_{j=1}^\infty \frac{1}{N_j}<\infty,\end{eqnarray*}
by assumption.  By Doob's inequality (\cite{hall:heyde:80}~ Theorem 2.2) applied to the martingale $\{\sum_{j=1}^k V_k\}$, 
\begin{eqnarray*}
\PE\left[\sup_{k\geq 1}\left|\sum_{j=1}^k V_j\right|\right]&=&\lim_{N\to\infty} \PE\left[\sup_{1\leq k\leq N}\left|\sum_{j=1}^k V_j\right|\right]\leq 2\lim_{N\to\infty} \PE^{1/2} \left[\left|\sum_{j=1}^N V_j\right|^2\right] \\
 &=&2\left\{ \sum_{j=1}^\infty\PE(V_j^2)\right\}^{1/2}. \end{eqnarray*}
Using again the facts that the proximal operator is non-expansive and $\thetaen=T_\gamma(\thetaen;\alpha)$, we have $|V_j|\leq \textbf{1}_{\{\tau_\epsilon>j-1\}} \normfro{\theta_{j-1}-\thetaen} \normfro{\Sigma_j-\theta_{j-1}^{-1}}$. Therefore, with similar calculations as above, we have
\[\PE(V_j^2)  =\PE\left[\PE(V_j^2\vert \F_{j-1})\right]\leq \ell_\star(\epsilon)^{-2}(p+p^2)N_j^{-1}\PE\left(\textbf{1}_{\{\tau_\epsilon>j-1\}}\normfro{\theta_{j-1}-\thetaen}^2\right).\]
On $\{\tau_\epsilon>j-1\}$, $\normfro{\theta_{j-1}-\thetaen}\leq \sqrt{p}\|\theta_{j-1}-\thetaen\|_2\leq \sqrt{p}(\psi_{UB}-\ell_\star(\epsilon))$. Hence
\[\PE(V_j^2) \leq \frac{p(p+p^2)(\psi_{UB}-\ell_\star(\epsilon))^2}{\ell_\star(\epsilon)^2}\frac{1}{N_j},\]
which together with the assumption $\sum_j N_j^{-1}<\infty$, and the above calculation show that
$\PE\left[\sup_{k\geq 1}\left|\sum_{j=1}^k V_j\right|\right]<\infty$.

\paragraph{\tt Convergence of $\theta_n$}\;\; We sum (\ref{thm4:eq2}) from $j=1$ to $k$, which gives, for all $k\geq 1$:
\begin{multline*}
\textbf{1}_{\{\tau_\epsilon>k\}}\normfro{\theta_{k}-\thetaen}^2 + \frac{\gamma}{\kappamod^2}\sum_{j=1}^{k}\textbf{1}_{\{\tau_\epsilon>j-1\}}\normfro{\theta_{j-1}-\thetaen}^2   \\
\leq  2\gamma \sup_{k\geq 1}\left|\sum_{j=1}^k V_j\right| +2\gamma^2 \sum_{j=1}^\infty \textbf{1}_{\{\tau_\epsilon>j-1\}}\normfro{\Sigma_j-\theta_{j-1}^{-1}}^2.
\end{multline*} 
We have seen above that the term on the right-hand side of this inequality  has a finite expectation. This implies the series $\sum_{j=1}^{\infty}\textbf{1}_{\{\tau_\epsilon>j-1\}}\normfro{\theta_{j-1}-\thetaen}^2$ is finite almost surely, which in turn implies that on $\{\tau_\epsilon=\infty\}$, we necessarily have $\lim_k\theta_k=\thetaen$, as claimed.

\end{proof}

\subsection{Proof of Theorem \ref{thm4}}\label{proof:thm4}
\begin{proof}
Taking the expectation on both sides on (\ref{thm4:eq2}) yields
\begin{multline*}
\PE\left[\textbf{1}_{\{\tau>j\}}\normfro{\theta_{j}-\thetaen}^2\right]\leq \left(1-\frac{\gamma}{\kappamod^2}\right)\PE\left[\textbf{1}_{\{\tau>j-1\}}\normfro{\theta_{j-1}-\thetaen}^2\right] \\
+ 2\gamma^2\PE\left[\textbf{1}_{\{\tau>j-1\}}\PE\left(\normfro{\Sigma_j-\theta_{j-1}^{-1}}^2\vert\F_{j-1}\right)\right].\end{multline*}
Iterating this inequality yields
\begin{multline*}
\PE\left[\textbf{1}_{\{\tau>k\}}\normfro{\theta_{k}-\thetaen}^2\right]\leq \left(1-\frac{\gamma}{\kappamod^2}\right)^k \normfro{\theta_0-\thetaen}^2 \\
+ 2\gamma^2\sum_{j=1}^k\left(1-\frac{\gamma}{\kappamod^2}\right)^{k-j}\PE\left[\textbf{1}_{\{\tau>j-1\}}\PE\left(\normfro{\Sigma_j-\theta_{j-1}^{-1}}^2\right)\right]. \end{multline*}
Recall that $\Sigma_j = \frac{1}{N_j}\sum_{k=1}^{N_j}z_kz_k'$, where $z_{1:N_j}\stackrel{i.i.d.}{\sim}\textbf{N}(0,\theta_{j-1}^{-1})$. We easily calculate (See Lemma \ref{lem:exp:bound} for details) that on the event $\{\tau>j-1\}$,
\[\PE\left(\normfro{\Sigma_j -\theta_{j-1}^{-1}}^2\vert \F_{j-1}\right) = \frac{1}{N_j}\left(\textsf{Tr}(\theta_{j-1}^{-1})^2 + \normfro{\theta_{j-1}^{-1}}^2\right),\]
and for $\theta_{j}\in\M_+(\epsilonmod,\kappamod)$, $\textsf{Tr}(\theta_{j}^{-1})^2 + \normfro{\theta_{j}^{-1}}^2\leq \epsilonmod^{-2}(p+p^2)$. The stated bound on the term $\PE\left[\textbf{1}_{\{\tau>k\}}\normfro{\theta_{k}-\thetaen}^2\right]$ then follows.
\end{proof}

\subsection{Proof of Theorem \ref{thm5}}\label{sec:proofthm5}
\begin{proof}
We write $\tau = \tau(\epsilonmod,\kappamod)$. On $\{\tau>k\}$, $\theta_0,\ldots,\theta_k\in\M_+(\epsilonmod,\kappamod)$, and $\theta_{i+1} = \Prox_\gamma(\theta_i-\gamma(S-\Sigma_{i+1};\alpha)$ for $i\geq 0$. We apply Lemma \ref{lem4} with $H=\Sigma_{i+1}$ to write
\[\normfro{\theta_{i+1}-\thetaen}^2\leq \left(1-\frac{\gamma}{\kappamod^2}\right) \normfro{\theta_{i}-\thetaen}^2 + 2\gamma\pscal{\theta_{i+1}-\thetaen}{\Sigma_{i+1}-\theta_i^{-1}}.\]
By iterating this bound, we obtain
\begin{eqnarray}\label{eq1:thm5}
\normfro{\theta_{k}-\thetaen}^2&\leq& \left(1-\frac{\gamma}{\kappamod^2}\right)^k \normfro{\theta_{0}-\thetaen}^2 \nonumber\\
&&+ 2\gamma\sup_{k\geq 0}\normfro{\theta_k-\thetaen}^2\sum_{j=1}^k \left(1-\frac{\gamma}{\kappamod^2}\right)^{k-j} \normfro{\Sigma_{j+1}-\theta_j^{-1}}.\end{eqnarray}
On $\{\tau(\epsilonmod,\kappamod)=\infty\}$, $\sup_{i\geq 0}\normfro{\theta_i-\thetaen}^2$ is finite and if $\lim_j \normfro{\Sigma_{j+1}-\theta_j^{-1}} =0$, the bound (\ref{eq1:thm5}) would easily imply that 
$\lim_k \normfro{\theta_{k}-\thetaen}^2 = 0$. Hence the theorem is proved by showing that on $\{\tau=\infty\}$,  $\lim_k \normfro{\Sigma_{k+1}-\theta_k^{-1}} =0$.
From (\ref{algo2:rec1}), we write
\[\Sigma_{k+1}-\theta_k^{-1} = (1-\zeta_{k+1})\left(\Sigma_k-\theta_{k-1}^{-1}\right) + (1-\zeta_{k+1})(\theta_{k-1}^{-1}-\theta_k^{-1}) +\zeta_{k+1}\eta_{k+1},\]
where
\[\eta_{k+1}\eqdef \frac{1}{N}\sum_{k=1}^Nz_kz_k'-\theta_k^{-1},\;\;\;z_{1:N}\stackrel{i.i.d.}{\sim}\textbf{N}(0,\theta_k^{-1}).\]
We expand this into 
\begin{multline*}
\textsf{1}_{\{\tau>k\}}\left(\Sigma_{k+1}-\theta_k^{-1}\right) =(1-\zeta_{k+1})\textsf{1}_{\{\tau>k-1\}}\left(\Sigma_{k}-\theta_{k-1}^{-1}\right) + R_{k+1}^{(1)} + R_{k+1}^{(2)} + R_{k+1}^{(3)} + R_{k+1}^{(4)},\end{multline*}
where the remainders are given by
\[R_{k+1}^{(1)}\eqdef -\textsf{1}_{\{\tau=k\}}(1-\zeta_{k+1})\Sigma_k,\]
\[R_{k+1}^{(2)} \eqdef (1-\zeta_k)\textsf{1}_{\{\tau>k-1\}}\theta_{k-1}^{-1} -(1-\zeta_{k+1})\textsf{1}_{\{\tau>k\}}\theta_{k}^{-1} ,\]
\[R_{k+1}^{(3)} \eqdef (\zeta_k-\zeta_{k+1})\textsf{1}_{\{\tau>k-1\}}\theta_{k-1}^{-1},\]
and
\[R_{k+1}^{(4)}\eqdef \zeta_{k+1} \textsf{1}_{\{\tau>k\}}\eta_{k+1}.\]
Since $\textsf{1}_{\{\tau>k,\tau=\infty\}} = \textsf{1}_{\{\tau=\infty\}}$, and $\textsf{1}_{\{\tau=k,\tau=\infty\}}=0$, it follows that for all $n\geq 0$,
\begin{multline*}
\textsf{1}_{\{\tau=\infty\}}\left(\Sigma_{k+1}-\theta_k^{-1}\right) 
= \textsf{1}_{\{\tau=\infty\}}\prod_{k=1}^k(1-\zeta_{k+1})(\Sigma_1-\theta_0^{-1}) \\
+ \textsf{1}_{\{\tau=\infty\}}\sum_{j=1}^k\left(R_j^{(2)}+R_j^{(3)}+R_j^{(4)}\right)\prod_{i=j+1}^k(1-\zeta_{i+1}).
\end{multline*}
Clearly, we have $\prod_{k=1}^k(1-\zeta_{k+1})\leq \exp\left(-\sum_{k=1}^k\zeta_{k+1}\right)\to 0$ as $k\to\infty$ by (\ref{cond:zeta}), and if the series $\sum_{j\geq 1}\left(R_j^{(2)}+R_j^{(3)}+R_j^{(4)}\right)$ is finite on $\{\tau=\infty\}$, then by Kronecker lemma, it would follow that $\sum_{j=1}^k\left(R_j^{(2)}+R_j^{(3)}+R_j^{(4)}\right)\prod_{i=j+1}^k(1-\zeta_{i+1})\to 0$, as $k\to\infty$ on $\{\tau=\infty\}$. Hence, it suffices to prove that the series $\sum_{j\geq 1}\left(R_j^{(2)}+R_j^{(3)}+R_j^{(4)}\right)$ is finite on $\{\tau=\infty\}$. 

We have $\sum_{k=1}^k R_k^{(2)}= (1-\zeta_1)\textsf{1}_{\{\tau>0\}}\theta_{0}^{-1} - (1-\zeta_{k+1})\textsf{1}_{\{\tau>k\}}\theta_{k}^{-1}$. The assumption that $\theta_k$ has a limit and $\theta_k\in\M_+(\epsilonmod,\kappamod)$ easily implies that $\sum_{k} R_k^{(2)}$ is finite. Similarly, we have $\sum_k \normfro{R_k^{(3)}}\leq \epsilonmod^{-1}\zeta_0<\infty$, and 
\[\PE\left(\normfro{\sum_k R_k^{(4)}}^2\right)=\sum_k \zeta_k^2 \PE\left(\textsf{1}_{\{\tau>k\}}\normfro{\frac{1}{N}\sum_{k=1}^N z_kz_k'-\theta_k^{-1}}^2\right)\leq \epsilonmod^{-2}(p+p^2)\sum_k \zeta_k^2<\infty.\]
\end{proof}

\subsection{Proof of Theorem~\ref{thm:conn-comp-1}}\label{proof-conn-comp-1}
\begin{proof}
The proof follows~\citep{mazumder2012exact} with appropriate modifications, and we provide a brief sketch here. 

\noindent {\emph{First Part:}} \\
We start with the connected component decomposition of the non-zeros of $\hat{\theta}$.
Let us assume that the rows/columns of the matrix $\hat{\theta}$ have been arranged such that 
it is block diagonal. We proceed by writing the KKT conditions of Problem~\eqref{eq-glasso-enet-1}:
\begin{equation}\label{kkt-enet-1}
-\hat{\theta}^{-1} + S  + \tau_{1} \sign(\widehat\theta) + 2\tau_{2} \widehat\theta  = 0, 
\end{equation}
where, $\tau_{1} = \alpha \lambda_{1}$ and $\tau_{2} = \frac{1- \alpha}{2}\lambda_{2}$ and $\sign(\widehat\theta)$ is a matrix where $\sign(\cdot)$ is applied 
componentwise to every entry of $\widehat\theta$.  Since $\hat{\theta}$ is block diagonal so is $\hat{\theta}^{-1}$. If we take the
$(i,j)$th entry of the matrix appearing in~\eqref{kkt-enet-1} such that $i$ and $j$ belong to two different connected components then:
$-(\hat{\theta}^{-1})_{ij}  + 2 \tau \hat{\theta}_{ij} = 0$ which implies that 
$ s_{ij}  + \tau_{1} \sign(\widehat\theta_{ij}) = 0$. Thus we have:
$ |s_{ij}| \leq \tau_{1}$ for all pairs $i,j$ such that they belong to two different connected components.
Thus the binary matrix $((\mathbf{1} ( | s_{ij}| > \tau_{1})))$ will have zeros for all $i,j$ belonging to two different components $ {\mathcal{\widehat{V}}}_{r} $
and  ${\mathcal{\widehat{V}}}_{s}$ for $ r \neq s$.
The connected components of $((\mathbf{1} ( | s_{ij}| > \tau_{1})))$ have a finer resolution than ${\mathcal{ \widehat{V}} }_{j}, j = 1, \ldots,\widehat{J}$ and in particular 
$\widehat{J}  \leq J$.

\noindent {\emph{Second Part:}} \\
For the other part, let us start by assuming that  the symmetric binary matrix $((\mathbf{1} ( | s_{ij}| > \tau_{1})))$ breaks down into $J$ many connected components; and let 
$\widetilde{\theta} = \diag( \hat{\theta}_{1}, \ldots, \hat{\theta}_{J})$ be a block diagonal matrix, where, each 
$ \hat{\theta}_{r}$ is obtained by solving Problem~\eqref{eq-glasso-enet-1} restricted 
to the $r$th connected component ${\mathcal{ V }}_{r}$ where, $r = 1, \ldots, J.$ For any $i,j$ belonging to two different components ${\mathcal{ V }}_{r}$ and ${\mathcal{ V }}_{s}$
with $ r \neq s$ we have that $ |s_{ij} |   \leq  \tau_{1} $ and in addition, $\widetilde{\theta}_{ij}=0$ and $(\widetilde{\theta}^{-1})_{ij} = 0$. This implies that 
$ \widetilde{\theta}$ satisfies the KKT conditions~\eqref{kkt-enet-1} and is hence a solution to Problem~\eqref{eq-glasso-enet-1}. This in particular, 
implies that $\widehat{J} \geq J$ and the connected components of ${\mathcal{ \widehat{V}} }_{j}, j = 1, \ldots, \widehat{J}$ are a finer resolution than 
${\mathcal{ V }}_{r}, r = 1, \ldots, J.$

Combining the above two parts, we conclude that the connected components 
of the two binary matrices $((\mathbf{1} ( | s_{ij}| > \tau_{1})))$  and $((\mathbf{1} ( | \hat{\theta}_{ij}  \neq  0)))$ are indeed equal. 

\end{proof}

\subsection{Proof of Lemma~\ref{closed-soon-ridge}}\label{proof-ridge-soln-1}
\begin{proof}
To see this we take the derivative of the objective function wrt $\theta$ and set it to zero:
\begin{equation}\label{line-ridge-1}
- \theta^{-1} + S +    \lambda \theta = 0.
\end{equation}
Suppose that the sample covariance matrix $S$ can be written as:
$$ S = U D U',$$
where the above denotes the full eigen-value decomposition of $S$ which is a ${p \times p}$ matrix. 
Let $d_{i}$ denote the diagonals of $D$. 
We will show that the solution to Problem~\eqref{eq-ridge} is of the form $\hat{\theta} = U \diag(\sigma) U'$, where, $\diag(\sigma)$ is a diagonal matrix with the $i$th diagonal entry being $\sigma_{i}$.

Let us multiply both sides of~\eqref{line-ridge-1} by $U'$ and $U$ on the left and right respectively. It is then easy to see that 
the optimal values of $\sigma$ can be computed as follows:
$$ -1/\sigma_{i} + d_{i} + \lambda_{} \sigma_{i} = 0$$ for all $i = 1, \ldots, p$.
The above can be solved for every $i$ separately leading to:
$$\hat{\sigma}_{i} = \frac{ - d_{i} + \sqrt{d_{i}^2 + 4 \lambda_{}}}{2 \lambda_{}}, \;\; \forall i$$
Thus we have the statement of Lemma~\ref{closed-soon-ridge}.
\end{proof}

\medskip
\medskip

\noindent {\large{\textbf{Acknowledgements}}}

\medskip

		Yves F. Atchad\'e is partly supported by NSF grant DMS 1228164.  Rahul Mazumder was supported by ONR grant ONR - N00014-15-1-2342,
		Columbia University's start-up fund and an interface grant from the Betty-Moore Sloan Foundation. R.M. will like to thank Robert Freund for helpful 
		comments and encouragement.

%\begin{figure}
%\includegraphics[width=0.5\textwidth]{Fig_100}
%\caption{Simulation output, $p=100, n=500$}
%\label{fig:p50}
%\end{figure}

%%\newpage

\bibliographystyle{ims}

\bibliography{biblio,biblio_rahul}

\newpage

\appendix

\begin{center}
\large{\textbf{Appendix}}
\end{center}

\section{Related Work and Algorithms}\label{related-work}
%\section{APPENDIX A: Related Work and Algorithms}\label{related-work}
%\textcolor{red}{RM: Yves, maybe we can put this section in the appendix?}

In this section we review some of the state-of-the art methods and approaches for the \Glasso problem (Problem~\eqref{eq-glasso-1}). Problem~\eqref{eq-glasso-1} is a nonlinear convex semidefinite optimization problem~\citep{vandenberghe1996semidefinite} and off-the-shelf interior point solvers typically have a per-iteration complexity of $O(p^6)$ that stems from solving a typically dense system with $O(p^2)$ variables~\citep{vandenberghe1998determinant}. 
This makes generic interior point solvers inapplicable for solving problems with $p$ of the order of a few hundred.

A popular approach to optimize problem~\eqref{eq-glasso-1} is to focus on its dual optimization problem, given by:
\begin{equation}\label{eq-glasso-1-dual}
\maxi_{ w \in \M_+ }\;\;\;\; \log\det(w) \;\;\; \sbt \;\;  \| S - w \|_{\infty} \leq \lambda,
\end{equation}
with primal dual relationship given by $ w = \theta^{-1}$. 
The dual problem has a smooth function appearing in its objective. Many efficient solvers for Problem~\eqref{eq-glasso-1}
optimize the dual Problem~\eqref{eq-glasso-1-dual} --- see 
for example~\cite{BGA2008,FHT2007,Lu:09,mazumder2012graphical} and references therein.

In one of the earlier works,~\cite{BGA2008} consider solving the dual Problem~\eqref{eq-glasso-1-dual}.
They propose a smooth accelerated gradient based method \citep{nest_05} with complexity $O(\frac{p^{4.5}}{\deltamod})$ to obtain a $\deltamod$-accurate solution --- the per iteration cost being $O(p^3)$. They also proposed a block coordinate method which requires solving at every iteration, a box constrained quadratic program (QP) which they solve using Interior point methods---leading to an overall complexity of $O(p^4)$.

The \emph{graphical lasso} algorithm \citep{FHT2007a} is widely regarded as one of the most efficient and practical algorithms for Problem~\eqref{eq-glasso-1}.
The algorithm uses a row-by-row block coordinate method that requires to solve a $\ell_{1}$ regularized quadratic program for every row/column---the authors use 
one-at-a-time cyclical coordinate descent to solve the QPs to high accuracy. While it is difficult to provide a precise complexity result for this method, 
the cost is roughly $O(p^3)$ for  (reasonably) sparse-problems with $p$ nodes. 
For dense problems the cost can be as large as $O(p^4)$, or even more. \cite{mazumder2012graphical} further investigate the properties of the \emph{graphical lasso} algorithm, its operational characteristics and propose another block coordinate method for Problem~\eqref{eq-glasso-1} that often enjoys better numerical behavior than \emph{graphical lasso}.

The algorithm proposed in \cite{Lu:09} employs accelerated gradient based algorithms  \citep{nest_05,nest-07-new}. 
The algorithm \textsc{smacs} proposed in the paper has a 
per iteration complexity of $O(p^3)$ and an overall complexity of $O(\frac{p^4}{\sqrt{\deltamod}})$ to reach a $\deltamod$-accurate optimal solution.

\cite{Li:10} propose a specialized interior point algorithm for problem~\eqref{eq-glasso-1}. By rewriting the objective as a smooth convex optimization problem by doubling the number of variables, the paper proposes a scheme to scale interior point like methods up to $p = 2000$.

\cite{scheinberg2010sparse} propose alternating direction based methods for the problem, the main complexity per iteration being $O(p^3)$ associated 
with a full spectral decomposition of a $p \times p$ symmetric matrix and a matrix inversion.
\cite{yuan2012alternating} propose an alternating direction method for problem~\eqref{eq-glasso-1}, with per iteration complexity of $O(p^3)$.
Computational scalability of a similar type can also be achieved by using the alternating direction method of multipliers ADMM~\citet{boyd-admm1} which perform
spectral decompositions and/or matrix inversions with  per iteration  complexity $O(p^3)$.

Fairly recently, \cite{JMLR:v15:hsieh14a} propose a Newton-like method for Problem~\eqref{eq-glasso-1}, the algorithm is known as {\textsc{Quic}}.
The main idea is to reduce the problem to iteratively solving large scale $\ell_{1}$ regularized quadratic programs, which are solved using 
one-at-a-time coordinate descent update rules. The authors develop asymptotic convergence guarantees of the algorithm.
It appears that several computational tricks and fairly advanced implementations in~{\texttt{C++}}
are used to make the approach scalable to large problems. At the time of writing this paper, {\textsc{QUIC}} seems to be one of the most advanced algorithms for \Glassoperiod~
\cite{oztoprak2012newton} propose a related approach based on a Newton-like quadratic approximation of the log-determinant function.

\section{Additional Computational Details}\label{comp-details-11}

We initialize all the solvers using the diagonal matrix obtained by taking the inverse sample variances. For all the simulated-data experiments, the step-size and the Monte Carlo batch-size are taken as follows. The step-size is set to $\gamma=10$, the Monte Carlo batch-size is set to $N_k= \lceil 30 +k^{1.8} \rceil$ at iteration $k$. Additionally, for Algorithm~\ref{algo2} we use $N=400$, and $\zeta_k = k^{-0.7}$.

For $p=1000$,  the values of the regularization parameters were taken as
$(\alpha, \lambda) = (0.89, 0.01)$.
We computed $\hat{\theta}$ (the target solution to the optimization problem) by running the deterministic algorithm for 1000 iterations.

The size of the largest component is  967, one component had size two with all  other components having size one. In this case, the splitting offered marginal improvements since the size of the maximal component 
was quite close to $p$.

For $p=5000$  the values of the regularization parameters were taken as
$(\alpha, \lambda) = (0.93, 0.0085)$ and
we computed $\hat{\theta}$ (the target solution to the optimization problem) by running the deterministic algorithm for 1000 iterations.

For the case, $p=5,000$ splitting leads to $76$ connected components, 
The size of the largest component is 4924 with all  other components having size one.

For $p=10,000$ ,  the values of the regularization parameters were taken as
$(\alpha, \lambda) = (0.96,   0.01)$.
 We computed $\hat{\theta}$ (the target solution to the optimization problem) by running the deterministic algorithm for 500 iterations.

For the case,  $p=10,000$ splitting leads to $1330$ connected components, 
The size of the largest component is 8670, one component has size two with all  other components having size one.

We present the results for the cases $p=5,000$ and $p=10,000$ in Table~\ref{table-comp-quic-etc}.

For the real-data example, the stochastic algorithms are set up as follows. The step-size is set to $\gamma=5\times 10^{-5}$, the Monte Carlo batch-size is set to $N_k= \lceil 100 +k^{1.8} \rceil$ at iteration $k$. Additionally, for Algorithm~\ref{algo2} we use $N=200$, and $\zeta_k = k^{-0.52}$.

\section{Some Technical Lemmas and Proofs}

\begin{lemma}\label{lem2}
Consider the function $f(\theta)=-\log\det\theta +\textsf{Tr}(\theta S)$, $\theta\in\M_+$. Take $0<\epsilonmod<\kappamod\leq \infty$. If $\theta\in\M_+(\epsilonmod,\kappamod)$, and $H\in\M$ are such that $\theta+ H\in\M_+(\epsilonmod,\kappamod)$, then 
\[ f(\theta)  + \pscal{S-\theta^{-1}}{H}  +\frac{1}{2\kappamod^2} \|H\|^2\leq f\left(\theta+ H\right) \leq  f(\theta)  + \pscal{S-\theta^{-1}}{H}  +\frac{1}{2\epsilonmod^2} \|H\|^2.\]
\end{lemma}
\begin{proof}
First notice that $\M_+(\epsilonmod,\kappamod)$ is a convex set. Hence for all $t\in [0,1]$, $\theta+ tH=(1-t)\theta +t(\theta+ H)\in \M_+(\epsilonmod,\kappamod)$. 
Then by Taylor expansion we have,
\[\log\det(\theta+ H) = \log\det\theta +\pscal{\theta^{-1}}{H} +\int_0^1\pscal{(\theta+tH)^{-1}-\theta^{-1}}{H}\rmd t.\]
This gives
\[f(\theta+H) -f(\theta) -\pscal{S-\theta^{-1}}{H} =- \int_0^1\pscal{(\theta+tH)^{-1}-\theta^{-1}}{H}\rmd t.\]
However $(\theta+tH)^{-1}-\theta^{-1}=-t\theta^{-1}H(\theta+tH)^{-1}$. Therefore, if $\theta=\sum_{j=1}^p \lambda_j u_ju_j'$ denotes the eigen-decomposition of $\theta$, we have
\begin{eqnarray*}
-\pscal{(\theta+tH)^{-1}-\theta^{-1}}{H}&=&t\textsf{Tr}\left(\theta^{-1}H(\theta+tH)^{-1}H\right)\\
&=&t\sum_{j=1}^p \lambda_j^{-1} u_j'H(\theta+tH)^{-1}H u_j\\
&\leq &\frac{t}{\epsilonmod^2}\sum_{j=1}^p \|Hu_j\|^2 = \frac{t}{\epsilonmod^2}\|H\|^2.
\end{eqnarray*}
Similarly calculations gives
\[-\pscal{(\theta+tH)^{-1}-\theta^{-1}}{H} \geq \frac{t}{\kappamod^2}\|H\|^2.\]
The lemma is proved.
\end{proof}

We also use the following well known property of the proximal map.
\begin{lemma}\label{lem3}
For all $\theta,\vartheta\in\M$, and for all $\alpha\in [0,1]$, $\gamma>0$,
\[g_\alpha(\Prox_\gamma(\theta;\alpha)) \leq g_\alpha(\vartheta) +\frac{1}{\gamma}\pscal{\vartheta-\Prox_\gamma(\theta;\alpha)}{\Prox_\gamma(\theta;\alpha)-\theta}.\]
\end{lemma}
\begin{proof}
See \cite[Propositions~12.26 and ~12.27]{bauschke:combettes:2011}.
\end{proof}

Lemma \ref{lem2} amd Lemma \ref{lem3} together give the following key result.

\begin{lemma}\label{lem4}
Fix $0<\epsilonmod<\kappamod\leq \infty$, and $\gamma\in(0,\epsilonmod^2]$. Suppose that $\thetaen, \theta\in\M_+(\epsilonmod,\kappamod)$, and $H\in\M$ are such that $\bar\theta\eqdef\Prox_\gamma(\theta -\gamma(S-H);\alpha)\in\M_+(\epsilonmod,\kappamod)$. Then 
\begin{multline*}\normfro{\bar\theta-\thetaen}^2\leq 2\gamma \left(\phi_\alpha(\bar\theta)-\phi_\alpha(\thetaen)\right)  + \normfro{\bar\theta-\thetaen}^2\\
\leq \left(1-\frac{\gamma}{\kappamod^2}\right)\normfro{\theta-\thetaen}^2 +2\gamma \pscal{\bar\theta-\thetaen}{H-\theta^{-1}},\end{multline*}
where we recall that $\phi_\alpha(\theta) = f(\theta)+g_\alpha(\theta)$. 
\end{lemma}
\begin{proof}
Set $f(\theta)=-\log\det\theta +\textsf{Tr}(\theta S)$, $\theta\in\M_+$. By Lemma \ref{lem2},
\begin{equation}\label{prop1_eq1}
f(\bar \theta) \leq  f(\theta) +\pscal{S-\theta^{-1}}{\bar \theta-\theta} + \frac{1}{2\gamma}\normfro{\bar \theta-\theta}^2.\end{equation}
Subtracting $f(\thetaen)$ from both sides of  the above inequality and re-arranging gives
\begin{eqnarray}\label{prop1_eq2}
f(\bar \theta)-f(\thetaen) &\leq &\left[f(\theta)+\pscal{S-\theta^{-1}}{\thetaen-\theta} -f(\thetaen)\right]\nonumber\\
&&+\pscal{S-\theta_i^{-1}}{\bar \theta-\thetaen} + \frac{1}{2\gamma}\normfro{\bar \theta-\theta}^2.\end{eqnarray}
Since $\theta, \thetaen\in\M_+(\epsilonmod,\kappamod)$, the strong convexity of $\theta\mapsto -\log\det\theta+\textsf{Tr}(\theta S)$ established in Lemma \ref{lem2} implies that $f(\theta)+\pscal{S-\theta^{-1}}{\thetaen-\theta} -f(\thetaen)\leq -\frac{1}{2\kappamod^2}\normfro{\theta-\thetaen}^2$. Using this in (\ref{prop1_eq2}) gives
\begin{equation}\label{prop1_eq3}  f(\bar\theta)-f(\thetaen) \leq -\frac{1}{2\kappamod^2}\normfro{\theta-\thetaen}^2 +\pscal{S-\theta^{-1}}{\bar\theta-\thetaen} + \frac{1}{2\gamma}\normfro{\bar\theta-\theta}^2.\end{equation}
By Lemma \ref{lem3},
\begin{eqnarray}\label{prop1_eq4}
g_\alpha(\bar\theta) -g_\alpha(\thetaen) & \leq & \frac{1}{\gamma}\pscal{\thetaen-\bar\theta}{\bar\theta-\left(\theta-\gamma(S-H)\right)},\nonumber\\
&=&  \frac{1}{\gamma}\pscal{\thetaen-\bar\theta}{\bar\theta-\theta} + \pscal{\thetaen-\bar\theta}{S-H}.\;\;\;\;\;\;
\end{eqnarray}
We combine (\ref{prop1_eq3}) and (\ref{prop1_eq4}) and  re-arrange to deduce that
\begin{multline}\label{thm1:eq5}
\phi_\alpha(\bar\theta)-\phi_\alpha(\thetaen)\leq -\frac{1}{2\kappamod^2}\normfro{\theta-\thetaen}^2 +\frac{1}{2\gamma}\pscal{\bar\theta-\theta}{2\thetaen-\bar\theta-\theta} +\pscal{\bar\theta-\thetaen}{H-\theta^{-1}}\\
=\frac{1}{2}\left(\frac{1}{\gamma}-\frac{1}{\kappamod^2}\right)\normfro{\theta-\thetaen}^2 -\frac{1}{2\gamma}\normfro{\bar\theta-\thetaen}^2 +\pscal{\bar\theta-\thetaen}{H-\theta^{-1}}.
\end{multline}
Since $\phi_\alpha(\bar \theta)\geq \phi_\alpha(\thetaen)$, we conclude that
\[\normfro{\bar\theta-\thetaen}^2\leq 2\gamma \left(\phi_\alpha(\bar\theta)-\phi_\alpha(\thetaen)\right)  + \normfro{\bar\theta-\thetaen}^2\leq \left(1-\frac{\gamma}{\kappamod^2}\right)\normfro{\theta-\thetaen}^2 +2\gamma\pscal{\bar\theta-\thetaen}{H-\theta^{-1}},\]
as claimed.
\end{proof}

\begin{lemma}\label{lem:exp:bound}
Take $\epsilonmod>0$, and $\theta\in\M_+(\epsilonmod)$. Let $z_{1:N}\stackrel{i.i.d.}{\sim}\textbf{N}(0,\theta^{-1})$, and set $G_N\eqdef N^{-1}\sum_{i=1}^N z_iz_i'$.  Then 
\[\PE\left[\normfro{G_N-\theta^{-1}}^2\right]\leq \frac{p+p^2}{N\epsilonmod^2},\]
and for any $\deltamod>0$ such that $\epsilonmod\deltamod\leq 4$,
\[\PP\left(\|G_N-\theta^{-1}\|_\infty>\deltamod\right)\leq 4p^2\exp\left(-\min(1,\epsilonmod^2\deltamod^2/16) N\right).\]
\end{lemma}
\begin{proof}
\begin{multline*}
\PE\left[\normfro{G_N-\theta^{-1}}^2\right]=\sum_{j,k}\PE\left[\left(\frac{1}{N}\sum_{i=1}^N(z_iz_i')_{j,k}-\theta^{-1}_{j,k}\right)^2\right] = \frac{1}{N}\sum_{j,k}\PE\left[\left(z_1z_1')_{j,k}-\theta^{-1}_{j,k}\right)^2\right]\\
=\frac{1}{N}\sum_{l,k}\left(\theta_{j,j}^{-1}\theta_{k,k}^{-1}+(\theta^{-1}_{j,k})^2\right) = \frac{1}{N}\left(\textsf{Tr}(\theta^{-1})^2 +\normfro{\theta^{-1}}^2\right)\leq \frac{1}{N}\left(\left(\frac{p}{\epsilonmod}\right)^2 + \frac{p}{\epsilonmod^2}\right).\end{multline*}

For the exponential bound, we reduce the problem to an exponential bound for chi-squared distributions, and apply the following corollary of Lemma 1 of \cite{laurent:massart:00}. Let $W_{1:N}\stackrel{i.i.d.}{\sim}\chi^2_1$, the chi-square distribution with one degree of freedom. For any $x\in [0,1]$,
\begin{equation}\label{lem:lm}
\PP\left[\left|\sum_{k=1}^N(W_k-1)\right|>4\sqrt{x} N\right]\leq 2e^{-Nx}.\end{equation}

For $1\leq i,j\leq p$, arbitrary, set $Z^{(k)}_{ij}=z_{k,i}z_{k,j}$, and $\sigma_{ij}=\theta^{-1}_{ij}$. Suppose that $i\neq j$. It is easy to check that
\begin{multline*}
\sum_{k=1}^N \left[Z^{(k)}_{ij}-\sigma_{ij}\right]=\frac{1}{4}\sum_{k=1}^N \left[(z_{k,i}+z_{k,j})^2-\sigma_{ii}-\sigma_{jj}-2\sigma_{ij}\right] \\
-\frac{1}{4}\sum_{k=1}^N \left[(z_{k,i}-z_{k,j})^2-\sigma_{ii}-\sigma_{jj}+2\sigma_{ij}\right].\end{multline*}
Notice that $z_{k,i}+z_{k,j}\sim\textbf{N}(0,\sigma_{ii}+\sigma_{jj}+2\sigma_{ij})$, and $z_{k,i}-z_{k,j}\sim\textbf{N}(0,\sigma_{ii}+\sigma_{jj}-2\sigma_{ij})$. It follows that for all $x\geq 0$,
\begin{multline*}
\PP\left[\left|\sum_{k=1}^N \left[Z^{(k)}_{ij}-\sigma_{ij}\right]\right|> x\right]\leq \PP\left[\left|\sum_{k=1}^N(W_k-1)\right|>\frac{2x}{\sigma_{ii} + \sigma_{jj} +2\sigma_{ij}}\right] \\
+2\PP\left[\left|\sum_{k=1}^N(W_k-1)\right|>\frac{2x}{\sigma_{ii} + \sigma_{jj} -\sigma_{ij}}\right],\\
\leq 2\PP\left[\left|\sum_{k=1}^N(W_k-1)\right|>\epsilonmod x\right].
\end{multline*}
where $W_{1:N}\stackrel{i.i.d.}{\sim}\chi^2_1$, the chi-square distribution with one degree of freedom. The last inequality uses the fact that $\sigma_{ii} + \sigma_{jj} +2\sigma_{ij} =u'\theta^{-1}u\leq \frac{1}{\epsilonmod}\|u\|^2\leq \frac{2}{\epsilonmod}$, where $u$ is the vector with $1$ on components $i$ and $j$ and zero everywhere else (similarly for $\sigma_{ii} + \sigma_{jj} -2\sigma_{ij}$ by putting $-1$ on the $j$-th entry). Then we apply (\ref{lem:lm}) to obtain
\[\PP\left[\left|\sum_{k=1}^N \left[Z^{(k)}_{ij}-\sigma_{ij}\right]\right|> N\deltamod\right]\leq 4e^{-\min(1,\epsilonmod^2\deltamod^2/16) N}.\]
When $i=j$, the bound $\PP\left[\left|\sum_{k=1}^N \left[Z^{(k)}_{ij}-\sigma_{ij}\right]\right|> x\right]\leq \PP\left[\left|\sum_{k=1}^N(W_k-1)\right|>\epsilonmod x\right]$ is straightforward. The lemma follows from a standard union-sum argument.
\end{proof}

\end{document}